\documentclass[notitlepage]{report}

\usepackage[utf8]{inputenc}
\usepackage{amsmath, amsthm, amssymb, amsfonts, dsfont, mathrsfs, dirtytalk, tikz-cd, adjustbox, url, nicefrac, booktabs,array, graphicx,todonotes,longtable, tabu,bm,cite, todonotes,comment,multirow,tabularx,epsfig,parskip}
\usepackage{centernot}
\usepackage{mathtools}
\usepackage{stmaryrd}
\usepackage[T1]{fontenc}    
\usepackage[margin=1in]{geometry}
\usepackage[english]{babel}
\usepackage[pagebackref=true]{hyperref}
\usepackage[capitalise]{cleveref}
\usepackage[numbers]{natbib}

\let\OLDthebibliography\thebibliography
\renewcommand\thebibliography[1]{
  \OLDthebibliography{#1}
  \setlength{\parskip}{0pt}
  \setlength{\itemsep}{0pt plus 0.3ex}
}

\usepackage{etoolbox}
\makeatletter
\patchcmd{\chapter}{\if@openright\cleardoublepage\else\clearpage\fi}{}{}{}
\makeatother
\usepackage{titlesec}
\titleformat{\chapter}[hang] 
{\normalfont\huge\bfseries}{\chaptertitlename\ \thechapter:}{1ex}{}

\renewcommand*{\backref}[1]{}
\renewcommand*{\backrefalt}[4]{%
    \ifcase #1%
          \or (p.~#2.)%
          \else (p.~#2.)%
    \fi%
    }
    
\newcommand{\e}{\varepsilon}

\newcommand{\R}{\mathbb R}

\newcommand{\p}{\mathbb P}

\newcommand{\N}{\mathbb N}

\newcommand{\E}{\mathbb E}

\renewcommand{\L}{\mathcal L}
\newcommand{\G}{\mathbb G}

\newcommand{\F}{\operatorname{F}}
\newcommand{\FL}{\operatorname{F_L}}

\newcommand{\T}{\mathbb{T}}
\newcommand{\C}{\mathbb{C}}

\renewcommand{\det}{\operatorname{d e t}}

\newcommand{\tr}{\operatorname{T r}}

\newcommand{\Mat}{\operatorname{Ma t}}                 % Integrator stepsize
\newcommand{\dkl}{\operatorname{D_{K L}}}
\renewcommand{\d}{\: \mathrm{ d }}

\newcommand{\I}{\operatorname{I }_d}

\newcommand{\rmx}{\mathrm{x}}
\newcommand{\rmz}{\mathrm z}

\newcommand{\rmw}{\mathrm w}

\newcommand{\rmy}{\mathrm y}

\newcommand{\bw}{\mathbf w}
\newcommand{\bmu}{\boldsymbol{\mu}}
\newcommand{\bx}{\mathbf x} %generalised coordinate solution
\newcommand{\by}{\mathbf y} %generalised coordinate solution
\newcommand{\bz}{\mathbf z}
\newcommand{\bSigma}{\mathbf \Sigma}

\newcommand{\bvx}{\bar{\mathbf{ x}}}%least action generalised coordinate
\newcommand{\lx}{\bar{ x}} %least action not generalised coordinate

\newcommand{\vrw}{{\mathbf w}} %generalised noise vector
\newcommand{\vrx}{\bx} %generalised coordinate solution

\theoremstyle{plain}
\newtheorem{theorem}{Theorem}[section]
\newtheorem{lemma}[theorem]{Lemma}
\newtheorem{proposition}[theorem]{Proposition}

\theoremstyle{definition}
\newtheorem{definition}[theorem]{Definition}

\newtheorem{example}[theorem]{Example}

\theoremstyle{remark}
\newtheorem{remark}[theorem]{Remark}

\newtheorem{assumption}[theorem]{Assumption}
\newtheorem{approximation}[theorem]{Approximation}

\newtheorem{fact}[theorem]{Fact}

\newenvironment{claim}[1]{\par\noindent\emph{Claim:}\space#1}{}
\newenvironment{proofclaim}[1]{\par\noindent\textit{Proof of claim.}\space#1}{\hfill $\blacksquare$}

\title{A theory of generalised coordinates for\\stochastic differential equations}

\author{Lancelot Da Costa$^{1,2,3,}$\thanks{Equal contribution. Correspondence: \textit{lance.dacosta@verses.ai}.}\:, Nathaël Da Costa$^{4,*}$, Conor Heins$^{1}$, Johan Medrano$^{3}$,\\
Grigorios A. Pavliotis$^{2}$, Thomas Parr$^{5}$, Ajith Anil Meera$^{6}$,
Karl Friston$^{1,3}$
}
\date{\small\it
    $^1$VERSES AI Research Lab, Los Angeles, CA 90016, USA\\
    $^2$Department of Mathematics, Imperial College London, London, SW7 2AZ, UK\\
    $^3$Wellcome Centre for Human Neuroimaging, University College London, London, WC1N 3AR, UK\\
    $^4$Tübingen AI Center, University of Tübingen, Tübingen, 72076, Germany \\
    $^5$Nuffield Department of Clinical Neurosciences, University of Oxford, Oxford, UK\\
    $^{6}$Donders Institute, Radboud University, Nijmegen, 6500HE, The Netherlands
}

\begin{document}

\maketitle

\begin{abstract}
       % 1) Introduction (2 sentences):
%    --> Sentence 1: Basic introduction to the field; accessible to scientists of any  discipline.
Stochastic differential equations are ubiquitous modelling tools in applied mathematics and the sciences.
%   --> Sentence 2: Background of the specific research question; comprehensible to scientists in the same or closely related fields of research.
In most modelling scenarios, random fluctuations driving dynamics or motion have some non-trivial temporal correlation structure, which renders the SDE non-Markovian; a phenomenon commonly known as ``coloured'' noise.
% 2) Problem/objective (1 sentence):
Thus, an important objective is to develop effective tools for mathematically and numerically studying (possibly non-Markovian) SDEs.

% 3) "Here we show" (1 sentence):
In this report, we formalise a mathematical theory for analysing and numerically studying SDEs based on so-called `generalised coordinates of motion'.
Like the theory of rough paths, we analyse SDEs pathwise for any given realisation of the noise, not solely probabilistically.
Like the established theory of Markovian realisation, we realise non-Markovian SDEs as a Markov process in an extended space.
Unlike the established theory of Markovian realisation however, the Markovian realisations here are accurate on short timescales and may be exact globally in time, when flows and fluctuations are analytic.

% 4) Main results and conclusions (~ 3 – 5 sentences)
This theory is exact for SDEs with analytic flows and fluctuations, and is approximate when flows and fluctuations are differentiable.
It provides useful analysis tools, which we employ to solve linear SDEs with analytic fluctuations.
It may also be useful for studying rougher SDEs, as these may be identified as the limit of smoother ones.

This theory supplies effective, computationally straightforward methods for simulation, filtering and control of SDEs; among others, we re-derive generalised Bayesian filtering, a state-of-the-art method for time-series analysis.

% 5) Implications (1 – 2 sentences)
Looking forward, this report suggests that generalised coordinates have far-reaching applications throughout stochastic differential equations.

\end{abstract}

\textit{MSC2020 subject classifications: }Primary 60H10; secondary 60H30, 34F05, 60G15, 65C30, 93E11.\\

%journals: aap (https://imstat.org/journals-and-publications/annals-of-applied-probability/annals-of-applied-probability-manuscript-preparation/), studies in applied maths, IMA Journal of Applied Mathematics, siam journal on applied mathematics, Annales de l'Institut Henri Poincaré: Probabilités et Statistiques, Statistics and Computing, Stochastic Processes and their Applications, Probability Theory and Related Fields

\tableofcontents

\section*{Introduction}

\subsection*{Stochastic differential equations}
Stochastic differential equations are ubiquitous tools in applied mathematics, physics and the sciences.
%Stochastic differential equations are ubiquitous in modelling phenomena of interest to physicists, appliedmathematicians, physicists and beyond.
%Stochastic differential equations are ubiquitous in physics, specifically in statistical, quantum or classical mechanics, and in modelling phenomena of interest to physicists and beyond.
%Loosely speaking, a stochastic differential equation how a system evolves when subjected to a combination of deterministic and fluctuating ("random") flows
Physically speaking, a stochastic differential equation 
\begin{equation}
\label{eq: intro SDE}
    \frac{d}{dt} x_t = f(x_t)+w_t,
\end{equation}
expresses the fact that a system $x_t$ exhibits slowly varying coarse behaviour described by a vector field $f$ superimposed with rapid random fluctuating motions $w_t$ and inherits from an adiabatic approximation (i.e.~a time-scale separation) between fast and slow degrees of freedom \cite{kawasakiSimpleDerivationsGeneralized1973a}. %where the stochastic variable $x_t$ is then separated into a motion that is partly associated with the gross variables $f$ and the rest. 
In modelling, one can motivate a stochastic differential equation as describing how a system evolves when subjected to a combination of deterministic (\textit{known}) and random (\textit{unknown}) fluctuations.
Stochastic differential equations are ubiquitous in physics---specifically in statistical, quantum or classical mechanics---as well as in many branches of applied mathematics, as a model of random dynamical systems.

\subsection*{White noise SDEs}

Despite this intuitive view, the mathematical development of stochastic differential equations has been very challenging and a major endeavour of mathematics and physics since the early 20th century. The first mathematical theories of stochastic differential equations, due to Ito in the 1940s and later Stratonovich, were developed for the rigorous study of Brownian motion and related (Markovian) diffusion equations, which stand for `white' noise fluctuations $w_t$ in \eqref{eq: intro SDE}, and which can be motivated as a model of the motion of particles subject to thermal random fluctuations and other forces \cite{rey-belletOpenClassicalSystems2006,einsteinUberMolekularkinetischenTheorie1905,pavliotisStochasticProcessesApplications2014}.

\subsection*{Colored noise SDEs}

Beyond thermal fluctuations, random fluctuations in physics and in complex systems modelling are usually ``coloured'': this is to say that fluctuations $w_t$ have some temporal correlation structure---think for instance of waves in the ocean and fluctuations in the wind. These coloured fluctuations make the solution to the SDE a non-Markovian process. With the limitations of white noise in mind, three main theories have been developed for analysing SDEs driven by coloured noise:
%To address called for extended theories of stochastic differential equations to study these phenomena:
%which cannot (directly) be accounted by white noise and related (Markovian) diffusion equations. 

\begin{itemize}
    \item The theory of stochastic differential equations driven by semi-martingales \cite{heSemimartingaleTheoryStochastic1992}: a theory of stochastic integration with respect to general (semi-martingale) noise processes, that extends Ito and Stratonovich calculus.
    \item The modern theory of stochastic differential equations driven by rough paths \cite{frizMultidimensionalStochasticProcesses2010,frizCourseRoughPaths2020}, which analyses SDEs driven by general noise signals of possibly very low regularity. Contrasting with the other theories, this is an analytic theory in the sense that SDEs are analysed pathwise---that is for any given realisation of the noise---rather than simply probabilistically.
    \item The theory of Markovian realisation, which approximates non-Markovian processes by Markovian ones \cite{lindquistLinearStochasticSystems2015,mitterTheoryNonlinearStochastic1981,tayorNonlinearStochasticRealization1989}. As a primary use case, this theory allows one to approximate coloured noise driven SDEs by white noise driven SDEs evolving in an extended state-space~\cite[Ch. 8]{pavliotisStochasticProcessesApplications2014}, \cite{rey-belletOpenClassicalSystems2006}. In particular, this approximation can be carried out for systems of non-Markovian interacting particles and their mean-field limit~\cite{SGGPUV2019, DuongPavliotis2018}, as well as for coloured multiplicative noise~\cite{HanggiJung1995}. The resulting Markovian approximation is nevertheless quite challenging to analyse~\cite{GPGSUV2021} as its generator and Fokker-Planck operator are necessarily degenerate (usually hypoelliptic and hypocoercive)~\cite{OttobrePavliotis11}.
\end{itemize}

To set the stage for this work, we draw attention to the scope of the \textit{current} mathematical theory of Markovian realisation. As was noted by Stratonovich: ``The study of Markov processes is particularly appropriate, since effective mathematical methods are available for analysing them. However, a certain care must be taken in replacing an actual process by Markov process, since Markov processes have many special features, and, in particular, differ from the processes encountered in radio engineering by their lack of smoothness'' \cite[p122]{stratonovichTopicsTheoryRandom2014}.
It turns out that Markovian realisations furnish reasonable approximations to fluctuations over time-scales considerably larger than the correlation time; however, they fail at shorter timescales \cite[p122]{stratonovichTopicsTheoryRandom2014}. ``Thus the results obtained by applying the techniques of Markov process theory are valuable only to the extent to which they characterise just these 'large-scale' fluctuations. [...] %any random process actually encountered in radio engineering is analytic, and all its derivatives are finite with probability one. 
However, any random process actually encountered in radio engineering is analytic, and all its derivatives are finite with probability 1. Therefore, we cannot describe an actual process within the framework of Markov process theory, and the more accurately we wish to approximate such a process by a Markov process, the more components the latter must have and the higher the order of the corresponding fluctuation equation must be.'' \cite[p123-125]{stratonovichTopicsTheoryRandom2014}.

\subsection*{A theory of generalised coordinates}

In this paper, we present an alternative theory of Markovian realisation for stochastic differential equations based on generalised coordinates of motion; hereafter, ``generalised coordinates''.
%The terminology of 'generalised coordinates' here refers that it is about finding a set of coordinates selected to simplify calculations, but should not be taken literally as generalised coordinates from multibody dynamics.

\textbf{Scope:} This is a theory for the analysis and numerical study of SDEs, which is exact for SDEs with analytic flows and fluctuations, and which is approximate when flows and fluctuations are differentiable. This theory may be applicable to rougher SDEs, as the Wong-Zakai theorems tell us that many of these can be expressed as limits of smoother SDEs. Like the theory of rough paths, we analyse SDEs pathwise for any given realisation of the noise, not solely probabilistically. Like the current theory of Markovian realisation we realise non-Markovian SDEs as a Markov process in an extended (possibly infinite dimensional) space. Unlike the established theory of Markovian realisation however, the Markovian realisations here are accurate on short timescales and may be exact globally in time (when flows and fluctuations are analytic). Crucially, this theory %can very naturally be applied computationally: it
affords computationally straightforward methodologies for simulation, filtering and control of stochastic differential equations.

\textbf{The basic idea:} We can express the solution of a many times differentiable SDE by its Taylor expansion at any time point. (If the SDE's coefficients are analytic we can consider the Taylor series). We can then reformulate the SDE as a dynamically-evolving Taylor expansion. If we define the coefficients of these Taylor expansions to be our generalised coordinates, the SDE becomes a linear ordinary differential equation in the high-dimensional space of generalised coordinates. We can then analyse the behaviour of this simpler dynamical system which (depending on context) gives us useful approximations or exact analytical solutions to the original equation. The question throughout is: how far can you go with Taylor expansions?

\textbf{History:} The use of generalised coordinates to analyse physical systems is ubiquitous in physics. To the best of our knowledge, the specific approach presented here was originally applied to SDEs as the basis of computationally straightforward and robust filtering of smooth time-series---laying the foundations of the generalised filtering algorithm detailed later---specifically in the context of neuroimaging time-series obtained through functional magnetic resonance imaging \cite{fristonVariationalTreatmentDynamic2008,fristonGeneralisedFiltering2010}. This approach has since percolated into other fields including robotics and control, where it showed state-of-the-art performance for \textit{i)} radar tracking \cite{balajiBayesianStateEstimation2011}, \textit{ii)} filtering during a real quadrotor flight experiment \cite{bos2022free}, \textit{iii)} noise covariance estimation under coloured noise \cite{meera2023adaptive}, and \textit{iv)} system identification under coloured noise \cite{anil2021dynamic}. More recently, this approach was applied for analysing of SDEs in statistical physics, specifically in derivations of the free-energy principle \cite{fristonFreeEnergyPrinciple2023a,fristonPathIntegralsParticular2023}. In this paper we develop and formalise these ideas, and show that they have far-reaching applications throughout stochastic differential equations.

\textbf{Generalised filtering as a generic application:}
A key payoff of the theory of generalised coordinates is a Bayesian filtering scheme that handles non-Markovian SDEs with coloured noise gracefully, while remaining relatively lightweight and scalable. Indeed, classical filters generally appear to lack this joint property:
\begin{enumerate}
    \item Extended/Unscented Kalman filters (EKF/UKF) are derived under the assumption that the process and measurement noises are white \cite{Jazwinski1970}, so their error grows with temporal or serial noise correlations \cite{Farina1986,Julier2004,Brown2012} unless this temporal correlation structure is accounted for explicitly, e.g. through Markovian realisation, in which case however, the filter remains inaccurate on short timescales.
    \item Particle filters can, in principle, handle arbitrary noise, however they suffer from a severe curse of dimensionality: maintaining a non-degenerate ensemble essentially requires an \textit{exponentially} increasing number of particles in the effective state-space dimension \cite{Snyder2008,Bengtsson2008,Bickel2009,Rebeschini2015}.
\end{enumerate}
Generalised filtering sidesteps both issues. Generalised filtering can be thought of as a higher order version of extended Kalman filtering that was developed specifically for coloured noise. It \textit{i)} treats coloured noise natively and accurately, and \textit{ii)} replaces sampling with deterministic gradient flows on a variational free-energy, which reduces computational cost. Head-to-head studies in fMRI time-series analysis \cite{fristonGeneralisedFiltering2010}, radar tracking \cite{balajiBayesianStateEstimation2011} and quadrotor flight estimation \cite{bos2022free} suggest that generalised filtering attains or exceeds EKF-level accuracy under coloured noise while running one–to–two orders of magnitude faster than particle filters tuned for the same tasks. In \cref{sec: GF} we derive this filter rigorously and provide open-source code.

\subsection*{Structure of paper}

\cref{chap: 1} sets the stage by explaining generalised coordinates for stochastic differential equations. We see how to reformulate an SDE to a dynamical system in generalised coordinates, and how to go back; and how much information about the initial SDE is preserved when applying this transform. We briefly review the Wong-Zakai theorems to delineate the class of SDEs where generalised coordinates might usefully be applied.

In \cref{chap: 2}, we delve deeper by examining basic analysis tools and results for stochastic differential equations using generalised coordinates. We start by analysing the statistical structure of random fluctuations in generalised coordinates, and then focus on the case where fluctuations are stationary Gaussian processes with many times differentiable sample paths. Under these conditions, we derive a comprehensive analysis of the solution to linear SDEs. Finally, we formulate the Fokker-Planck equation and the path integral formulation in generalised coordinates.

\cref{chap: 3} develops numerical methods that make use of generalised coordinates. We start by numerical integration (``zigzag'') methods of SDEs that are accurate on short timescales. We then turn to an intriguing identity for numerically recovering the path of least action from the path integral formulation. Finally, we derive some new and existing methods for Bayesian filtering of time-series, variously known as generalised filtering.
All of these numerical methods are supplemented with illustrative simulations and freely available code.

In \cref{chap: 4}, we briefly discuss ways in which this theory and numerical methods could be further developed, and provide concluding remarks. We also discuss our main omission from \cref{chap: 3} which are methods for stochastic control using generalised coordinates. For reference, the \cref{app: notation} has a glossary of some frequently used notations.

\chapter{Fundaments of generalised coordinates}
\label{chap: 1}

\section{Realisation in generalised coordinates}

\subsection{The setup}

Consider a stochastic differential equation on $\R^d$ over some open time interval $\T$ (which, without loss of generality contains $t=0$)

\begin{equation}
\label{eq: SDE}
    \frac{d}{dt} x_t = f(x_t)+w_t, \quad x_0=z \in \R^d.
\end{equation}
%where the solution evolves on $\R^d$ over time $t\in \T$.
The deterministic part is given by the \textit{drift} or \textit{flow}, a vector field $f : \R^d \to \R^d$, and the stochastic part is given by the \textit{random fluctuations} or \textit{noise}, a stochastic process $w: \T \times \Omega \to \R^d$. Here $\Omega$ denotes the sample space of the underlying probability space, made implicit throughout, so that we write $w_t \triangleq w(t, \cdot)$ for the random fluctuations at time $t$ (an $\R^d$-valued random variable). 

\begin{assumption}[Differentiability]
\label{as: main assumption}
    We assume
%\begin{enumerate}
there exists $N \in \N \backslash \{0\}\cup \{\infty\}$ such that the flow $f$ and the noise sample paths $t \mapsto w_t$ are $(N-1)$-times continuously differentiable (i.e.,  $f \in C^{N-1}(\R^d,\R^d) $ and $ w(\cdot,\omega)\in C^{N-1}(\T,\R^d)$ for $\omega \in \Omega$ a.s.).
% \begin{itemize}
%     \item The flow $f$ is $N$-times continuously differentiable (ie. $f \in C^N(\R^d,\R^d) $),
%     \item The noise sample paths are $N$-times continuously differentiable (ie. $ w(\cdot,\omega)\in C^N(\T,\R^d)$ for $ \omega \in \Omega$ a.s.).
% \end{itemize}
%\item \label{as: det init cond}\textbf{Deterministic initial condition:}  %The initial condition $x_0$ is deterministic ($x_0=z \in \R^d$ a.s.).
%\end{enumerate}
\end{assumption}

%The first assumption is essential in what follows while the latter could be dispensed with (cf.~discussion).

\subsection{Motivation% for generalised coordinates
}

We can express the motion of derivatives of the SDE \eqref{eq: SDE}, simply by differentiating w.r.t. time 
% ($i=0,...,N+1$)
\begin{equation}
\label{eq: SDE motion of derivatives}
    \begin{split}
    %x_0&=z \\
        \frac d {dt}x_t &= f(x_t)+w_t
        \\
        \frac{d^2}{dt^2}x_t &= \frac d {dt}f(x_t)+\frac d {dt}w_t\\
                        &\vdots \\
        \frac{d^N}{dt^N} x_t &=\frac {d^{N-1}} {dt^{N-1}} f(x_t)+\frac {d^{N-1}} {dt^{N-1}} w_t\\   
        % &\vdots
    \end{split}
\end{equation}

%The above is a finite or countably infinite number of equations depending on whether $N$ is finite or infinite.

Fixing an initial time $t$ in \eqref{eq: SDE motion of derivatives}---without loss of generality $t=0$---yields a system of algebraic equations for the serial derivatives of the solution up to order $N$ at this time.

\begin{equation}
\label{eq: algebraic equation derivatives}
    \begin{split}
    %x_0&=z \\
        \frac d {dt}x_t\big|_{t=0} &= f(x_0)+w_0
        \\
        \frac{d^2}{dt^2}x_t\big|_{t=0} &= \frac d {dt}f(x_t)\big|_{t=0}+\frac d {dt}w_t\big|_{t=0}\\
                        &\vdots \\
        \frac{d^N}{dt^N} x_t\big|_{t=0} &=\frac {d^{N-1}} {dt^{N-1}} f(x_t)\big|_{t=0}+\frac {d^{N-1}} {dt^{N-1}} w_t\big|_{t=0}\\   
        % &\vdots
    \end{split}
\end{equation}
%Since $\frac {d^{m}} {dt^{m}} f(x_t)$ depends only on $f$ and on the derivatives 
Assuming one can solve \eqref{eq: algebraic equation derivatives} %for some $t$--- say $t=0$ ---
we can approximately recover %(an approximation of)
the solution trajectories of the SDE \eqref{eq: SDE motion of derivatives} from the serial derivatives via a Taylor expansion:

\begin{equation}\label{eq: taylor approximations}
\begin{split}
x_t &\approx \sum_{n=0}^{N} \frac{d^{n}}{dt^{n}}x_t\big|_{t=0} \frac{t^n}{n!}\\
        \frac{d}{dt}x_t &\approx \sum_{n=0}^{N-1} \frac{d^{n+1}}{dt^{n+1}}x_t\big|_{t=0} \frac{t^{n}}{n!}\\
        % \frac{d^2}{dt^2}x_t &\approx \sum_{i=0}^{N-1} \frac{d^{i+2}}{dt^{i+2}}x_t\big|_{t=0} \frac{t^i}{i!}\\
         &\vdots \\
          \frac{d^N}{dt^N}x_t &\approx \frac{d^{N}}{dt^{N}}x_t\big|_{t=0}
\end{split}
\end{equation}

Provided that the approximation by the Taylor expansion is sufficiently accurate, the typically hard problem of computing (or inferring) solution trajectories of SDEs yields to the much easier problem of computing (or inferring) serial derivatives of solutions at the time origin. It turns an uncountable number of equations---one \eqref{eq: SDE motion of derivatives} for each $t$---into a finite, or countably infinite, system of equations for $ \frac d {dt}x_t\big|_{t=0}, \ldots, \frac{d^{N}}{dt^{N}} x_t\big|_{t=0}$.

\subsection{Generalised coordinates}

%Generalised coordinates are a convenient set of notation and tools for analysing SDEs in this fashion.

This paper introduces a toolset for the analytic and numerical study of SDEs in this fashion. This requires embedding the SDE in a space of larger dimension, with auxiliary variables, which can be interpreted---in a first instance---as the position, velocity, acceleration, and higher order motion of the process. The goal will be to reformulate the SDE as a simple dynamical system in this higher dimensional space.

\subsubsection{The space of generalised coordinates}

The space of generalised coordinates of base dimension $d$ and order $n \in \N \cup \{\infty\}$ is
\begin{equation}
\label{eq: space of generalised coordinates}
   \G^{d,n}\triangleq \left\{\rmz^{(:n)}\triangleq \rmz^{(0:n)}\triangleq
   \begin{pmatrix}
\rmz^{(0)} \\
\rmz^{(1)}  \\
 \vdots\\
 \rmz^{(n)}  
\end{pmatrix} \mid  \rmz^{(i)} \in \R^{d},\forall i=0,...,n \right\}\cong \R^{d(n+1)}
\end{equation}
We treat elements of $\G^{d,n}$ as column vectors in $\R^{d(n+1)}$, which are finite or countably infinite dimensional depending on whether $n$ is finite or infinite. %The goal now is to reformulate \eqref{eq: SDE motion of derivatives} in generalised coordinates.

\subsubsection{The generalised flow}

Reformulating \eqref{eq: algebraic equation derivatives} as an algebraic equation in generalised coordinates requires the following:

\begin{claim}
\label{claim: gen flow}
There exists functions $\mathrm{f}^{(n)}:  \G^{d,n}\to \R^d, n=0, \ldots, N-1$ such that for all $t \in \T$
    \begin{equation}
        \mathrm{f}^{(n)}\left(x_t, \frac{d}{dt}x_t, \ldots, \frac{d^n}{dt^n}x_t\right)\triangleq\frac{d^n}{dt^n}f(x_t).
    \end{equation}
\end{claim}
Indeed, by the chain rule
\begin{equation}
\begin{split} 
\mathrm{f}^{(0)}\left(x_t\right)&= f(x_t), \quad
\mathrm{f}^{(1)}\left(x_t, \frac{d}{dt}x_t\right)=\nabla f(x_t)\frac{d}{dt}x_t \\
\mathrm{f}^{(2)}\left(x_t, \frac{d}{dt}x_t, \frac{d^2}{dt^2}x_t\right)&= \nabla f(x_t)\frac{d^2}{dt^2}x_t+\left(\frac{d}{dt}x_t\right)^\top\nabla^2 f(x_t)\frac{d}{dt}x_t \\
&\vdots
%\mathbf{f}^{(n)}\left(x_t, \frac{d}{dt}x_t\right)&\triangleq \nabla f(x_t)\frac{d}{dt}x_t
\end{split}
\end{equation}
where the remaining $\mathrm{f}^{(n)}$ can be obtained analytically through repeated applications of the chain rule on $f$.
Altogether, for $n\geq 1$
\begin{equation}
    \mathrm{f}^{(n)}\left(x_t, \frac{d}{dt}x_t, \ldots, \frac{d^n}{dt^n}x_t\right)= \nabla f(x_t)\frac{d^n}{dt^n}x_t+ \text{terms involving higher order derivatives of } f
\end{equation}
where these higher order terms\footnote{These higher order terms could be obtained explicitly using Faà di Bruno's formula.} vanish when $f$ is linear or $n=1$.

\begin{definition}[Generalised flow]
\label{def: gen flow}
    The generalised flow is the function $\mathbf f: \G^{d,N}\to \G^{d,N-1}$
    \begin{equation}
    \label{eq: def gen flow}
        \mathbf f(\rmz^{(:N)})\triangleq\begin{pmatrix}
        \mathrm{f}^{(0)}\left(\rmz^{(0)}\right)\\
         \mathrm{f}^{(1)}\left(\rmz^{(:1)}\right)\\
            \vdots\\
            \mathrm{f}^{(N-1)}\left(\rmz^{(:N-1)}\right)
        \end{pmatrix}
    \end{equation}
Note that the generalised flow is invariant w.r.t $\rmz^{(N)}$.
\end{definition}
This approximation will be convenient at various stages later on:
\begin{approximation}[Local linear approximation]
\label{ap: local lin approx}
We say that we operate under the \emph{local linear approximation} when we discard all contributions of derivatives of $f$ of order strictly higher than $1$.\footnote{Note that this is the same approximation that lies at the heart of the Extended Kalman Filter \cite{einicke2012smoothing}.}
\end{approximation}
Under the local linear approximation~\ref{ap: local lin approx} the generalised flow equals:

\begin{equation}
\label{eq: bold f local linear}
        \mathbf f(\rmz^{(:N)})=\begin{pmatrix}
        f\left(\rmz^{(0)}\right)\\
        \nabla f(\rmz^{(0)})\rmz^{(1)}
         \\
            \vdots\\
            \nabla f(\rmz^{(0)})\rmz^{(N-1)}
        \end{pmatrix}
\end{equation}

\subsection{Formulation in generalised coordinates}

We formulate a random dynamical system
\begin{equation}
    \mathbf x_t\triangleq \rmx^{(:N)}_t \triangleq  \rmx^{(0:N)}_t \triangleq \begin{pmatrix}
            \rmx^{(0)}_t  \\
              \rmx^{(1)}_t \\
             \vdots \\
            \rmx^{(N)}_t 
             \end{pmatrix}, \quad t\in \T
\end{equation}
on the space of generalised coordinates $\G^{d,N}$ that captures the essential dynamics of the original SDE in the sense that it is equivalent to the RHS of \eqref{eq: taylor approximations} with \eqref{eq: algebraic equation derivatives} as initial condition.

$\mathbf x_t$ is defined as the unique solution to the \textit{generalised Cauchy problem}

\begin{equation}
\label{eq: gen Cauchy problem}
    \begin{cases}
        \frac{d}{dt}\mathbf x_t= \mathbf D \mathbf x_t\\
        \mathbf D' \mathbf x_0 = \mathbf f(\mathbf x_0)+\mathbf w_0\\
            \rmx^{(0)}_0=z
    \end{cases}
\end{equation}
where
\begin{equation}
\label{eq: generalised variables for gen cauchy problem}
\begin{split}
       \mathbf D &\triangleq \underbrace{ \I\otimes
\overbrace{
\begin{pmatrix}
            0 &  1&  &&  \\
             & 0& 1 & &  \\
             &&0&\ddots& \\
             &&&\ddots&1\\
             &&&&0
        \end{pmatrix}}^{\Mat_{N+1,N+1}}}_{\Mat_{(N+1)d,(N+1)d}}, \quad \mathbf D'\triangleq\underbrace{ \I\otimes
\overbrace{\begin{pmatrix}
            0 &  1&  &&  \\
             & 0& 1 & &  \\
             &&\ddots&\ddots& \\
             &&&0&1
             \end{pmatrix}}^{\Mat_{N,N+1}}}_{\Mat_{Nd,(N+1)d}} \\ 
           % &\Rightarrow \mathbf D \mathbf x_t=(\rmx^{(1:N)}_t,0)\in \G^{d,N}, \mathbf D' \mathbf x_0=\rmx^{(1:N)}_0 \in \G^{d,N-1}\\
             \mathbf w_t&\triangleq  \mathrm w^{(:N-1)}_t \triangleq  \mathrm w^{(0:N-1)}_t \triangleq \begin{pmatrix}
            \mathrm w^{(0)}_t  \\
              \mathrm w^{(1)}_t \\
             \vdots \\
            \mathrm w^{(N-1)}_t 
             \end{pmatrix}\triangleq \begin{pmatrix}
             w_t  \\
              \frac {d} {dt} w_t \\
             \vdots \\
            \frac {d^{N-1}} {dt^{N-1}} w_t 
             \end{pmatrix}
\end{split}
\end{equation}
so that $\mathbf D \mathbf x_t=( \rmx^{(1:N)}_t\:\: 0 )^\top\in \G^{d,N}, \mathbf D' \mathbf x_0=\rmx^{(1:N)}_0 \in \G^{d,N-1}, \mathbf w_t\in \G^{d,N-1}$.
Note that the initial condition in \eqref{eq: gen Cauchy problem} is equivalent to \eqref{eq: algebraic equation derivatives} with $x_0=z$. Furthermore, the dynamics in \eqref{eq: gen Cauchy problem} are equivalent to saying that each coordinate $\rmx^{(n)}_t$ equals a Taylor polynomial centred at $t=0$ with coefficients $\rmx^{(n)}_0,\ldots, \rmx^{(N)}_0$.
Indeed, they are equivalent to
\begin{equation}
\label{eq: equation generalised process}
\begin{split}
\mathbf x_t&= \exp(t\mathbf D)\mathbf x_0\\
        \exp(t\mathbf D)&= \I\otimes\begin{pmatrix}
            1 &  t &  t^2/2&\ldots &  \frac {t^{N}}{N!}\\
             & 1& t &\ldots & \frac {t^{N-1}}{(N-1)!}  \\
             &&1&\ddots& \vdots\\
             &&&\ddots&t\\
             &&&&1
        \end{pmatrix}
\end{split}
\end{equation}
so that $\rmx_t^{(n)}=\sum_{i=0}^{N-n} \rmx^{(i+n)}_{0} \frac{t^i}{i!}$ as in the RHS of \eqref{eq: taylor approximations}. It is interesting to note that \eqref{eq: equation generalised process} can be read either as the solution to the ordinary differential equation in the first line of \eqref{eq: gen Cauchy problem} or as a Taylor approximation to \eqref{eq: SDE}. That the two are equivalent here is the key insight that underwrites use of generalised coordinates of motion.

\subsection{The full construct}

% If for $N=\infty$ the Taylor series are exact for a range of $t$ (we will see conditions and examples of this later), we obtain 
% \begin{equation}
%      \mathbf x_t= \begin{pmatrix}
%          \frac d {dt}x_t 
%         \\
%         \frac{d^2}{dt^2}x_t \\
%                         \vdots \\
%      \end{pmatrix} \quad \text{for such } t \in \T,
% \end{equation}
% and equivalently
% \begin{equation}
% \begin{split}
%     \mathbf D \mathbf x_t &= \mathbf f(\mathbf x_t)+\mathbf w_t, \
% \end{split}
% \end{equation}
% which is a restatement of \eqref{}.

Consider $N=\infty$. The following is an equivalent condition to that given by the generalised process solution to \eqref{eq: gen Cauchy problem} to match the serial derivatives of the SDE \eqref{eq: SDE}
at some time $t$
\begin{equation}
\label{eq: equiv condition for exactness}
     \mathbf x_t \triangleq \begin{pmatrix}
            \rmx^{(0)}_t  \\
              \rmx^{(1)}_t \\
             \vdots  
             \end{pmatrix}= \begin{pmatrix}
         \frac d {dt}x_t 
        \\
        \frac{d^2}{dt^2}x_t \\
                        \vdots \\
     \end{pmatrix} \iff     \mathbf D' \mathbf x_t = \mathbf f(\mathbf x_t)+\mathbf w_t
\end{equation}
where the right equation in \eqref{eq: equiv condition for exactness} corresponds to \eqref{eq: SDE motion of derivatives}. This characterisation will be convenient later on.

\section{Faithfulness}

We now look at the faithfulness of the formulation in generalised coordinates w.r.t. the original SDE. That is, to what extent and over which time interval the generalised Cauchy problem \eqref{eq: gen Cauchy problem} accurately encodes the solutions of \eqref{eq: SDE}. Obviously this depends on the value of $N$ and how many times the solutions of \eqref{eq: SDE} are differentiable, or whether they are analytic. So we mainly focus on the conditions for the regularity of the solutions to \eqref{eq: SDE}.

Throughout this section, we consider the solution of the SDE \eqref{eq: SDE} pathwise, that is, we fix an element of the sample space $\omega \in \Omega$, so that the corresponding noise sample path is determined $t \mapsto w_{\omega,t}\triangleq w(t,\omega)$, and we consider the regularity of the corresponding solution $t \mapsto x_{\omega,t} \triangleq x(t,\omega)$. We can summarise our setup as an ordinary non-homogeneous differential equation:
\begin{equation}
\label{eq: non-homogeneous nonlinear differential equation}
\begin{split}
    \frac{d}{dt}x_{\omega,t}  = F_\omega(t,x_{\omega,t}), \quad x_{\omega,0}=z\in \R^d,
\end{split}
\end{equation}
where $F_\omega:  \T \times \R^d \to \R^d,F_\omega(t,x)\triangleq f(x) + w_{\omega,t}$. Furthermore, we will write $x_{\omega}\triangleq x(\cdot,\omega),w_{\omega}\triangleq w(\cdot,\omega)$ for solution and noise trajectories.

\begin{remark}[Faithfulness under the local linear approximation]
    At present, it is unclear what faithfulness guarantees can be obtained under local linear approximation~\ref{ap: local lin approx} for non-linear flows $f$. We will derive our results without this approximation, and return this important problem in \cref{sec: future directions}.
\end{remark}

\subsection{Many times differentiable solutions}

\subsubsection{Sufficient conditions}

\begin{theorem}
\label{thm: regularity many times diff ODE}
If $F_\omega$ is $C^{N-1}$ (i.e.~$f$ and $w_{\omega}$ are $C^{N-1}$) for $N \in \N, N\geq 2$ in a neighbourhood of $(0,z)$, the initial value problem \eqref{eq: non-homogeneous nonlinear differential equation} has a unique solution $x_{\omega}: (-R_\omega, R_\omega)\to \R^d$ for some radius $R_\omega >0$. Furthermore, $x_{\omega} \in C^{N}((-R_\omega, R_\omega))$. The same holds with $N=1$ provided $f$ is also locally Lipschitz.
\end{theorem}

\begin{proof}
    Let $D\subset \T\times \R^d$ be a closed rectangle on which $F_\omega$ is $C^1$, containing $(0,z)$ in its interior. Then, $F_\omega(t,x)$ is Lipshitzian in $x$ for any fixed $t$ in this rectangle \cite[Lemma 1, Chapter 1.10]{ birkhoffOrdinaryDifferentialEquations1989}. The Picard-Lindelof theorem guarantees a unique solution $x_\omega$ to the differential equation \eqref{eq: non-homogeneous nonlinear differential equation}, defined on a time interval $t \in (-R_\omega, R_\omega)$ for some $R_\omega>0$ \cite[Theorem I.3.1]{coddingtonTheoryDifferentialEquations1984}. %https://en.wikipedia.org/wiki/Picard%E2%80%93Lindel%C3%B6f_theorem
    Furthermore, this solution is $C^1$. %This is because the derivative equals the RHS which is continuous
    %\LD{Question: can we extend the solution maximally and does that add anything to the table?} 
    % Thus, the composition $F_\omega(t,x(t))$ is continuous, but this is precisely $\frac{dx^\omega}{dt}$. Therefore, $x^\omega(t)$ is continuously differentiable. 
    Proceeding iteratively as in \eqref{eq: SDE motion of derivatives}, we obtain that the solution $x_\omega$ is $C^{N}$.
    We can repeat the proof in the case where $N=1$ and $f$ locally Lipschitz, as there exists a (possibly smaller) domain $D$ such that $F_\omega$ is $C^0$ \textit{and} Lipschitzian in the space variable for any fixed $t$.
\end{proof}

\subsubsection{Accuracy statement}

The accuracy of the generalised coordinate reformulation of the SDE is a direct consequence of Taylor's theorem. Under the conditions of Theorem~\ref{thm: regularity many times diff ODE}, we have $\e_\omega >0$ such that for all $t \in (-\e_\omega,\e_\omega)$
\begin{equation}
    \rmx^{(0)}_{\omega,t}=\sum_{n=0}^{N} \rmx^{(n)}_{\omega,0} \frac{t^n}{n!} =\sum_{n=0}^{N} \frac{d^{n}}{dt^{n}}x_{\omega,t}\big|_{t=0}\frac{t^n}{n!}= x_{\omega,t} + o_\omega(|t|^{N}).
\end{equation}
where $\rmx^{(0)}_{\omega,t}$ is the solution to the generalised Cauchy problem \eqref{eq: gen Cauchy problem} where we fixed the noise sample path $w_\omega$, and $o_\omega$ refers to the fact that the little-o remainder may also depend on $\omega$.

Thus, increasing the order $N$ yields improved approximations of $x_{\omega,t}$ on possibly smaller neighbourhoods of $t=0$. For any fixed neighbourhood however, it is possible that increasing $N$ decreases the quality of the approximation at non-zero points, e.g., as the Taylor polynomial blows up. This will be important to consider later in simulations (e.g. in numerical integration) as increasing $N$ does not necessarily yield more accurate results on a fixed time interval. It is only when the solutions of \eqref{eq: SDE} are analytic that increasing the order surely yields a better approximation on any fixed (albeit sufficiently small) interval.

\subsection{Analytic solutions}
\label{sec: Analytic solutions}

\subsubsection{Sufficient conditions}

\begin{definition}%[]%p111 birkhoff rota
A function $f:\R^d  \to \R^d$ is analytic on a domain $D\subseteq \R^d$ if, for any $y \in D$, there exist constants $b_{ij_0 j_1\hdots j_d}$ $(i=1,\ldots, d; j_1\hdots j_d=0,1,2, \ldots)$ and $\delta>0$, such that for any $i\in \{1,\ldots, d\}$:
\begin{equation}
\label{eq: analytic vector function}
    f(x)_i=\sum_{j_1\hdots j_d=0}^{\infty}  b^{(i)}_{ij_1\hdots j_d}\left(x_1-y_1\right)^{j_1}\ldots\left(x_d-y_d\right)^{j_d} \quad 
\end{equation}
provided that $\left|x_1-y_1\right|<\delta,\ldots, \left|x_d-y_d\right|<\delta. $
\end{definition}
The definition of analyticity for $F_\omega:\T\times\R^d\to\R^d$ can be stated analogously. Obviously, the noise $w_{\omega,t}$ and the flow $f$ are analytic on domains $U_w\subset \T,U_f\subset \R^d$, if and only if $F_\omega$ is analytic on $U_w\times U_f\subset \T\times \R^d$. 

We now state a special case of a classic theorem due to Cauchy \cite{cauchyCALCULINTEGRALMemoire2009} and Kovalevskaya \cite{kowalevskyZurTheoriePartiellen1875}.

\begin{theorem}[Cauchy–Kovalevskaya]
\label{thm: Cauchy–Kovalevskaya}
If $F_\omega$ is analytic (i.e.~$f$ and $w_\omega$ are analytic) in a neighbourhood of $(0\:\:z)$, then the initial value problem \eqref{eq: non-homogeneous nonlinear differential equation} has a unique solution $x_\omega: (-R_\omega, R_\omega)\to \R^d$ for some $R_\omega >0$. Furthermore, $x_\omega$ is analytic on $(-R_\omega, R_\omega)$. 
\end{theorem}

As usual, a function $f:D\subset \R^d  \to \R^d$ is analytic if and only if it can be extended to a complex holomorphic function $f:D'\subset \C^d  \to \C^d$ (where $D'\supset D$). This means that when $F_\omega$ is analytic around $(0,z)$, the initial value problem \eqref{eq: non-homogeneous nonlinear differential equation} can be extended to the complex plane. Complex analysis then provides an additional toolset to analyse these equations, cf.~\cite[Chapter 4]{ teschlOrdinaryDifferentialEquations2012}.

\begin{proof}
Uniqueness of solutions is guaranteed by the Picard-Lindelof theorem  \cite[Theorem I.3.1]{coddingtonTheoryDifferentialEquations1984}. Existence of an analytic solution can be shown in at least two ways and we provide references below:
\begin{itemize}
    \item \cite[Theorem 5, p127]{birkhoffOrdinaryDifferentialEquations1989} proves the result through the classical method of analytic majorisation: 1) finding a formal power series solution, and then showing (through a majorisation argument) that it has a positive radius of convergence. The argument is made for $d=1$, but it generalises to arbitrary $d$.
    \item \cite[Theorem 4.2]{ teschlOrdinaryDifferentialEquations2012} proves the result for the equivalent initial value problem on the complex domain, exploiting the fact that the Picard iterates---in the proof of Picard iteration---are analytic and converge to an analytic function.
\end{itemize}
\end{proof}

\subsubsection{Accuracy statement}

Under the conditions of Theorem~\ref{thm: Cauchy–Kovalevskaya}, $N=\infty$ implies that the generalised coordinate solution coincides with the original SDE solution %, we recover the true solution to the SDE. %($f$ and $w_{\omega}$ analytic around the initial condition).

\begin{equation}
\label{eq: analytic case accuracy}
    \rmx^{(0)}_{\omega,t}=\sum_{n=0}^{\infty} \rmx^{(n)}_{\omega,0} \frac{t^n}{n!} =\sum_{n=0}^{\infty} \frac{d^{n}}{dt^{n}}x_{\omega,t}\big|_{t=0}\frac{t^n}{n!}= x_{\omega,t}
\end{equation}
over some time interval $t\in (-R_\omega,R_\omega)$. The question now is, what is the radius $R_\omega>0$ and what can we say about its distribution as $\omega$ varies?

\begin{itemize}
    \item We will give an uniform value for $R_\omega$ for the linear SDE driven by a stationary Gaussian process with analytic sample paths in Section~\ref{sec: smooth and analytic fluctuations}.
    \item In the non-linear $f$ case, \cite[Theorem 4.1]{teschlOrdinaryDifferentialEquations2012} gives a tight (non-uniform in $\omega$) lower bound on $R_\omega$, however this requires first extending and analysing $F_\omega$ (and thus $w_{\omega}$) in the complex domain, which may be impractical. Relevant here is also the Cauchy–Kovalevskaya bound on the radius of the solutions of analytic PDEs \cite[Theorem 9.4.5]{hormanderAnalysisLinearPartial2013}, also non-uniform in $\omega$.
    \item Lastly, the following simple example shows that even though the coefficients of an ODE such as \eqref{eq: non-homogeneous nonlinear differential equation} may be entire, the solution may have an arbitrarily small radius of analyticity $R_\omega$:
    \begin{example}
The one-dimensional differential equation $\dot x_t = 1 +x^2_t$ with initial condition $x_0=z\in \R$ has a unique solution $x_t=\tan(t+c)$, where $c=c(z)=\tan^{-1} z\in (-\pi/2,\pi/2)  $. Its maximal domain of definition around the time origin is $(-\pi/2-c, \pi/2-c)$. Its radius of convergence around the time origin is therefore $R(z)=\min \{|-\pi/2-c|, \pi/2-c\}$. In particular, $z\nearrow+\infty$ implies $c(z) \nearrow \pi/2$ and $R(z)\searrow 0 $. This holds despite the fact that the radius of convergence of the flow $F_\omega(t,x)=1+x^2$ is infinite.
\end{example}
See also \cite[p. 113]{teschlOrdinaryDifferentialEquations2012} for an example where the ODE solution is defined for all $t \geq 0$ but its radius of convergence can be made arbitrarily small. These non-examples remove the hope that the generalised coordinate solution $\mathbf{x}$ always exists globally in time when the flow $f$ and random fluctuations $w$ are entire almost surely.
\end{itemize}

\section{Usefulness}

We now observe that the generalised coordinate formulation \eqref{eq: gen Cauchy problem} can be used to approximate the solutions of a wide range of stochastic differential equations with possibly singular drift and rough noise. Indeed, given an SDE with a (possibly singular) flow $f$ driven by (possibly rough) noise fluctuations $w$, there usually is an \textit{approximating SDE} of the form \eqref{eq: SDE} with analytic flow and fluctuations---for which the reformulation in generalised coordinates \eqref{eq: gen Cauchy problem} is exact by Section~\ref{sec: Analytic solutions}.% obtained by smoothing differentiation equation a noise that is sample analytic. 

The intuition is that many types of noise can be approximated by noises with analytic sample paths through smoothing. The smoothing operation considered here is convolution with a Gaussian: if a noise process is smoothed in this way, one can recover it---or approximate it arbitrarily well---by letting the Gaussian standard deviation tend to zero. Intuitively, two SDEs with noises that are arbitrarily close and that are identical otherwise should produce solutions that are arbitrarily close (see \cref{sec: wong zakai}). The same ideas go for approximating the drift in SDEs through smoothing.

% approximate a stochastic differential equation buy another where the noise has been involved with a Gaussian kernel, one should recover the solutions of the original equation by 

\subsection{Analytic approximation of noise}

We now formalise these intuitions by showing sufficient conditions under which Gaussian convolution yields analytic sample paths.

\subsubsection{The case of white noise}

White noise is the prototypical example of random fluctuations in stochastic differential equations (known as diffusion equations), due to their long history as an idealised model of thermal molecular noise \cite{einsteinUberMolekularkinetischenTheorie1905}.

\begin{definition}[White noise, informal]
    \label{def: white noise}
 White noise $\omega_t$ is a mean-zero (stationary) Gaussian process on $t\in \R$ with autocovariance $\mathbb E[\omega_t \omega_s]= \delta(t-s)$ where $\delta$ is the Dirac delta function. 
\end{definition}

It is desirable that generalised coordinates be able to approximate the solutions of diffusion equations, at least in principle. This turns out to be the case: we show that diffusion equations where the driving (white) noise is smoothed satisfy the necessary conditions for generalised coordinates.

To show this we analyse white noise convolved with a Gaussian 
\begin{equation}
    \label{eq: Gaussian convolved white noise}
    w_t=\int_{\mathbb R} \omega_s \phi({t-s}) d s,  \quad t\in\R, \quad \phi_t\triangleq\phi(t) \triangleq (\sqrt{2\pi} \sigma)^{-1}e^{-\frac{t^2}{2\sigma^2} }
    %e^{-\beta  t^2}
\end{equation}
where $\sigma>0$ is the standard deviation parameter. %, measuring the \textit{smoothness} of the resulting approximation to white noise.
Formally, as $\sigma\searrow 0$, $\phi \to \delta$ and thus $w_t \to  \omega_t$.

Gaussian convolved white noise \eqref{eq: Gaussian convolved white noise} is a mean-zero stationary Gaussian process with Gaussian autocovariance 
\begin{equation}
\label{eq: autocov Gaussian convolved white noise}
    \begin{split}
       \E[w_t w_{t+h}] &= \E\left[\int_\R \omega_s \phi_{t-s} d s\int_\R \omega_\tau \phi_{t+h-\tau} d \tau\right]\\
       &= \int_\R \int_\R \phi_{t-s}\phi_{t+h-\tau} \E[w_s w_\tau] ds d\tau\\
       &= \int_\R \phi_{t-s}\phi_{t+h-s} ds\\
       &= \frac{1}{2\pi \sigma^2}\int_\R  e^{-\frac{(t-s)^2}{2\sigma^2}-\frac{(t+h-s)^2}{2\sigma^2}} ds\\
        & =\frac{1}{2\pi \sigma^2}e^{- \frac{h^2}{4\sigma^2}}\int_\R e^{-\frac{\left(s-(t+h/2)\right)^2}{\sigma^2}} d s \\ 
        & =\frac{1}{2\sqrt\pi \sigma}e^{- \frac{h^2}{4\sigma^2}}
    \end{split}
\end{equation}
where the last line follows from the well-known value of the Gaussian integral.

\begin{lemma}
Gaussian convolved white noise \eqref{eq: Gaussian convolved white noise} has the power series representation
    \begin{equation}
    \label{eq: power series convolved white noise}
        w_{\sqrt{2}\sigma t}\stackrel{\ell}{=} (4\pi\sigma^2)^{-\frac 14}e^{-\frac{t^2}{2}} \sum_{n=0}^\infty \frac{ t^n}{\sqrt{n!}}\xi_n ,\quad t\in \R,
    \end{equation}
where $\xi_n \sim \mathcal N(0,1), n \in \N$, are i.i.d. standard normal random variables and $\stackrel{\ell}{=} $ denotes equality in law.\footnote{Equality in law means that both processes have the same finite dimensional distributions.}
%, however, this does not necessarily mean that they take the sample paths for each realisation of the sample space $\omega \in \Omega$.}
\end{lemma}
\begin{proof}
First, we define
\begin{equation*}
        y_t\triangleq(4\pi\sigma^2)^{-\frac 14}e^{-\frac{t^2}{2}}\sum_{n=0}^\infty \frac{ t^n}{\sqrt{n!}}\xi_n, \quad t\in \R.
\end{equation*}
     Since the law of a Gaussian process is determined by the mean and autocovariance, we just need to show that 1) $y$ is a Gaussian process, 2) $y$ has the same mean as $w$, 3) $y$ has the same autocovariance as $w$. To do this we need the following basic fact about Gaussian random variables (which can be shown using Levy's continuity theorem; cf.~partial proofs here\cite{stackexchange674128, stackexchange1110461}):
         \begin{fact}
         \label{fact: gaussian rvs}
            Given $v_n \in \mathbb{R}^m$ for $n \in \N$, and $\sum_{n \geq 0}\left\|v_n\right\|^2<+\infty$, then $v\triangleq\sum_{n \geq 0}  v_n\xi_n$ is an $m$-dimensional Gaussian random variable, with mean $\E[v]=0$ and covariance $\E\left[v v^\top\right]_{ij}=\sum_{n \geq 0}\left(v_n\right)_i\left(v_n\right)_j$. 
         \end{fact}
     \begin{enumerate}
         \item Consider a finite sequence of times $t_1, \ldots, t_m \in \R$, and define $v\triangleq (y_{t_1}, \ldots, y_{t_m})$. Showing that $y$ is a Gaussian process amounts to showing that $v$ is joint normal. By construction
         \begin{equation*}
             v=\sum_{n \geq 0} v_n \xi_n  \quad \text{where } v_n \in \R^m, \:(v_n)_j\triangleq(4\pi\sigma^2)^{-\frac 14}e^{- {t_j^2}/{2}} \frac{ t_j^n}{\sqrt{n !}}.
             %v_n=\begin{cases}
                % \left(\frac{\pi}{2 \beta}\right)^{\frac{1}{4}} e^{-\frac{\beta}{2} t_n^2} \frac{1}{\sqrt{n !}}\left(\sqrt{\beta} t_n\right)^n, \quad \text{if } \quad 1 \leq n \leq m\\
                %  0, \quad \text{otherwise}
             % \end{cases}
         \end{equation*}
         To use the basic fact we observe that
         \begin{equation}
         \begin{split}
            &\sum_{n \geq 0}\left\|v_n\right\|^2 =\sum_{n \geq 0}\sum_{1 \leq j \leq m}  \left(2\sqrt\pi\sigma\right)^{-1} e^{- t_j^2} \frac{ \left(t_j^{2}\right)^n}{n!} \\
            =&\sum_{1 \leq j \leq m}  \left(2\sqrt\pi\sigma\right)^{-1} e^{-t_j^2} \sum_{n \geq 0}\frac{ \left(t_j^{2}\right)^n}{n!}  =\sum_{1 \leq j \leq m} \left(2\sqrt\pi\sigma\right)^{-1} <\infty.  
         \end{split}
        \end{equation}
         By the basic fact~\ref{fact: gaussian rvs}, $v$ is joint normal.
         \item By the basic fact~\ref{fact: gaussian rvs}, %(i.e.~dominated convergence theorem applied to the series in \eqref{eq: power series convolved white noise})
         $ \E[v]=0$, so $y$ has zero-mean.
         \item The basic fact~\ref{fact: gaussian rvs} yields the autocovariance of $y$
\begin{align*}
    \E[y_ty_{s}]&=\sum_{n\geq0 }(4\pi\sigma^2)^{-\frac 14} e^{-\frac{t^2}{2} } \frac{t^n}{\sqrt{n !}}(4\pi\sigma^2)^{-\frac 14} e^{-\frac{s^2}{2} } \frac{ s^n}{\sqrt{n !}} \\
    &= \left(2\sqrt\pi\sigma\right)^{-1} e^{-\frac{(t^2 +s^2)}{2}}\sum_{n \geq 0}\frac{( t s)^n}{\sqrt{n !}} \\&=\left(2\sqrt\pi\sigma\right)^{-1} e^{-\frac{(t^2 +s^2)}{2}}e^{ ts}
    =\left(2\sqrt\pi\sigma\right)^{-1} e^{-\frac{(t -s)^2}{2}}
\end{align*}
By inspection of \eqref{eq: autocov Gaussian convolved white noise} this is the autocovariance of $t\mapsto w_{\sqrt 2 \sigma t}$.
 \end{enumerate}   
\end{proof}

\begin{theorem}
    Gaussian convolved white noise \eqref{eq: Gaussian convolved white noise} $w_t$
    % \begin{align*}
    %     y_t\triangleq  (\frac{\pi}{2 \beta})^\frac{1}{4} e^{-\frac{\beta}{2} t^2} \sum_{n \geq 0} \frac{X_n}{\sqrt{n!}} (\sqrt{\beta} t)^n
    % \end{align*}
    extends to a process for $t\in \C$, with entire sample paths (i.e.~analytic on $\C$), almost surely.
\end{theorem}

\begin{proof}
    We just need to show that the series \eqref{eq: power series convolved white noise} converges for every $t \in \C$ almost surely. It will then define an (almost surely) entire random function $\C\to \C$.
    
    By Cauchy-Hadamard's theorem \cite[Theorem 2.5]{steinComplexAnalysis2003} the radius of convergence $R$ of the power series \eqref{eq: power series convolved white noise} is:
    \begin{equation}
1 / R=\limsup_{n \to \infty} \left|\frac{\xi_n}{\sqrt{n!}} \right|^{1 / n}
\end{equation}
\begin{claim}
We claim that $\left|\frac{\xi_n}{\sqrt{n!}}\right|^{1/n}\xrightarrow{n \to \infty} 0$ almost surely.
\end{claim}

The claim implies that $1 / R =0$ and $R =+\infty$ almost surely.
\begin{proofclaim}
    %Recall the first Borel-Cantelli lemma: if for a sequence of events $A_n, n\in \N$, we have $\sum_{n \geq 0} \p (A_n)<+\infty$, then $\p\left(\limsup _{n \rightarrow \infty} A_n\right)=0 ,$ where $\limsup _{n \rightarrow \infty} A_n\triangleq \bigcap_{n=0}^{\infty} \bigcup_{k=n}^{\infty} A_k .$ 
        Note that $\frac{\xi_n}{\sqrt{n!}}$ is a mean-zero normal random variable with standard deviation $\frac{1}{\sqrt{n!}}$. Defining the sequence of events $A_n\triangleq\left\{\left|\frac{\xi_n}{\sqrt{n!}}\right|^{1/n}\geq \frac {1}{n^{1/4}}\right\},$ we have, by Chebyschev's inequality:
    \begin{align*}
        \p(A_n)&=\p\left(\frac{|\xi_n|}{\sqrt{n!}}\geq\frac {\sqrt{n!}}{n^{n/4}}\frac {1}{\sqrt{n!}}\right)\leq \frac{n^{n/2}}{n!}.
    \end{align*}
    By Stirling's formula, there exists a constant $C>0$ such that $n!\geq Cn^n$ for all $n$. So $\sum_{n \geq 0} \p (A_n)\leq \sum_{n \geq 0} \frac{n^{n/2}}{n!}\leq \sum_{n \geq 0} \frac{1}{C n^{n/2}} < +\infty.$ The first Borel-Cantelli lemma then yields $\p\left(\liminf _{n \rightarrow \infty} A_n^c\right)=1 ,$ where $\liminf _{n \rightarrow \infty} A_n^c \triangleq \bigcup_{n=0}^{\infty} \bigcap_{k=n}^{\infty} \left\{\left|\frac{\xi_n}{\sqrt{n!}}\right|^{1/n}< \frac {1}{n^{1/4}}\right\} .$ In other words, with probability $1$, $\left|\frac{\xi_n}{\sqrt{n!}}\right|^{1/n}< \frac {1}{n^{1/4}}$ occurs for all but finitely many $n$. In particular, $\left|\frac{\xi_n}{\sqrt{n!}}\right|^{1/n}\to 0$ with probability $1.$
\end{proofclaim}
\begin{remark}
        Note that the convergence of $\left|\frac{\xi_n}{\sqrt{n!}}\right|^{1/n}\to 0$ holds almost surely but not surely as any sequence of real numbers can be realised by $\left|\frac{\xi_n}{\sqrt{n!}}\right|^{1/n}, n \geq 0$. However the set of sequences not converging to $0$ has probability zero. This shows that $w_t$ extends to the complex plane only almost surely, %but not surely, and thus, 
 and thus all our future statements leveraging the analytic sample paths of Gaussian convolved white noise are almost sure statements.
\end{remark}
\end{proof}

\subsubsection{The case of integrable noise}

Now consider an arbitrary noise process $w: \T \times \Omega \rightarrow \mathbb{R}$: we unpack conditions under which its sample paths are analytic when convolved with a Gaussian. For this, we fix an element of the sample space $\omega \in \Omega$, so that the corresponding noise sample path $w_\omega$ is determined. %$t \mapsto \triangleq w(t,\omega)$.

Consider the Gaussian convolution
\begin{align}
    \label{eq: convolution}
 w_\omega* \phi: \C\to \C, \: z\mapsto \int_\T w_{\omega,t}\phi(z-t)\d t ,
\end{align}
where the Gaussian is given, without loss of generality (up to re-scaling of space and the convolution) by $\phi : \C\to \C,  z \mapsto e^{-z^2}$.

\begin{remark}
    Note that the Gaussian $\phi$ is unbounded on the complex plane
\begin{align}
\label{eq: phi unbounded}
    \phi(x+i y)=e^{-(x+i y)^2}=e^{-x^2-2 i x y+y^2} =e^{-x^2+y^2} e^{-2 i x y}\Rightarrow |\phi(x+i y)|=e^{-x^2+y^2}
\end{align}
but bounded on its compact subsets.
\end{remark}

\textbf{Question:} under what conditions on the sample path $w_\omega$ is the convolution \eqref{eq: convolution} well-defined?

\begin{lemma}
\label{lemma: suff conds conv exists}
        Suppose that $ w_\omega\in L^1(\T)$, then \eqref{eq: convolution} is well-defined (i.e.~the integral converges).
\end{lemma}

Lemma~\ref{lemma: suff conds conv exists} is relatively sharp given no further assumptions on the noise as the integral \eqref{eq: convolution} may not converge for $w_\omega\in L^1_\text{loc}(\T)$ and unbounded $\T$. %For instance, the integral diverges almost surely when with $w_t=t \xi , \xi \sim \mathcal N(0,1)$, $\T=(0,\infty)$.

\begin{proof}[Proof of Lemma~\ref{lemma: suff conds conv exists}]
Let $z=x+iy\in \C$. By \eqref{eq: phi unbounded},
    \begin{align*}
      \left| \int_\T  w_{\omega,t}\phi(z-t) \d t\right| \leq \int_\T | w_{\omega,t}\phi(z-t)| \d t=\int_\T | w_{\omega,t}| e^{-(x-t)^2+y^2}\d t\leq e^{y^2}\int_\T | w_{\omega,t}|\d t<\infty.
    \end{align*}
\end{proof}

\textbf{Question:} under what conditions on $w_\omega$ is the convolution \eqref{eq: convolution} analytic on $\T$?

\begin{proposition}
\label{prop: Gaussian conv is entire}
        If $w_\omega \in L^1(\T)$, then the convolution \eqref{eq: convolution} is entire. In particular it is analytic on $\T$.
\end{proposition}
\begin{proof}[Proof of Proposition~\ref{prop: Gaussian conv is entire}]
Note that for $z=x+iy\in\C$,
\begin{equation}\label{eq: complex derivative gaussian}
    \left|\frac{d}{dz}\phi(z)\right| = |-2ze^{-z^2}| = |z||e^{-z^2}| = \sqrt{x^2+y^2}e^{-x^2+y^2}.
\end{equation}
Letting $A_{y,\epsilon} \triangleq \R\times (y-\epsilon, y+\epsilon) \subset \C$ for $y\in \R$ and $\epsilon>0$, we see that \eqref{eq: complex derivative gaussian} is bounded for $z\in A_{y,\epsilon}$. Thus it is in particular Lipschitz continuous on $A_{y,\epsilon}$: $\exists C_{y,\epsilon}>0, \forall z,z'\in A_{y,\epsilon}$,
\begin{equation}
\label{eq: Gaussian approx Lipschitz}
        |\phi(z')-\phi(z)|\leq C_{y,\epsilon}|z'-z|.
\end{equation}
So for $z=x+iy\in\C$,
\begin{equation}
          \label{eq: complex derivative convolution}
              \frac{d}{dz}(w_{\omega}* \phi)(z)=\lim_{h\rightarrow0} \int_\T w_{\omega,t}\frac{\phi(z+h-t)-\phi(z-t)}{h}\d t=\int_\T w_{\omega,t}\frac{d}{dz}\phi(z-t)\d t
          \end{equation}
where we moved the limit inside the integral by the dominated convergence theorem, as we have the domination
          \begin{equation}\label{eq: estimate difference ratio}
              \left |w_{\omega,t}\frac{\phi(z+h-t)-\phi(z-t)}{h}\right|\leq |w_{\omega,t}|C_{y,\epsilon},
          \end{equation}
for all $|h|<\epsilon$ by \eqref{eq: Gaussian approx Lipschitz}, and the RHS of \eqref{eq: estimate difference ratio} is integrable in $t$. So \eqref{eq: complex derivative convolution} shows that $w_{\omega}* \phi$ is holomorphic on $\C$, i.e.~it is entire.
\end{proof}

\begin{remark}
    \begin{itemize}
        \item If $\T$ is bounded then $C(\bar \T) \subset L^1(\bar \T)$, so any noise process $w$ with almost surely continuous sample paths becomes almost surely analytic on its domain when convolved with a Gaussian. For separable, stationary Gaussian processes on a compact interval, almost sure integrability and continuity of sample paths are equivalent \cite[Theorem 1]{belyaevLocalPropertiesSample1960}).%https://www.proquest.com/openview/e22a6a27a8593b057a55a07e30c72899/1?pq-origsite=gscholar&cbl=666297
        \item If $\T$ is unbounded and the process $w$ is non-zero and stationary, then one should not expect it to have samples in $L^1(\bar \T)$; intuitively, stationarity implies that the sample paths behave the same everywhere, and thus the accumulated non-zero $L^1$ norms on any bounded interval add up to infinity on an unbounded time interval $\T$.  
    \end{itemize}
\end{remark}

\subsection{Analytic approximation of SDEs}
\label{sec: wong zakai}

After approximating many types of noises by analytic ones through Gaussian convolution, it begs the question: how do the initial and resulting approximating SDE relate? This is in essence the content of the Wong-Zakai theorems~\cite{wong} and extensions thereof.

\subsubsection{The Wong-Zakai theorem for white noise}

% The most basic form of the Wong-Zakai theorem says that the is it an approximation

%p.497]{Ikeda Watanabe} : 
% Next we obtain an approximation theorem for solutions of stochastic differential equations. Such theorems have been discussed by several authors: McShane [115], Wong-Zakai [181], Stroock-Varadhan [158], Kunita [93], Nakao-Yamato [126], Malliavin [106] and Ikeda-Nakao- Yamato [52]. \cite[Theorem 7.2. p.497]{Ikeda Watanabe}, 

The Wong-Zakai theorem \cite[Theorem 2.8]{gyongyWongZakaiApproximations2004} says that for a sequence of piece-wise differentiable approximations $W_{\sigma, t}$ to an ($m$-dimensional) standard Wiener process $W_t$, in the sense that for every $T>0$,
\begin{equation}
     \lim _{\sigma \searrow 0} \E\left[\sup_{0 \leq t \leq T}\left|W_t-W_{\sigma, t}\right|^2 \right]= 0,
\end{equation}
the solutions to the corresponding SDE
\begin{align}
\label{eq: smoothed SDE}
    d x_{\sigma, t}&=f\left(x_{\sigma, t}\right) d t+\varsigma d W_{\sigma, t}, \quad x_{\sigma, 0}=z
\end{align}
converge in the following sense
\begin{align}
\label{eq: convergence Wong Zakai theorem}
 \lim _{\sigma \searrow 0} \E\left[\sup _{0 \leq t \leq T}\left|x_t-x_{\sigma, t}\right|^2\right]=0,\quad \forall T>0,
\end{align}
to the solution of the limiting Stratonovich SDE
\begin{align}
\label{eq: limit SDE}
    % d x_{\sigma, t}&=f\left(x_{\sigma, t}\right) d t+\sigma d W_{\sigma, t}, \quad x_{\sigma, 0}=z\\
    d x_{t}&=f(x_t) d t+\varsigma d W_{ t}, \quad
x_{0}=z,
\end{align}
assuming that the noise is additive, i.e.~$\varsigma \in \Mat_{d,m}$ is constant. In other words, the limit of the SDE \eqref{eq: smoothed SDE}
 is the Ito/Stratonovich SDE \eqref{eq: limit SDE} with the same drift and volatility.\footnote{Although not considered here, we emphasise that when the noise is \textit{not} additive, the drift of the limiting SDE is different; see \cite[Theorem 2.8]{gyongyWongZakaiApproximations2004},\cite[eq. 7.45 p.497]{watanabeStochasticDifferentialEquations1981}. This is because an additional L\'{e}vy area correction term needs to be taken into account~\cite{PhysRevE.88.062150},\cite[Sec. 11.7.7]{pavliotisMultiscaleMethodsAveraging2010},\cite[Sec. 5.1]{pavliotisStochasticProcessesApplications2014},\cite[Sec. 3.4]{frizCourseRoughPaths2020}. The L\'{e}vy area correction is important to consider as it can drastically change the dynamical and stationary properties of an SDE, such as the structure and nature of its stationary distributions~\cite{DiHoP_23}. \label{footnote: wong}} Crucially, \cite[Theorem 2.8]{gyongyWongZakaiApproximations2004} gives a rate of convergence for \eqref{eq: convergence Wong Zakai theorem}. 
% Then the Wong-Zakai theorem
% \cite[Theorem 7.2.]{Ikeda Watanabe} says that for every $T>0$ we have
% \begin{align*}
%  \lim _{\sigma \searrow 0} \E\left[\sup _{0 \leq t \leq T}\left|x_t-x_{\sigma, t}\right|^2\right]=0.  
% \end{align*}
This applies to the preceding section on convolved white noise, whence
\begin{align*}
    \begin{cases}
        \frac{d}{dt} x_{\sigma, t} =f\left(x_{\sigma, t}\right) +\varsigma  \left(\omega *\phi_{\sigma}\right)_t\\
        x_{\sigma, 0}=z
    \end{cases} \xrightarrow{\sigma \searrow 0} \quad\begin{cases}
        \frac{d}{dt} x_t = f(x_t)+\varsigma  \omega_t\\ 
x_{0}=z
    \end{cases} 
\end{align*}
where convergence is pathwise in the sense of \eqref{eq: convergence Wong Zakai theorem} and the Gaussian is $\phi_\sigma(t) = (\sqrt{2\pi} \sigma)^{-1}e^{-\frac{t^2}{2\sigma^2} }$.

All this means that we can use generalised coordinates to analyse stochastic differential equations driven by white noise, by smoothing the noise and analysing the resulting equation. When the flow is non-differentiable this needs to be smoothed as well for the approximating equation to have differentiable solutions. Luckily, there are also Wong-Zakai theorems for smooth approximations of the drift and the noise, e.g. \cite{lingWongZakaiTheoremSDEs2022}. 

Some care must be taken, noting that the radius of convergence of the generalised coordinate solution is expected to tend to zero as one takes the limit of rough noise \eqref{eq: convergence Wong Zakai theorem}. Most likely, one should exploit rates of convergence to the original diffusion equation such as \cite[Theorem 2.8]{gyongyWongZakaiApproximations2004} to relate the smooth generalised coordinate solution (i.e. with small finite $\sigma$) to the diffusion solution in expected sup-norm, in order to make generalised coordinates practically applicable to the study of rough equations.

\subsubsection{The Wong-Zakai theorem for rough paths}

But what about other types of rough noises that might be driving a stochastic differential equation? As there exists a fairly comprehensive theory for analysing SDEs driven by white noise, the appeal of generalised coordinates almost surely lies in offering analysis tools for SDEs driven by more complex and rough noises.

%As for white noise, 
As it turns out, Wong-Zakai approximation theorems for SDEs driven by more complex rough paths have been studied extensively. Such rough processes include semi-martingales \cite{konecnyWongZakaiApproximationStochastic1983,coutinSemimartingalesRoughPaths2005}, fractional Brownian motion \cite{viitasaariStationaryWongZakai2022,scorolliWongZakaiApproximationsQuasilinear2021} and many other Gaussian processes \cite{coutinStochasticAnalysisRough2002}, reversible Markov processes \cite{bassExtendingWongZakaiTheorem2002} as well as Markov processes with uniformly elliptic generator in divergence form \cite{frizUniformlySubellipticOperators2008,lejayStochasticDifferentialEquations2008}. Nowadays, smooth approximations of SDEs driven by rough paths are best understood within the comprehensive theory of rough paths \cite{frizRoughPathLimits2009,frizMultidimensionalStochasticProcesses2010,frizCourseRoughPaths2020}. The idea is that there is a suitable topology on the space of paths (the rough path topology) such that the map sending a noise sample path to the corresponding solution of a stochastic differential equation (the Ito-Lyons map) is continuous \cite[Theorem 8.5]{frizCourseRoughPaths2020}. A general Wong-Zakai approximation theorem analoguous to \eqref{eq: convergence Wong Zakai theorem} then follows trivially from this general result; see \cite{frizRoughPathLimits2009}, \cite[Section 17.4]{frizMultidimensionalStochasticProcesses2010} and \cite[Theorem 9.3]{frizCourseRoughPaths2020}. 

\chapter{Analysis in generalised coordinates}
 \label{chap: 2}

\section{Generalised fluctuations}

To be able to solve the generalised coordinate process $\mathbf x_t$ \eqref{eq: gen Cauchy problem}, we need knowledge about the random fluctuations in generalised coordinates $\mathbf w_t$ \eqref{eq: generalised variables for gen cauchy problem}. In this section we ask: what are the statistics of generalised fluctuations?

\subsection{Second order statistics}
\label{sec: second order statistics}

We consider \emph{wide-sense stationary} noise processes $w$ (i.e.~processes with stationary mean and covariance)\footnote{We discuss how to dispense with the wide-sense stationarity assumption in the Discussion (Section~\ref{sec: future directions}).} that have $(N-1)$-times continuously differentiable sample paths, almost surely.

The noise assumption is generally taken to imply zero-mean $\E[w_t]=0, \forall t\in \T$ (as the mean can be added to the drift), so trivially $\E[\mathbf w_t]=0, \forall t\in \T$.

The autocovariance function
\begin{equation}
\label{eq: autocovariance}
        \kappa(h)\triangleq \E[w_t w^{\top}_{t+ h}]
\end{equation}
is independent of $t$ by the wide-sense stationarity assumption. Interestingly, the second order statistics of $\mathbf w$ can formally be expressed in terms of the serial derivatives of $\kappa$.

\begin{lemma}
\label{lemma: autocov of generalised fluctuations}
Formally, for $n,m =0, \ldots, N-1$
\begin{equation}
\label{eq: autocov n m}
    \E[\rmw^{(n)}_t \rmw^{(m)\top}_{t+h}]= (-1)^n\kappa^{(n+m)}(h)
\end{equation}
where  $\kappa^{(n+m)}(h)\triangleq \frac{d^{n+m}}{dh^{n+m}}\kappa(h)$.
\end{lemma}

\begin{remark}[Formal calculation]\label{rmk: formal calculation}
    By `formally' we mean that the proof of Lemma~\ref{lemma: autocov of generalised fluctuations} may need additional conditions on $w$ to hold in practice. Specifically, we do not justify the steps for getting the limit outside of the expectation in the first equality in \eqref{eq: permut deriv integral} and \eqref{eq: permut deriv integral 2}. We will give sufficient conditions that make Lemma rigorous for Gaussian process $w$ in Section~\ref{sec: GP noise}.
\end{remark}

\begin{proof}[Proof of Lemma~\ref{lemma: autocov of generalised fluctuations}]
\begin{claim}
    \begin{equation}
    \label{eq: autocov 0 m}
    \E[\rmw^{(0)}_t\rmw^{(m)\top}_{t+h}]= \kappa^{(m)}(h)
\end{equation}
\end{claim}
\begin{proofclaim}
We prove \eqref{eq: autocov 0 m} by induction on $m$. The case $m=0$ is \eqref{eq: autocovariance}. We show the statement for $m\geq 1$ assuming it holds for $m-1$:
    \begin{equation}
    \label{eq: permut deriv integral}
    \begin{split}
    \E[\rmw^{(0)}_t \rmw^{(m)\top}_{t+h}] &= \lim_{s \to 0}\frac 1 s\E[\rmw_{t}(\rmw^{(m-1)\top}_{t+h+s}-\rmw^{(m-1)\top}_{t+h})]\\
    &= \lim_{s \to 0}\frac 1 s \left( \kappa^{(m-1)}(h+s)-\kappa^{(m-1)}(h)\right) \\
    &= \kappa^{(m)}(h).
    \end{split}
\end{equation}
\end{proofclaim}

    Fix $m\geq 0$. We prove \eqref{eq: autocov n m} by induction on $n$. The case $n=0$ is \eqref{eq: autocov 0 m}. We show the statement for $n\geq 1$ assuming it holds for $n-1$:
    \begin{equation}
    \label{eq: permut deriv integral 2}
    \begin{split}
    \E[\rmw^{(n)}_t \rmw^{(m)\top}_{t+h}]&= \lim_{s \to 0}\frac 1 s\E[(\rmw^{(n-1)}_{t+s}-\rmw^{(n-1)}_t)\rmw^{(m)\top}_{t+h}]\\
    &= \lim_{s \to 0}\frac 1 s \left( \E[\rmw^{(n-1)}_{t+s} \rmw^{(m)\top}_{t+h}]- \E[\rmw^{(n-1)}_{t} \rmw^{(m)\top}_{t+h}]\right)\\ %intermediate step
    &= (-1)^{n-1}\lim_{s \to 0}\frac 1 s \left( \kappa^{(n-1+m)}(h-s)-\kappa^{(n-1+m)}(h)\right) \\
   % &=(-1)^{n}\lim_{s \to 0}\frac 1 s \left( \kappa^{(n-1+m)}(h+s)-\kappa^{(n-1+m)}(h)\right) \\ %intermediate step
    &=(-1)^{n}\kappa^{(n+m)}(h).
    \end{split}
\end{equation}
\end{proof}
\eqref{eq: autocov n m} can be evaluated for any $2N-2$ times differentiable autocovariance function, and any $N-1$ times continuously differentiable noise $w$. Furthermore, Lemma~\ref{lemma: autocov of generalised fluctuations} shows that $\mathbf w$ is wide-sense stationary if only if $w$ is wide-sense stationary. Therefore, the autocovariance of $\mathbf w$ is
\begin{equation}
\label{eq: autocovariance generalised fluctuations}
    \E[\mathbf w_t \mathbf w_{t+h}^\top] = \begin{bmatrix} \kappa(h) & \kappa^{(1)}(h) & \kappa^{(2)}(h) & \hdots \\
    -\kappa^{(1)}(h) & -\kappa^{(2)}(h) & -\kappa^{(3)}(h) &\hdots \\
    \kappa^{(2)}(h) & \kappa^{(3)}(h) & \kappa^{(4)}(h) &\hdots \\
    \vdots & \vdots & \vdots & \ddots
    \end{bmatrix}\in \Mat_{Nd,Nd}.
\end{equation}

The covariance of generalised fluctuations has a checkerboard structure, where odd temporal derivatives of $\kappa$ in \eqref{eq: autocovariance generalised fluctuations} are replaced by zero submatrices. To see this, first note that the autocovariance of a wide-sense stationary process $h \mapsto \E[\rmw^{(n)}_t \rmw^{(n)\top}_{t+h}]$ attains a local maximum at $h=0$. %\LD{is this true?}
By Lemma~\ref{lemma: autocov of generalised fluctuations}, this implies that $\kappa^{(2n)}(h)$ attains a local extremum at $h=0$. Thus, $\kappa^{(2n+1)}(0)=0$. We can find this result in another way: namely, by observing that the covariance of generalised fluctuations is symmetric. From \eqref{eq: autocovariance generalised fluctuations},

\begin{equation}
\label{eq: covariance gen fluctuations}
    \E[\mathbf w_t \mathbf w_{t}^\top] = \begin{bmatrix} \kappa(0) & 0 & \kappa^{(2)}(0) & \hdots \\
    0 & -\kappa^{(2)}(0) & 0 &\hdots \\
    \kappa^{(2)}(0) & 0 & \kappa^{(4)}(0) &\hdots \\
    \vdots & \vdots & \vdots & \ddots
    \end{bmatrix} \in \Mat_{Nd,Nd}.
\end{equation}

\begin{example}[Gaussian autocovariance]
\label{eg: Gaussian autocovariance}
Consider the prototypical example of a Gaussian autocovariance $ \kappa(h)=(2\sqrt\pi\sigma)^{-1} e^{- \frac{h^2}{4\sigma^2}}$ \eqref{eq: autocov Gaussian convolved white noise} entailed by
Gaussian convolved white noise \eqref{eq: Gaussian convolved white noise}. By inspection,
\begin{equation*}
    \kappa^{(2n)}(0)=\frac{(-1)^n(2n-1)!!}{2^{n+1}\sqrt\pi \sigma^{2n+1}} , \quad n\in \N,
\end{equation*}
where $!!$ denotes the double factorial. %This gives the generalised covariance matrix with entries:
%\begin{equation*}
%    \tilde \Sigma_{nn}=
%\end{equation*}
In particular, the diagonal of the generalised covariance matrix \eqref{eq: covariance gen fluctuations} reads:
\begin{equation*} %found by inspection
    \frac{(2n-1)!!}{2^{n+1}\sqrt\pi \sigma^{2n+1}} , \quad n=0,1,2,\ldots, N-1.
\end{equation*}
The white noise limit corresponds to the case $\sigma\searrow 0$, whence orders of motion higher than $n=0$ can be discarded with impunity. %Empirically, we find that the precision of generalised fluctuations becomes
%\begin{equation} %found by simulation
%    \E[\tilde w_t \tilde w_{t}]^{-1} \xrightarrow{\beta\nearrow+\infty} \begin{bmatrix} ? & 0 & 0 & \hdots \\
%    0 & 0 & 0 &\hdots \\
%    0 & 0 & 0 &\hdots \\
%    \vdots & \vdots & \vdots & \ddots
%    \end{bmatrix}
%\end{equation}
%Here $?$ is an unknown positive real number which is independent of $\beta$, but which depends upon the order of truncation of the matrix (i.e.,  number of orders of motion). By independent of $\beta$ we mean that $\E[\tilde w_t \tilde w_{t}]^{-1}_{11}$ is independent of $\beta$. NOTE: this is only true for the unnormalised kernel but the normalisation depends on beta, when we add it back in, it is not true anymore.
\end{example}

\begin{example}[Square rational autocovariance]%what karl calls 1/f noise
\label{eg: Square rational autocovariance}
Suppose $\kappa(h)=\frac{1}{1+h^2}$. The Taylor series at the origin reads $\kappa(h)=\sum_{n=0}^\infty (-1)^nh^{2n}$ and converges for $|h|<1$. In particular, $\kappa^{(2n)}(0)=(-1)^n(2n)!$ and the diagonal of the matrix \eqref{eq: covariance gen fluctuations} reads $(2n)!$ for $n=0,1,2,\ldots, N-1$.
\end{example}

\subsection{Gaussian process noise}\label{sec: GP noise}

We consider the preceding section in the context where the driving noise process $w$ is a Gaussian process (GP). In this case, the wide-sense stationarity assumption is equivalent to stationarity; we are thus in the setting of stationary GPs with $(N-1)$-times continuously differentiable sample paths a.s.. It turns out that we can make the calculations of the preceding Section~\ref{sec: second order statistics} rigorous in this special setup.

\subsubsection{Conditions for sample differentiability}

First we must answer: what are the conditions on stationary GPs for the many times differentiability of their sample paths? In practice, Gaussian processes are often described and identified by their autocovariance (a.k.a. kernel) $\kappa$. Assuming zero mean, the regularity of the autocovariance determines the regularity of the GP sample paths almost surely \cite{dacostaSamplePathRegularity2024}. Consider the following assumption:

\begin{assumption}[Regular GP autocovariance]\label{as: lipschitz differentiable autocov}
 %$w: \T\times \Omega \to \R^d$
$w$ is a mean-zero stationary Gaussian process with autocovariance $\kappa$ such that $\kappa^{(2(N-1))} $ exists and is locally Lipschitz continuous.
\end{assumption}

Note that, by the mean value theorem, the existence of one additional continuous derivative of $\kappa$ is sufficient for local Lipschitz continuity of the lower derivatives, and hence for Assumption~\ref{as: lipschitz differentiable autocov}. We have \cite[Corollary 1]{dacostaSamplePathRegularity2024}:

\begin{proposition}[Sample differentiability of GPs]
\label{prop: Sample differentiability of GPs}
Under Assumption~\ref{as: lipschitz differentiable autocov}, $w$ has $(N-1)$-times continuously differentiable sample paths almost surely. In other words, Assumption~\ref{as: lipschitz differentiable autocov} implies Assumption~\ref{as: main assumption}.
\end{proposition}

Proposition~\ref{prop: Sample differentiability of GPs} is sharper than alternative results in the existing literature, that, unlike this, derive the regularity conditions on the autocovariance via sample path regularity in Sobolev spaces and the Sobolev embedding theorems, which requires more derivatives \cite{scheuerer_regularity_2010}. 

In addition Assumption~\ref{as: lipschitz differentiable autocov} implies a rigorous proof of Lemma~\ref{lemma: autocov of generalised fluctuations} (see Remark~\ref{rmk: formal calculation}) following the proof of the forward implication of \cite[Theorem 7]{dacostaSamplePathRegularity2024}. So:
\begin{proposition}
\label{prop: autocov generalised GP}
    Under Assumption~\ref{as: lipschitz differentiable autocov}, 
    \begin{equation}
    \E[\rmw^{(n)}_t \rmw^{(m)\top}_{t+h}]= (-1)^n\kappa^{(n+m)}(h).
\end{equation}
\end{proposition}
It follows that the calculations of the preceding section~\ref{sec: second order statistics} are made rigorous under Assumption~\ref{as: lipschitz differentiable autocov}.

\subsubsection{Gaussianity of generalised fluctuations}

Now that we have established sufficient conditions on the autocovariance $\kappa$ for many times differentiability of the sample paths, we look at the structure of the generalised GP fluctuations $\mathbf w$.

% For this we need a lemma:

\begin{proposition}%multivariate smooth Gaussian process. Serial derivatives is also a Gaussian process
\label{prop: serial derivatives of GP are GP}
    Consider a mean-zero $\R^d$-valued Gaussian process $w$ with almost surely $(N-1)$-times continuously differentiable sample paths. Then, $\mathbf w$
    %(w_t^{(0)}, w_t^{(1)}, \ldots, w_t^{(N)})_{t\in \T}$
    is an $\R^{dN}$-valued Gaussian process. %with autocovariance \eqref{eq: autocovariance generalised fluctuations}.% in the sense that, for any $t_0,\ldots t_m\in \R$, $\left(w_{t_j}^{(i)}\right)_{1\leq j\leq N}^{1\leq i \leq n}$ is joint Gaussian.
\end{proposition}

\begin{proof}
        Consider the process $\rmw^{(0:n)}\triangleq (\rmw^{(0)}, \ldots, \rmw^{(n)})$ for any $0\leq n\leq N-1$: we show that this is an $\R^{dn}$-valued Gaussian process by induction on $n$. The case $n=0$ holds by assumption. Suppose that the statement holds for $n=k$ for some arbitrary $k <N-1$. We show it for $n=k+1$. Consider some arbitrary finite set of times: $m\in \N, t_0,\ldots t_m\in \T$. For any $\e>0$
    \begin{align}
    \label{eq: induction derivatives GP}
       &\left(\rmw_{t_0}^{(0)},\ldots, \rmw_{t_m}^{(0)},\ldots, \rmw_{t_0}^{(k)},\ldots, \rmw_{t_m}^{(k)}, \frac{\rmw_{t_0+\e}^{(k)}-\rmw_{t_0}^{(k)}}{\e},\ldots, \frac{\rmw_{t_m+\e}^{(k)}-\rmw_{t_m}^{(k)}}{\e}\right) %\\
       %&\xrightarrow{\e \to 0} \left(w_{t_0}^{(0)},\ldots, w_{t_m}^{(0)},\ldots, w_{t_0}^{(k)},\ldots, w_{t_{m}}^{(k)}, w_{t_0}^{(k+1)},\ldots, w_{t_m}^{(k+1)}\right) %\text{ a.s.}
    \end{align}
    %The first line of \eqref{eq: induction derivatives GP}
    is joint Gaussian. Indeed, \eqref{eq: induction derivatives GP} is a linear transformation of the random variable
$\left(\rmw_{t_0}^{(0)},\ldots, \rmw_{t_m}^{(0)},\ldots, \rmw_{t_0}^{(k)},\ldots, \rmw_{t_m}^{(k)}\right)$, which is multivariate Gaussian by our induction hypothesis. \eqref{eq: induction derivatives GP} converges almost surely to 
    \begin{align}
    \label{eq: induction derivatives GP 2}
        \left(\rmw_{t_0}^{(0)},\ldots, \rmw_{t_m}^{(0)},\ldots, \rmw_{t_0}^{(k)},\ldots, \rmw_{t_{m}}^{(k)}, \rmw_{t_0}^{(k+1)},\ldots, \rmw_{t_{m}}^{(k+1)}\right)
    \end{align}
as $\e \to 0$. Recall that almost sure convergence implies convergence in distribution (a.k.a. weak convergence). 
%By Lemma~\ref{lemma: gaussian converges to a gaussian}
Therefore, we deduce that \eqref{eq: induction derivatives GP 2} is joint Gaussian \cite{mse1,mse2}. Since $t_0,\ldots t_m$ are arbitrary, $\mathbf w$ is a Gaussian process. %The autocovariance was already computed in \eqref{eq: autocovariance generalised fluctuations}.
\end{proof}

It follows that the generalised GP fluctuations $\mathbf w_t$ have a simple Gaussian form that is constant in time:

\begin{theorem}
    Under Assumption~\ref{as: lipschitz differentiable autocov}, the generalised GP fluctuations have the distribution
    \begin{equation}
        \mathbf w_t\sim\mathcal N(0, \mathbf \Sigma), \:\forall t\in \T,
    \end{equation}
    with covariance $\mathbf \Sigma$ given by \eqref{eq: covariance gen fluctuations}.
\end{theorem}

\begin{proof}
Under Assumption~\ref{as: lipschitz differentiable autocov}, Proposition~\ref{prop: autocov generalised GP} justifies all the calculations of Section~\ref{sec: second order statistics}, and in particular the covariance \eqref{eq: covariance gen fluctuations} and zero-mean. Furthermore, Assumption~\ref{as: lipschitz differentiable autocov} implies Proposition~\ref{prop: serial derivatives of GP are GP}, which implies that $ \mathbf w_t$ is a multivariate Gaussian random variable. %By Assumption~\ref{as: lipschitz differentiable autocov} finally, the generalised noise $ \mathbf w_t$ is Since the generanoise 
\end{proof}

\section{Analysis of the linear equation}

In this section, we analyse the linear SDE driven by a stationary Gaussian process $w$ using generalised coordinates (i.e.~\eqref{eq: SDE} with linear flow $f$)
\begin{equation}
\label{eq: linear SDE}
   \frac{d}{dt}x_t= Ax + %\sigma
   w_t, \quad A\in\Mat_{d,d},\quad   x_0=z \in \R^d.
\end{equation}

\subsection{Many times differentiable fluctuations}

We operate under sufficient regularity on the noise autocovariance---i.e.~Assumption~\ref{as: lipschitz differentiable autocov}---which implies a.s. many times differentiability of sample paths: i.e.~Assumption~\ref{as: main assumption}, by Proposition~\ref{prop: Sample differentiability of GPs}.

The algebraic equation for relating the various derivatives of \eqref{eq: linear SDE} at $t=0$ equals \eqref{eq: gen Cauchy problem} 
\begin{equation}
    \begin{split}
        \rmx_0^{(1)} &= A\rmx_0^{(0)}+\rmw_0^{(0)}
        \\
        \rmx_0^{(2)} &= A\rmx_0^{(1)}+\rmw_0^{(1)}\\
                        &\vdots \\
        \rmx_0^{(N)} &= A\rmx_0^{(N-1)}+\rmw_0^{(N-1)}\\   
        % &\vdots
    \end{split}
\end{equation}
By substitution the generalised coordinates at $t=0$ equal
\begin{equation}\label{eq: recursion linear eq}
    \rmx_0^{(n)}=A^n \rmx_0^{(0)}+\sum_{k=0}^{n-1} A^{n-1-k} \rmw_0^{(k)}, \quad \forall 1\leq n \leq N.
\end{equation}
The deterministic initial condition $\rmx_0^{(0)}=z$ yields that the vector of generalised coordinates $\mathbf x_0$ is joint Gaussian (as an affine transform of the multivariate Gaussian random variable $\mathbf w_0$).

Thus the zero-th order generalised coordinate process
% The solution at time $t$ can then be expressed via a Taylor expansion or Taylor series, cf.~\eqref{eq: Markovian realisation}
\begin{equation}
\label{eq: taylor approximation linear eq}
    % x^{(0)}_t =
   \rmx_t^{(0)} = \sum_{n=0}^N \rmx^{(n)}_0 \frac{t^n}{n!},
\end{equation}
which by construction approximates the solution to \eqref{eq: linear SDE} is a Gaussian process.

\begin{proposition}
\label{prop: GP finite Taylor approx}
    Under Assumption~\ref{as: lipschitz differentiable autocov}, the Gaussian process \eqref{eq: taylor approximation linear eq} has mean and autocovariance
    \begin{equation}\label{eq: mean and cov linear eq}
    \begin{split}
        \E\left[\rmx_t^{(0)}\right] & = \sum_{n=0}^{N}A^nx_0\frac{t^n}{n!}, \quad
        \operatorname{cov}(\rmx_t^{(0)},\rmx_s^{(0)}) = \sum_{n,m=1}^{N} \frac{t^ns^m}{n!m!}\sum_{k=0}^{n-1}\sum_{l=0}^{m-1}A^{n-1-k}(-1)^k\kappa^{(k+l)}(0)(A^{m-1-l})^\top.
    \end{split}
\end{equation}
\end{proposition}
\begin{proof}
Using \eqref{eq: recursion linear eq} at $t=0$ we have
\begin{equation}
    \E[\rmx_t^{(0)}] = \E\left[\sum_{n=0}^{N}\rmx_0^{(n)}\frac{t^n}{n!}\right] = \E\left[\sum_{n=0}^{N}\left(A^nx_0+\sum_{k=0}^{n-1}A^{n-1-k}\rmw_0^{(k)}\right)\frac{t^n}{n!}\right] = \sum_{n=0}^{N}A^nx_0\frac{t^n}{n!}.
\end{equation}
Furthermore,
\begin{equation}
    \begin{split}
        \operatorname{cov}(\rmx_t^{(0)},\rmx_s^{(0)}) &=\E\left[\left(\sum_{n=0}^{N}\rmx_0^{(n)}\frac{t^n}{n!}-\sum_{n=0}^{N}A^nx_0\frac{t^n}{n!}\right)\left(\sum_{n=0}^{N}\rmx_0^{(n)}\frac{s^n}{n!}-\sum_{n=0}^{N}A^nx_0\frac{t^n}{n!}\right)^\top\right] \\
&=\E\left[\sum_{n=1}^{N}\frac{t^n}{n!}\sum_{k=0}^{n-1}A^{n-1-k}\rmw_0^{(k)}\left(\sum_{n=1}^{N}\frac{s^n}{n!}\sum_{k=0}^{n-1}A^{n-1-k}\rmw_0^{(k)}\right)^\top\right] \\
&= \sum_{n,m=1}^{N} \frac{t^ns^m}{n!m!}\sum_{k=0}^{n-1}\sum_{l=0}^{m-1}A^{n-1-k}\E\left[\rmw_0^{(k)}\rmw_0^{(l)\top}\right](A^{m-1-l})^\top \\
&= \sum_{n,m=1}^{N} \frac{t^ns^m}{n!m!}\sum_{k=0}^{n-1}\sum_{l=0}^{m-1}A^{n-1-k}(-1)^k\kappa^{(k+l)}(0)(A^{m-1-l})^\top
    \end{split}
\end{equation}
where we used Proposition~\ref{prop: autocov generalised GP} for the last line.
\end{proof}

\subsection{Smooth and analytic fluctuations}
\label{sec: smooth and analytic fluctuations}

Now we make an assumption that is strictly stronger than Assumption~\ref{as: lipschitz differentiable autocov}.
\begin{assumption}
\label{as: smooth sample paths}
%$w: \T\times \Omega \to \R^d$
$w$ is a mean-zero stationary Gaussian process with smooth autocovariance $\kappa$.
\end{assumption}

By \cite[Corollary 13]{dacostaSamplePathRegularity2024}, Assumption~\ref{as: smooth sample paths} is equivalent to $w$ having smooth sample paths, a stronger statement than Proposition~\ref{prop: Sample differentiability of GPs} which solely yields the forward implication. We make this assumption so that we can consider the zero-th order coordinate process $\rmx^{(0)}$ \eqref{eq: taylor approximation linear eq} as $N\to \infty$.

\begin{proposition}\label{lemma: mean convergence}
Under Assumption~\ref{as: smooth sample paths}, the mean of the zero-th coordinate process $\rmx^{(0)}$ \eqref{eq: taylor approximation linear eq} has the limit
    \begin{equation}
    \label{eq: limit mean}
    \E\left[\rmx_t^{(0)}\right] \xrightarrow{N\to \infty} e^{At}x_0.
    \end{equation}
\end{proposition}
Proposition~\ref{lemma: mean convergence} is obvious from Proposition~\ref{prop: GP finite Taylor approx}. Furthermore, this is exactly what we expect as this coincides with the mean of the linear diffusion process \cite[eq. 2.7]{godrecheCharacterisingNonequilibriumStationary2019}.

Now we consider the limiting behaviour of the autocovariance. For the limit to even exist we need a stronger assumption than~\ref{as: smooth sample paths}:

\begin{assumption}
\label{as: analytic autocov}
    The autocovariance $\kappa$ of the noise process $w$ in \eqref{eq: linear SDE} is analytic around $0$ with radius of convergence $2R$.
\end{assumption}
From this assumption, we obtain the limit of the autocovariance of the zero-th order process:

\begin{proposition}\label{lemma: cov convergence}
Under Assumption~\ref{as: analytic autocov},
    \begin{equation}
    \label{eq: limit autocovariance}
    \operatorname{cov}(\rmx_t^{(0)},\rmx_s^{(0)}) \xrightarrow{N\to \infty} \sum_{n,m=1}^{\infty} \frac{t^ns^m}{n!m!}\sum_{k=0}^{n-1}\sum_{l=0}^{m-1}A^{n-1-k}(-1)^k\kappa^{(k+l)}(0)(A^{m-1-l})^\top
    \end{equation}
where the convergence is locally uniform in $t,s\in (-R/\lambda,R/\lambda)$, where $\lambda\triangleq  \max(1,\|A\|_\infty, \|A^\top\|_\infty)$ and where $\|\cdot\|_\infty$ is the matrix operator norm induced by the vector $l^\infty$ norm, which equals the maximum absolute row sum $\|A\|_\infty\triangleq  \max_{1\leq i\leq d} \sum_{j=1}^d |A_{ij}|.$
\end{proposition}

\begin{remark}
    We have $R/\lambda \geq R/(\alpha d)$ where $\alpha \triangleq  \max(\{1\}\cup \{|A_{ij}|\}_{1\leq i,j\leq d})$.
\end{remark}

\begin{proof}
$\|\cdot\|_\infty$ is sub-multiplicative, so for $0\leq k\leq n-1$, $0\leq l\leq m-1$,
$$
\begin{aligned}
\left\| A^{n-1-k}(-1)^k\kappa^{(k+l)}(0)(A^{m-1-l})^T\right\|_{\infty}&\leq 
\|A\|_{\infty}^{n-1-k}\|\kappa^{(k+l)}(0)\|_{\infty}\|A^\top\|_\infty^{m-1-l} \\
&\leq \lambda^{n+m}\kappa_{n+m}.
\end{aligned}
$$
where $\kappa_n \triangleq  \max_{0\leq i\leq n} \|\kappa^{(i)}(0)\|_{\infty}$.
From \eqref{eq: mean and cov linear eq}, we have the following upper bound on the autocovariance 
\begin{equation}\label{eq:kN_bound}
|\operatorname{cov}(\rmx_t^{(0)},\rmx_s^{(0)})| \leq
\sum_{n,m=1}^N \frac{|t|^n|s|^m}{n!m!}\left\|\sum_{k=0}^{n-1}\sum_{l=0}^{m-1}A^{n-1-k}(-1)^k\kappa^{(k+l)}(0)(A^{m-1-l})^\top\right\|_\infty \leq \sum_{n,m=1}^N |t|^n|s|^m\lambda^{n+m} \frac{nm\kappa_{n+m}}{n!m!}.
\end{equation}
On the other hand, note that for $-R< t,s< R$, we have by Assumption~\ref{as: analytic autocov}
\begin{equation}\label{eq:k_series}
\E[\rmw^{(0)}_t\rmw^{(0)\top}_s] = \sum_{n,m=0}^\infty t^ns^m\frac{(-1)^m\kappa^{(n+m)}(0)}{n!m!}.
\end{equation}
By the Cauchy-Hadamard theorem applied to \eqref{eq:k_series}, %\LD{note for self: in the Gaussian kernel case, where $r=+\infty$, the radius remains the same as we change $\beta$, and in the rough limit: $\lim_{\beta \to +\infty}r(\beta)=+\infty$, }
\begin{equation}
\label{eq: a lower bound}
\frac{1}{R} = \limsup_{\substack{n+m\to\infty \\ n,m\geq 0}} \left(\frac{\|\kappa^{(n+m)}(0)\|_{\infty}}{n!m!}\right)^{\frac{1}{n+m}}\geq \limsup_{\substack{n+m\to\infty \\ n,m\geq 1}} \left(\frac{\|\kappa^{(n+m)}(0)\|_{\infty}}{n!m!}\right)^{\frac{1}{n+m}}=\limsup_{\substack{n+m\to\infty \\ n,m\geq 1}} \left(\frac{\kappa_{n+m}}{n!m!}\right)^{\frac{1}{n+m}},
\end{equation}
where the last equality holds by a technical Lemma~\ref{lemma: limsup equality} derived in Appendix~\ref{app: technical}. Note that for $n,m\geq 1$ we have $n^{\frac{1}{n+m}}, m^{\frac{1}{n+m}} \to 1$ as $n+m\to\infty$. We apply this to \eqref{eq: a lower bound} to obtain
\begin{equation}
\label{eq: another lower bound}
\limsup_{\substack{n+m\to\infty \\ n,m\geq 1}} \left(\frac{nm\kappa_{n+m}}{n!m!}\right)^{\frac{1}{n+m}} \leq \frac{1}{R}.
\end{equation}
Applying the Cauchy-Hadamard theorem and \eqref{eq: another lower bound}
 to \eqref{eq:kN_bound}, yields that the majorising series for $|\operatorname{cov}(\rmx_t^{(0)},\rmx_s^{(0)})|$ converges absolutely for $-R/\lambda< t,s< R/\lambda$, and hence $\operatorname{cov}(\rmx_t^{(0)},\rmx_s^{(0)})$ converges as $N\to \infty$, locally uniformly within this radius, and the limit is analytic as it can be written as a power series.
\end{proof}

Now that we know the mean and autocovariance of the zero-th coordinate process \eqref{eq: taylor approximation linear eq} as $N\to \infty$, we just need to verify that the limiting process is a Gaussian process. It turns out that pointwise convergence of the autocovariance ensures weak convergence of the finite dimensional distributions of the zero-th coordinate process $\rmx^{(0)}_t$. Local uniform convergence of the autocovariance however, ensures the weak convergence in the sense of Gaussian processes, which is a stronger property.

\begin{lemma}\label{lemma: weak convergence}
Let $(g_N)_{N\geq 1}$ be a sequence of Gaussian processes $g_N:I\times \Omega \to \R^d$, with $I\subset \R$ an interval, with almost surely continuous sample paths, with means $(\mu_N)_{N\geq 1}$ and autocovariances $(k_N)_{N\geq 1}$. If
\begin{enumerate}
    \item $(\mu_N)_{N\geq 1}$ converges pointwise,
    \item $(k_N)_{N\geq 1}$ converges locally uniformly to a continuous limit,
\end{enumerate}
then $(g_N)_{N\geq 1}$ converges weakly in $C(I,\R^d)$ to a Gaussian process with mean $\lim_{N\to\infty} \mu_N$ and kernel $\lim_{N\to\infty} k_N$.
\end{lemma}
\begin{proof}
1. and 2. imply that the finite dimensional distributions of the $g_N$ converge. So if $(g_N)_{N\geq 1}$ converges weakly, it necessarily does so to the Gaussian process with mean $\lim_{N\to\infty} \mu_N$ and kernel $\lim_{N\to\infty} k_N$. First suppose $d=1$. Let $(t_N)_{N\geq 1} \subset\T$ be a bounded sequence of times, and $(h_N)_{N\geq 0}$ be a sequence of positive numbers tending to 0. Then
$$
\begin{aligned}
\E[ |g_N(t_N+h_N)-g_N(t_N)|^2] &= \E[g_N(t_N+h_N)^2]+\E[g_N(t_N)^2]-2\E(g_N(t_N+h_N)g_N(t_N)] \\
&=  k_N(t_N+h_N,t_N+h_N) + k_N(t_N,t_N) -2k_N(t_N+h_N,t_N) \\
&= |k_N(t_N+h_N,t_N+h_N)-k_\infty(t_N+h_N,t_N+h_N)| \\
&\;\;+|k_N(t_N,t_N)-k_\infty(t_N,t_N)| \\
&\;\;+|2k_N(t_N+h_N,t_N) -2k_\infty(t_N+h_N,t_N)| \\
&\;\;+|k_\infty(t_N+h_N,t_N+h_N) + k_\infty(t_N,t_N) -2k_\infty(t_N+h_N,t_N)| \\
&\to 0 \text{ as } N\to\infty
\end{aligned}$$
by uniform convergence of $(k_N)_{N\geq 1}$ and uniform continuity of $k_\infty\triangleq  \lim_{N\to\infty}k_N$ on $[\inf_{N\geq 1} (t_N),\sup_{N\geq 1} (t_N+h_N)]\times[\inf_{N\geq 1} (t_N),\sup_{N\geq 1} (t_N+h_N)]$. So $|g_N(t_N+h_N)-g_N(t_N)|$ tends to 0 in $L^2$, and in particular it converges in probability. Then \cite[Theorem 14.5, 14.6 \& 14.11]{kallenbergFoundationsModernProbability2021} implies that $(g_N)_{N\geq 1}$ converges weakly in $C(\T,\R)$, where we also used \cite[Lemma 14.1]{kallenbergFoundationsModernProbability2021} to view $g_N$ as random elements of $C(\mathbb T,\R)$. \par
For arbitrary $d \in \N$, note that $\E[\|g_N(t_N+h_N)-g_N(t_N)\|^2]= \sum_{i=1}^d\E[|g_N(t_N+h_N)_i-g_N(t_N)_i|^2]$, % where $\cdot_i$ denotes the $i$-th component of $g_N$, 
so the same proof strategy applies.
\end{proof}

\begin{proposition}
\label{prop: convergence of gen coord linear case}
Under Assumption~\ref{as: analytic autocov},
the generalised coordinate $\rmx^{(0)}$ converges weakly in $C((-R/\lambda,R/\lambda),\R^d)$, where $\lambda\triangleq  \max(1,\|A\|_\infty, \|A^\top\|_\infty)$, to a Gaussian process with mean and autocovariance given by \eqref{eq: limit mean} and \eqref{eq: limit autocovariance}.
\end{proposition}
\begin{proof}
    Follows from Lemma~\ref{lemma: weak convergence}, noting that conditions 1.~and 2.~have been shown to hold in Proposition~\ref{lemma: mean convergence} and Proposition~\ref{lemma: cov convergence} respectively.
\end{proof}

If, in addition, we assume that:
\begin{assumption}
\label{as: analytic sample paths}
    The sample paths of $w$ are analytic around $0$, almost surely.
\end{assumption}
It is entirely possible this follows from analyticity of the autocovariance, i.e.~Assumption~\ref{as: analytic autocov}; where this implies a.s. analyticity of the noise sample paths on the radius of convergence  $(-2R,2R)$; however, we have not found any results on the matter in the literature.

\begin{theorem}[Solution to the linear SDE]
\label{thm: solution linear SDE}
    Under Assumptions~\ref{as: analytic autocov} and~\ref{as: analytic sample paths}, the unique solution to the linear SDE \eqref{eq: linear SDE} on $\T= (-R/\lambda, R/\lambda)$, where $\lambda\triangleq  \max(1,\|A\|_\infty, \|A^\top\|_\infty)$, is a Gaussian process with mean \eqref{eq: limit mean} and autocovariance \eqref{eq: limit autocovariance}. 
\end{theorem}

\begin{proof}
We showed in \eqref{eq: analytic case accuracy} that 
        under the conditions of Theorem~\ref{thm: Cauchy–Kovalevskaya}, i.e.~Assumption \eqref{as: analytic sample paths}, $N=\infty$ implies that $\rmx^{(0)}$ coincides with the unique SDE solution, and we showed in \cref{prop: convergence of gen coord linear case} that when $N=\infty$, $\rmx^{(0)}$ is a Gaussian process with said statistics.
        % \item So what gives the conditions of the Theorem~\ref{thm: Cauchy–Kovalevskaya}? $w$ needs to be analytic. And it needs to be analytic on what time interval exactly?
\end{proof}

As a corollary of Theorem~\ref{thm: solution linear SDE}, the linear SDE driven by Gaussian convolved white noise \eqref{eq: Gaussian convolved white noise} has a Gaussian process solution defined for all time, which has a computable mean and autocovariance, where we insert the derivatives of the Gaussian kernel computed in \cref{eg: Gaussian autocovariance} into the expressions \eqref{eq: limit mean} and \eqref{eq: limit autocovariance}.

\section{Density dynamics via the Fokker-Planck equation}

% Here we look at density dynamics of the generalised coordinate process $\mathbf x_t$ \eqref{eq: gen Cauchy problem} through the Fokker-Planck equation.

How does the probability density of the generalised coordinate process $\mathbf x_t \sim \rho_t$ evolve over time? Dynamics in generalised coordinates are deterministic $ \dot{\mathbf{x}} = \mathbf{D}\mathbf{x}$ given a random initial condition, which determines the initial probability density $\rho_0$, i.e. \eqref{eq: gen Cauchy problem}. The time evolution of the density is given by the Fokker-Planck equation \cite{riskenFokkerPlanckEquationMethods1996,pavliotisStochasticProcessesApplications2014, Soize1994}, which in the absence of diffusion coefficient reduces to a Liouville equation. The following computations are formal when $N=\infty$. The Fokker-Planck equation reads:
\begin{equation}
\label{eq: FP equation}
    \begin{split}
        \partial_t \rho_t(\mathbf z) %& =-\nabla \cdot(f(\tilde{x}) \rho(\tilde{x})) \\
& =-\nabla \cdot(\mathbf D \mathbf z \rho_t(\mathbf z))\\
&= -\nabla \rho_t(\mathbf z) \cdot  \mathbf D \mathbf z +\rho_t(\mathbf z)\underbrace{\nabla \cdot( \mathbf D\mathbf z)}_{=0}\\
&= -\nabla \rho_t(\mathbf z) \cdot  \mathbf D \mathbf z 
    \end{split}
\end{equation}
where $\mathbf z\triangleq\rmz^{(:N)} \in \G^{d,N}$ is a generalised state. To show the equality in underbrace in \eqref{eq: FP equation} we observed that:
\begin{equation}
    \nabla \cdot( \mathbf D\mathbf z)=\sum_{n=0}^{N} \nabla_{\rmz^{(n)}}\cdot (\mathbf D \mathbf z)^{(n)} =\sum_{n=0}^{N-1} \nabla_{\rmz^{(n)}}\cdot \rmz^{(n+1)} + \nabla_{\rmz^{(N)}}\cdot 0=0.
\end{equation}

The first two moments of the initial density $\rho_0$ can be obtained analytically when the underlying SDE is linear by following Propositions~\ref{prop: GP finite Taylor approx} (finite $N$) and~\ref{lemma: cov convergence} (infinite $N$). When the underlying noise process is a Gaussian process, the initial density is a Gaussian determined by these two moments.
The Fokker-Planck equation \eqref{eq: FP equation} preserves Gaussianity of densities: the generalised process solves \eqref{eq: equation generalised process} which yields
\begin{equation}
   \rho_0=\mathcal N(\bmu ,\mathbf \Xi) \Rightarrow \rho_t=\mathcal N\left(e^{\mathbf D t} \bmu, e^{\mathbf D t}  \mathbf \Xi e^{\mathbf D^\top t} \right).
\end{equation}
When the underlying SDE is non-linear, we can still efficiently sample from the initial density---we will show how in \cref{sec: num int}.

The stationary Fokker-Planck equation reads
\begin{align}
 \nabla \rho(\mathbf z) \cdot \mathbf D \mathbf z=0.
\end{align}

To understand the (stationary) Fokker-Planck equation, it might help to know the spectrum of $ \mathbf D$. By inspection, its unique eigenvalue is $0$ with eigenvectors $\vec e_1,\ldots,\vec e_d$ of multiplicity $N+1$.

\section{Path integral formulation}

In statistical, quantum and classical mechanics, the path integral formulation of dynamics expresses the negative log probability of a path, called the \textit{action}, in terms of  the so-called \textit{Lagrangian}, which is a function of states \cite{ seifertStochasticThermodynamicsFluctuation2012, chaichianPathIntegralsPhysics2001, chaichianPathIntegralsPhysics2001a}. In generalised coordinates, the path integral formulation takes a very simple form, where the action %$\A$ 
and the Lagrangian %$\L$ 
coincide up to a simple transformation. %Precisely:

\subsection{Points and paths}

Indeed, in generalised coordinates, states embed injectively and canonically into paths 

\begin{equation}
\label{eq: embedding states paths}
\begin{split}
    \eta: \G^{d,N} &\to C(\T, \G^{d,N})\\
    \mathbf z&\mapsto\left(  \dot{\mathbf x}_t= \mathbf D \mathbf x_t, \: \mathbf x_0=\mathbf z\right)\\&=\left( t \mapsto  \exp( t\mathbf D)\mathbf z\right).
\end{split}
\end{equation}

\subsection{Lagrangian and action}

Recall the generalised state solution of our Cauchy problem \eqref{eq: gen Cauchy problem} at the initial time point solves the algebraic equation:
\begin{equation}
\label{eq: alg eq}
\begin{split}
\mathbf D' \mathbf x_0 =\mathbf f(\mathbf x_0)+ \mathbf w_0.
\end{split}
\end{equation}

% \begin{equation*}
%     \mathbf D_N  x^{(0:N)}_t=\mathbf f(x^{(0:N)}_t)+ \tilde w_t.
% \end{equation*}

We know that for a stationary Gaussian process $w$, we have $\mathbf w_0 \sim \mathcal N(0, \mathbf \Sigma)$. That is, up to a time-independent additive constant:

\begin{equation}
-\log p(\mathbf w_0)\doteq\mathbf w_0^{\top} (2\mathbf \Sigma)^{-1}\mathbf w_0.
\end{equation}

We thus define the Lagrangian, which expresses the negative log probability of a generalised state (up to an additive constant) as follows:

\begin{equation}
\mathcal L: \G^{d,N} \to \R, \quad    \mathcal L( \mathbf x_0)\triangleq\left(\mathbf D'  \mathbf x_0-\mathbf f(\mathbf x_0)\right)^{\top}(2\mathbf \Sigma)^{-1}\left(\mathbf D'  \mathbf x_0-\mathbf f(\mathbf x_0)\right)\doteq -\log p(\mathbf x_0).
\end{equation}
%Note that the Lagrangian is time independent! 
The Lagrangian satisfies (assuming $\mathbf \Sigma$ is non-singular\footnote{This is the case in the usual examples, see~\ref{eg: Gaussian autocovariance} and~\ref{eg: Square rational autocovariance}.} for `$\Leftarrow$')
\begin{equation}
        \mathcal L(\mathbf x_0)\geq 0,\quad 
   \mathbf D'  \mathbf x_0=\mathbf f(\mathbf x_0)  \iff 
   \mathcal L(\mathbf x_0)=0\Rightarrow \nabla\mathcal L(\mathbf x_0)=0.
\end{equation}

It follows that we can define the action (the negative log probability of a path in generalised coordinates up to an additive constant) as
\begin{equation}
    \mathcal A%(x^{(:N+1)}) 
    \triangleq  \L \circ \eta^{-1}%(x^{(:N+1)})%\doteq -\log p(x^{(:N+1)})
\end{equation}
where $\eta$ is the canonical embedding of states into paths \eqref{eq: embedding states paths}.

\subsection{Path of least action}
%\LD{From here we have a redundancy between the bold notation for path of least action and the notation for generalised state. use a segway notation bvx for path of least action for now. will use vec x for generalised variables later.}

% The \textit{path of least action} is the most likely path
% \begin{equation}
% \label{eq: path of least action}
%     \mathbf x^{(:N+1)}\triangleq \arg\min \mathcal A
% \end{equation}
% which, via the embedding of states into paths, corresponds to the most likely generalised state:
% \begin{equation}
% \label{eq: state of least action}
% \begin{split}
%     &\mathbf x^{(:N+1)}= \eta(\mathbf x_0^{(:N+1)})\\
%     \mathbf x_0^{(:N+1)}\triangleq \arg\min \mathcal L
%     %&\Leftrightarrow \mathbf D_N'  \mathbf x^{(:N+1)}_0 =\mathbf f(\mathbf x^{(:N)}_0)\Rightarrow \nabla\mathcal L(x^{(:N+1)}_t)=0
%    \iff &\mathcal L(\mathbf x_0^{(:N+1)})=0 \Rightarrow \nabla\mathcal L(\mathbf x_0^{(:N+1)})=0.
% \end{split}
% \end{equation}
%Assuming the covariance of generalised fluctuations $\tilde \Sigma$ is non-singular, we also have the reciprocal implication '$\Rightarrow$' in \eqref{eq: state of least action}. We will see how to (approximately) recover the path of least action by minimising the Lagrangian in Section~\ref{}.

A \textit{path of least action} is a most likely generalised path
\begin{equation}
\label{eq: path of least action}
    \bvx\in\arg\min \mathcal A
\end{equation}
which, via the embedding of states into paths \eqref{eq: embedding states paths}, corresponds to a most likely generalised state (uniquely defined in terms of $\bvx$):
\begin{equation}
\label{eq: state of least action}
\begin{split}
    &\bvx\triangleq  \eta(\bvx_0)\\
    \bvx_0\in \arg\min \mathcal L
    %&\Leftrightarrow \mathbf D_N'  \mathbf x^{(:N+1)}_0 =\mathbf f(\mathbf x^{(:N)}_0)\Rightarrow \nabla\mathcal L(x^{(:N+1)}_t)=0
   \iff &\mathcal L(\bvx_0)=0 \Rightarrow \nabla\mathcal L(\bvx_0)=0.
\end{split}
\end{equation}
we will denote paths of least action and their corresponding generalised state with a bar $\bar \cdot\:$.

As usual, the path of least action is uniquely determined given a boundary condition (e.g. start and end points in classical mechanics); here this is the initial state in the zero-th coordinate $\bx_0^{(0)}=z$. Indeed, given the initial state $\bx_0^{(0)}$ and assuming $\mathbf \Sigma$ is non-singular, the path of least action is the unique minimizer of the Lagrangian which is the unique solution to 
\begin{equation}
\label{eq: minimiser of Lagrangian}
    \mathbf D' \bvx_0=\mathbf f(\bvx_0), \quad \bvx_0^{(0)}=z.
\end{equation} 
Unsurprisingly, the most likely generalised state $\bvx_0$ given boundary conditions is determined by substituting the most likely generalised noise state in the initial condition of the generalised Cauchy problem \eqref{eq: alg eq}. This is because the mapping $\mathbf w_0\mapsto \bvx_0$ given a boundary condition implicit in \eqref{eq: alg eq} is injective. In the non-Gaussian noise case we would simply substitute the most likely value(s) of $\mathbf w_0$ in the initial condition \eqref{eq: alg eq} to obtain the most likely generalised state(s).

We will see how to solve equations of the form \eqref{eq: alg eq}, \eqref{eq: minimiser of Lagrangian} explicitly in Section~\ref{sec: zigzag}.

% in the Gaussian case is most likely and the equation for the most likely generalised state is the initial condition of the generalised coordinate equation \eqref{eq: gen Cauchy problem} where we substitute the most likely noise value.

\subsubsection{Relationship to initial SDE}

Of course, the path of least action in generalised coordinates for a finite order $N$ will accurately approximate the most likely path of the initial SDE \eqref{eq: SDE}, solution to
\begin{equation}
\label{eq: least action original}
    \dot{\lx}_t = f(\lx_t), \quad \lx_0=z\in \R^d 
\end{equation}
on a short time interval where the Taylor polynomial is accurate. The solution to \eqref{eq: least action original} satisfies $\forall t\in \T, N \in \N$
\begin{equation}
\label{eq: L minimiser true most likely path}
    \L\begin{pmatrix}
\lx_t\\
    \frac d {dt}\lx_t\\
                        \vdots \\
        \frac{d^N}{dt^N} \lx_t
\end{pmatrix}=0 \quad \text{and} \quad \mathbf D'  \begin{pmatrix}
\lx_t\\
    \frac d {dt}\lx_t\\
                        \vdots \\
        \frac{d^N}{dt^N} \lx_t
\end{pmatrix}=\mathbf f\begin{pmatrix}
\lx_t\\
    \frac d {dt}\lx_t\\
                        \vdots \\
        \frac{d^N}{dt^N} \lx_t
\end{pmatrix} \quad \text{and} \quad \nabla \L\begin{pmatrix}
\lx_t\\
    \frac d {dt}\lx_t\\
                        \vdots \\
        \frac{d^N}{dt^N} \lx_t
\end{pmatrix}=0.
\end{equation}
We only have the equality between $\lx_t=\bvx^{(0)}_t$ on some time interval around $0$ when the flow in \eqref{eq: least action original} is analytic, and taking $N\to \infty$ in \eqref{eq: state of least action}.

Interestingly enough, the most likely path of the initial SDE \eqref{eq: least action original} can be approximately recovered through a generalised gradient descent on the finite order Lagrangian, as we will see in \cref{sec: most likely path via Lagrangian}.

% When $N\to +\infty$ and that the problem is analytic, something interesting happens:

\chapter{Numerical methods in generalised coordinates}
\label{chap: 3}

\section{Numerical integration via generalised coordinates}
\label{sec: num int}

We can use generalised coordinates to numerically solve stochastic differential equations on short time intervals. This simply involves computing the Taylor expansions of the solution corresponding to any given noise sample path.

\subsection{Zigzag methods}
\label{sec: zigzag}

Here we demonstrate two numerical integration algorithms for SDEs of the form \eqref{eq: SDE} driven by a stationary Gaussian process $w$ with sample paths (at least) $(N-1)$-times differentiable, almost surely,\footnote{Sufficient conditions for many times differentiability are given in Proposition~\ref{prop: Sample differentiability of GPs}.} and with a deterministic initial condition $x_0=z\in \R^d$.

\begin{enumerate}
    \item Sample an instance of $\vrw_0\sim \mathcal N(0,\mathbf \Sigma)$. These are the serial derivatives of a given sample path of $w$ at $t=0$. Here, $\mathbf \Sigma$ is given as a function of the autocovariance of $w$, see \eqref{eq: covariance gen fluctuations}.
    \item Solve for the generalised state $\vrx_0$, corresponding to the serial derivatives of the solution $x$ associated with the given sample path of $w$. The system of equations to solve is, cf. \eqref{eq: algebraic equation derivatives}, \eqref{eq: gen Cauchy problem}: 
    \begin{equation}
    \label{eq: zigzag}
      \begin{tikzcd}
          \rmx^{(1)}_0 \arrow[dr,  dotted] &  \arrow[l,  dotted, "="']  \mathbf f^{(0)}( \rmx_0^{(0)})+\rmw_0^{(0)}  \\
          \rmx^{(2)}_0 \arrow[dr,  dotted] & \arrow[l, dotted, "="']\mathbf f^{(1)}(\rmx_0^{(:1)})+\rmw_0^{(1)}\\
         \vdots\arrow[dr,  dotted] & \vdots\arrow[l,  dotted, "="']\\
         \rmx_0^{(N)} &\arrow[l,  dotted, "="']  \mathbf f^{(N-1)} (\rmx_0^{(:N-1)})+ \rmw_0^{(N-1)}
        \end{tikzcd}
    \end{equation}   
%     \begin{align}
%             x^{(1)}_0 &= \mathbf f^{(0)}( x_0^{(0)})+w_0^{(0)} \label{eq: 1st line zigzag}
%         \\
% x^{(2)}_0 &= \mathbf f^{(1)}(x_0^{(:1)})+w_0^{(1)} \label{eq: 2nd line zigzag}\\
%  &\vdots \nonumber\\
%          x_0^{(N+1)} &=\mathbf f^{(N)} (x_0^{(:N)})+ w_0^{(N)}.\label{eq: Nth line zigzag}
%     \end{align}
    % This is done straightforwardly by computing the right hand side of \eqref{eq: 1st line zigzag}, which immediately yields $ x^{(1)}_0$, which yields the RHS of \eqref{eq: 2nd line zigzag}, which yields its $ x^{(2)}_0$, etc.
    The solving procedure proceeds in a zigzag pattern, where we compute the right-hand side of the first line in \eqref{eq: zigzag}, which equals $ \rmx^{(1)}_0$, which yields the RHS of the second line, which equals $ \rmx^{(2)}_0$, etc.
    We call this the \textit{zigzag method}. The zigzag method requires knowledge of the generalised flow $\mathbf f$ (\cref{def: gen flow}), which one can compute analytically, or through symbolic or automatic differentiation. Alternatively, one can use the simpler, albeit approximate, local linear approximation~\ref{ap: local lin approx} of the flow \eqref{eq: bold f local linear}; in which case we say that we are using the \textit{(local) linearised zigzag method}.
    \item Compute and evaluate the Taylor polynomial to the solution:
    \begin{equation}    \rmx_t^{(0)}=\sum_{n=0}^{N}\rmx_0^{(n)}\frac{t^n}{n!}
    \end{equation}
\end{enumerate}
These zigzag methods allow us to compute exact or approximate Taylor expansions of solution sample paths, respectively.

\subsection{Simulations}

\begin{figure}[h]
\centering\includegraphics[width=0.7\textwidth]{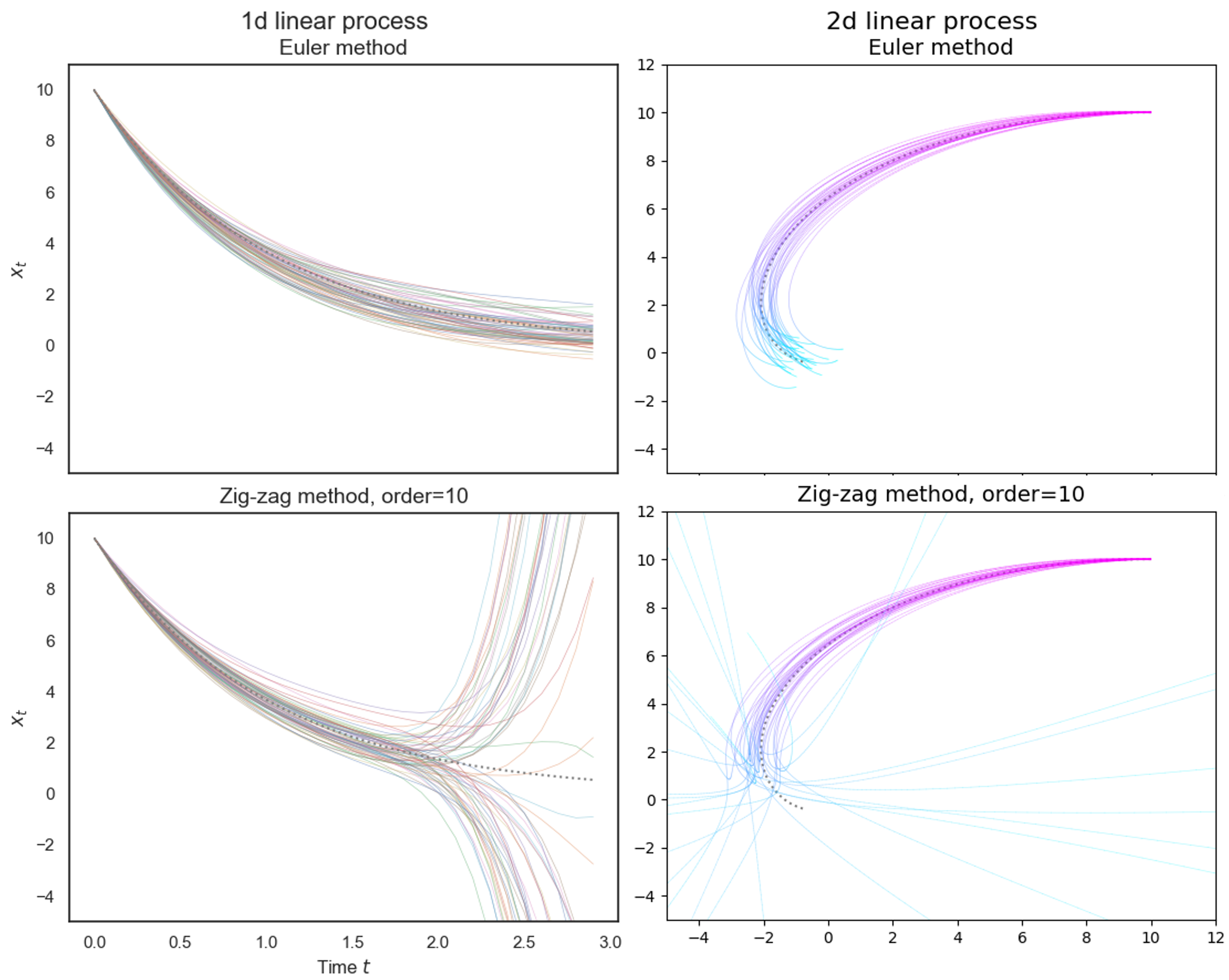}
\caption{\textbf{Numerical integration of 1 and 2 dimensional linear SDEs.} On the left (resp. right) panels, we simulate a one (resp. two) dimensional linear SDE driven by a stationary Gaussian process with Gaussian autocovariance (white noise convolved with a Gaussian). For the top panels, we sampled white noise paths, convolved them with a Gaussian, and numerically integrated the resulting SDE pathwise using Euler's method, providing a 'ground truth' comparative baseline. For the bottom panels, we used the zigzag numerical integration method (both zigzag methods coincide for this system), plotting an exact Taylor expansion of solution sample paths that uses derivatives up to order $10$. On the left, each (1d) sample path is represented as a function of time in a different colour, while on the right each (2d) sample path is plotted on a plane with a start in magenta and an end in cyan. The path of least action (i.e.~in the absence of noise) is plotted in each panel in a dotted grey line.
% Due to the distinct nature of the two numerical integration methods, the 
Note that the trajectories in the top and bottom panels cannot be compared individually---only statistically---since they correspond to distinct noise realisations (since the noise samples are generated differently in each method).
}
\label{fig: numint lin}
\end{figure}

\begin{figure}[h]
\centering\includegraphics[width=\textwidth]{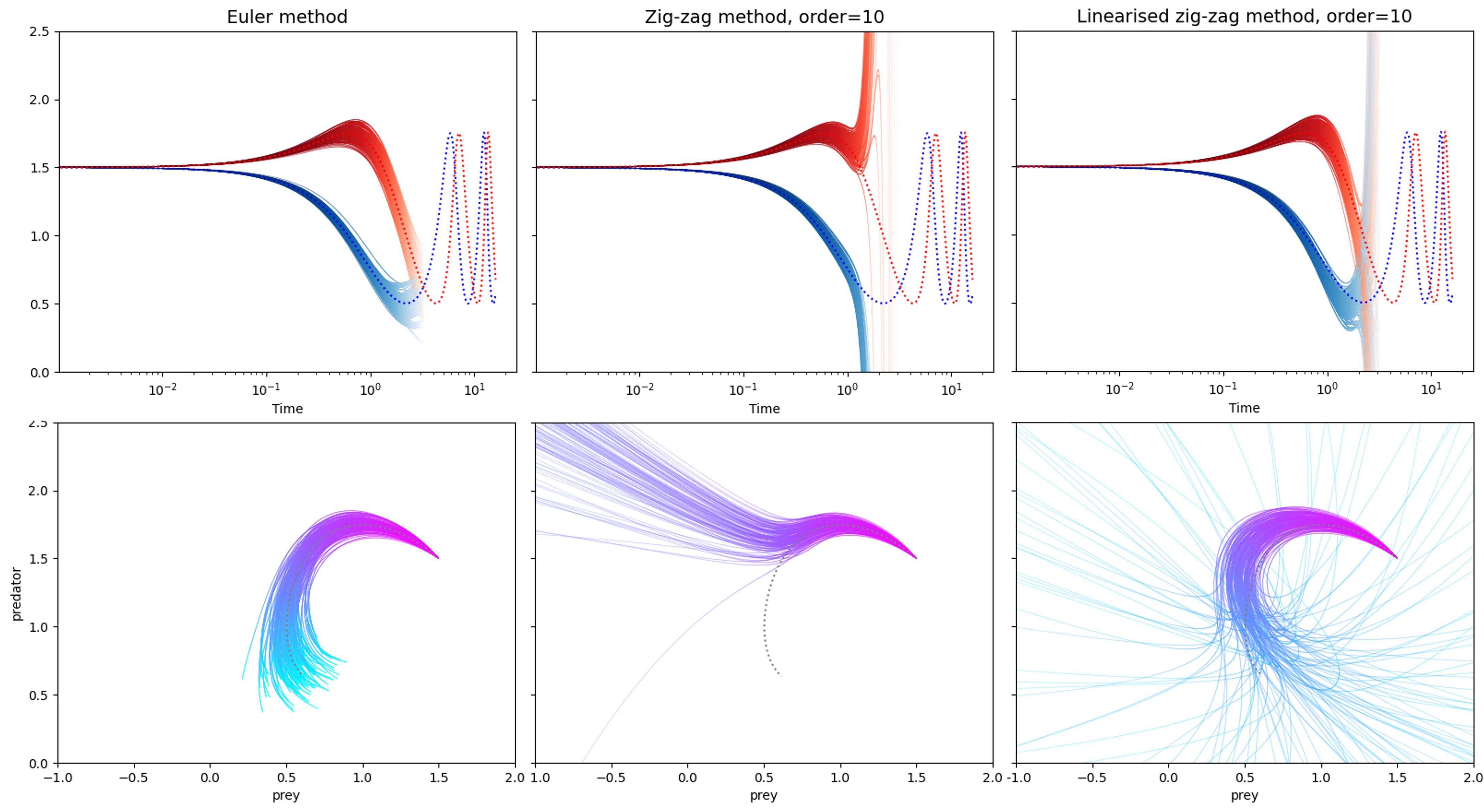}
\caption{\textbf{Numerical integration of stochastic Lotka-Volterra system.} We simulate a stochastic Lotka-Volterra system driven by a stationary Gaussian process with Gaussian autocovariance (white noise convolved with a Gaussian). The Lotka-Volterra system is a toy model for the relative proportions of predator and prey in an ecosystem~\cite{Mao_2002}. In the top panels we plot sample trajectories of the system over time, with the relative proportions of predators in red and prey in blue; in the bottom panels we plot these trajectories on a plane with a start in magenta and an end in cyan. The path of least action (i.e.~in the absence of noise) is plotted in each panel as a dotted line. In the first column, we plot trajectories generated with Euler's method, providing a 'ground truth' comparative baseline. In the second and third columns, we plot trajectories generated with the zigzag and linearised zigzag methods. The zigzag and linearised zigzag methods solved the system for the same set of noise sample paths, while Euler's method used a different set of noise samples (since the noise is realised differently in this method).}
\label{fig: numint LV}
\end{figure}

\begin{figure}[h]
\centering\includegraphics[width=\textwidth]{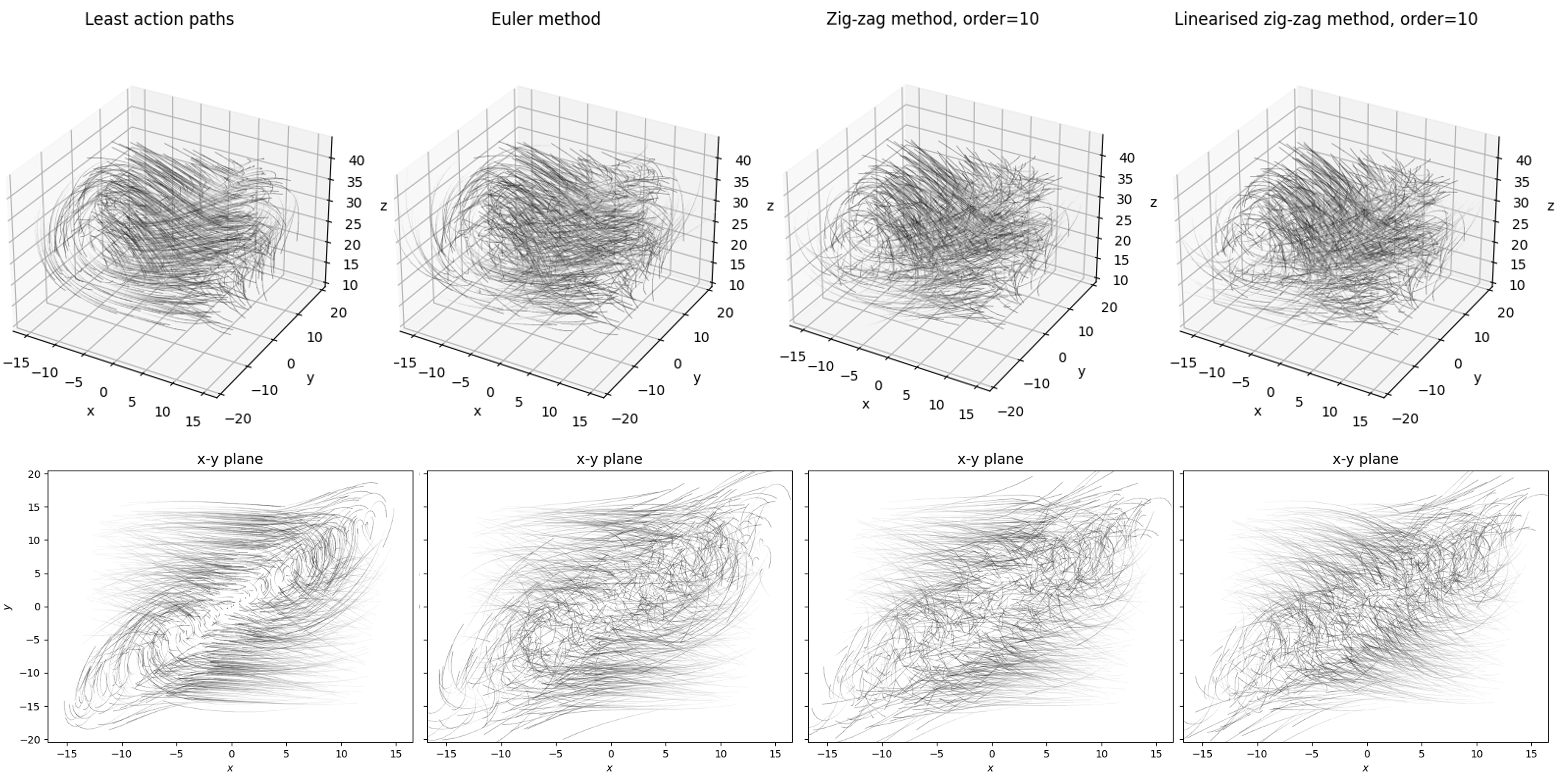}
\caption{\textbf{Numerical integration of stochastic Lorenz system.} We simulate a stochastic Lorenz system driven by a stationary Gaussian process with Gaussian autocovariance (white noise convolved with a Gaussian). The plots are sample trajectories at various initialisations (shared across panels), in 3d in the top panels, with their 2d projections on the $x$-$y$ plane in the bottom panels. The first column shows paths of least action (i.e.~in the absence of noise). The subsequent columns shows paths integrated with Euler, zigzag and linearised zigzag methods, respectively, where Euler serves as a comparative baseline. Importantly, the noise realisations %in these simulations
were shared only by the zigzag and linearised zigzag methods (since the noise is realised differently in Euler's method).}
\label{fig: numint Lorenz}
\end{figure}

We compare integration methods on four simple SDEs driven by a stationary Gaussian process with Gaussian autocovariance. This choice of noise has two benefits: 1) it is the canonical example of smooth noise; 2) it is perhaps the only example of coloured noise for which SDEs can be straightforwardly integrated using standard methods, providing a 'ground truth' comparative baseline. Indeed, this is because this noise can also be simulated by convolving white noise with a Gaussian, following which one can integrate the resulting SDE pathwise using Euler's method (or any other established ODE solver). 

The systems we simulate are: one dimensional and two dimensional linear SDEs, %(as in Section~\ref{}), 
see Figure~\ref{fig: numint lin}; a stochastic Lotka-Volterra system, see Figure~\ref{fig: numint LV}; and a stochastic Lorenz system, see Figure~\ref{fig: numint Lorenz}. For each system, we compare the zigzag methods with the baseline provided by Euler's method.

\subsection{Results}

As expected, the (non-linearised) zigzag method provides extremely accurate trajectories on short time-spans---in virtue of Taylor's theorem---until Taylor expansions blow up. The length of this timespan is somewhat inversely related with the amount of non-linearity of the underlying system. The surprising finding is that the linearised zigzag method yields trajectories that are almost as accurate as the exact method %on this timespan, %and that are more stable, in the sense that they remain accurate for a longer length of time.% timespan over which they are accurate is longer.
and which remain accurate on a longer timespan; in addition to being computationally cheaper, the local linearised method appears to provide reconstructions that are more stable in time.

By producing Taylor expansions to SDE solutions, these methods are only accurate on relatively short time-intervals. We will discuss how these methods could potentially be extended to accurately simulate smooth noise SDEs globally in time in \cref{sec: future directions}.

\section{Most likely path via Lagrangian}
\label{sec: most likely path via Lagrangian}

Now we show a method that %allows one to 
approximately recovers the most likely path of an SDE \eqref{eq: SDE} \textit{globally} in time from the finite order Lagrangian in generalised coordinates. The ideas herein %is unimportant when the SDE is known since the noise can simply be ignored then; but actually this method 
underlie the basis of (generalised) Bayesian filtering as presented in the next section, and 
have been used to relate the dynamics of multipartite processes to Bayesian filters in derivations of the free-energy principle \cite{fristonFreeEnergyPrinciple2023a,fristonPathIntegralsParticular2023}.

\subsection{The method}
\label{sec: path integral via Lagrangian method}

The most likely trajectory of our initial SDE \eqref{eq: SDE} is the unique solution to the SDE  without the noise \eqref{eq: least action original}. This trajectory minimises the Lagrangian $\L$ at all times and for all orders of motion $N$ \eqref{eq: L minimiser true most likely path}.
\textit{Suppose that we would like to recover this trajectory \emph{globally} in time from a generalised coordinates formulation with a \emph{finite} order of motion $N$}.

% When simulating the path of least action through generalised coordinates 
% In contrast, the path of least action for finite order $N$, solution to
% \eqref{eq: state of least action}, generally leaves the constraint set where $\L$ is minimised, and therefore diverges from the most likely path of the initial SDE. 

The path of least action in generalised coordinates \eqref{eq: state of least action} with a finite order of motion is the Taylor approximation to the target trajectory, which quickly diverges. 
To obtain a reconstruction globally in time, one can force the Taylor expansion \eqref{eq: state of least action} back onto the constraint set where the Lagrangian is minimised, which characterises our target trajectory. One simple
% Therefore, one would like to project the motion defining the path of least action \eqref{eq: state of least action} onto the constraint set, in order to approximate the most likely path of the SDE \eqref{eq: SDE}. One inexact yet simple 
way to do this is to subtract the gradient of the Lagrangian to the flow of the equation which characterises the path of least action in action generalised coordinates \eqref{eq: embedding states paths}, so that this motion is always pulled toward the constraint set. This results in the following dynamical equation
\begin{equation}
\label{eq: regularised equation for least action}
 \dot{\bvx}_t= \mathbf D \bvx_t
    -\lambda\nabla\mathcal L(\bvx_t).
\end{equation}
with the usual the initial condition \eqref{eq: minimiser of Lagrangian}.
Here $\lambda >0$ is a weight, which interpolates between the unconstrained Taylor expansion ($\lambda=0$) and the Taylor expansion projected onto the constraint set ($\lambda\to +\infty$). We chose this notation to be reminiscent of Lagrange multipliers. The intuition is that the solution to \eqref{eq: regularised equation for least action} should tend to the most likely path of \eqref{eq: SDE} as $\lambda$ grows large. 

%Furthermore, for any positive $\lambda$, the solution to \eqref{eq: regularised equation for least action} cannot escape to infinity. This is because the flow $\mathbf D_N x^{(:N+1)}_t$ does not influence the last order of motion (this part of the flow is zero for the last order) and $-\lambda\nabla\mathcal L( x^{(:N+1)}_t)$, which is non-zero for the last order, pulls 

\subsection{Simulations and results}

\begin{figure}[h]
\centering\includegraphics[width=\textwidth]{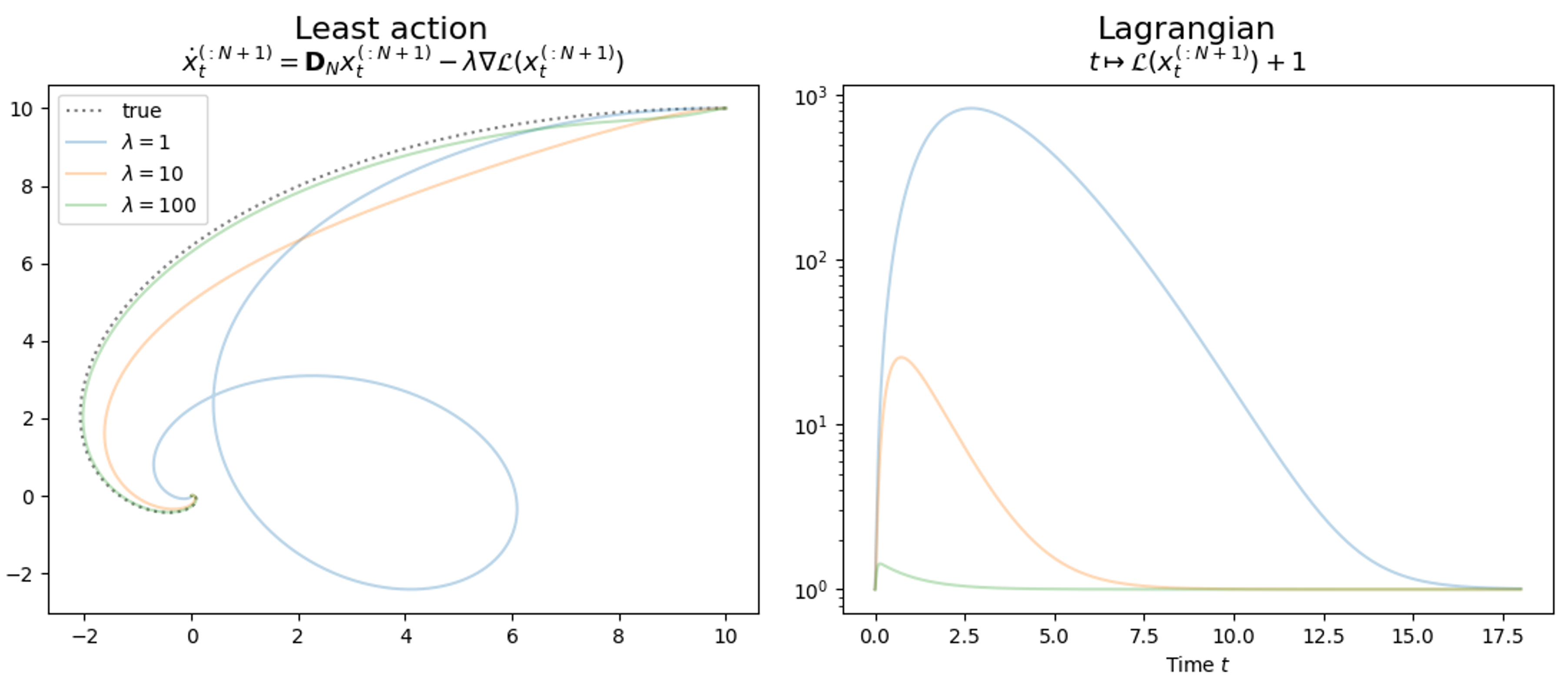}
\caption{\textbf{Recovering path of least action for linear SDE.} \textit{Left}: We show the true path of least action of a given 2 dimensional linear SDE with initial condition $(10,10)$ in a dashed grey line. We then show the solution to \eqref{eq: regularised equation for least action} for $N+1=4$ orders of motion and different weightings $\lambda=1,10,100$, confirming the claim that it tends to the path of least action as $\lambda$ grows large. \textit{Right}: We plot the Lagrangian over time ($+1$ to plot in log-scale) for solutions to \eqref{eq: regularised equation for least action} corresponding to respective weightings of $\lambda$. The closer the solution is to the path of least action the more the Lagrangian is minimised at all times, and vice-versa.
}
\label{fig: OU least}
\end{figure}

We verify this intuition on a simple simulation of a two dimensional linear SDE, where we numerically integrate \eqref{eq: regularised equation for least action} (with the initial condition \eqref{eq: minimiser of Lagrangian}) using Euler's method, for $N=4$ orders of motion, with several weightings of $\lambda >0$. The upshot is that indeed the solution to \eqref{eq: regularised equation for least action} tends to the most likely path of the SDE \eqref{eq: SDE} as $\lambda$ grows large; please see Figure~\ref{fig: OU least}. 

Note that for non-linear SDEs, higher order numerical integrators are usually required to successfully implement this method (e.g. the Runge–Kutta–Fehlberg \cite{hairerSolvingOrdinaryDifferential2008} or Ozaki solvers \cite{ozakiBridgeNonlinearTime1992,fristonVariationalTreatmentDynamic2008}) as the integration problem tends to be a stiff one.

\section{Generalised Bayesian filtering}
\label{sec: GF}

The goal of Bayesian filtering---like any filtering algorithm---is to infer latent variables from observations. What follows are methods that leverage generalised coordinates to infer stochastic signals driven by coloured noise, with potentially long-range dependencies. As we argue in the introduction, this is a more realistic noise assumption than white noise in many scenarios, which include radio signaling, neuroimaging, radar tracking and control systems. To our knowledge, generalised filtering was originally developed in the late 2000s to filter smooth medical time-series obtained through functional magnetic resonance imaging  \cite{fristonVariationalTreatmentDynamic2008,fristonGeneralisedFiltering2010}. The approach has since percolated into other fields, showing state-of-the-art performance in several experiments
\cite{balajiBayesianStateEstimation2011,bos2022free,meera2023adaptive,anil2021dynamic}. Before delving into the methods, note that generalised filtering can be thought of as a higher order version of an extended Kalman filter developed specifically for coloured noise. For interested readers, experiments comparing generalised filtering, extended Kalman filtering, particle filtering and other established schemes were reported in \cite{fristonVariationalTreatmentDynamic2008,fristonGeneralisedFiltering2010}.

\subsection{The simplest generative model}
\label{sec: gen mod}

As a Bayesian method for filtering, the first thing to specify is a generative model for our empirical observations. For didactic purposes, we consider a simple generative model and mention extensions in~\cref{sec: extensions GF}. The question of choosing a suitable generative model in any given situation is an important one, which we develop in~\cref{sec: extensions GF}.

We assume that latent variables follow a stochastic differential equation, which constitutes the prior:
\begin{equation}
\label{eq: prior model}
\frac{d x_t}{d t}=f\left(x_t\right)+w_t, \quad x_t \in \R^d
\end{equation}
where $w_t$ is an $\R^d$-valued stationary Gaussian process with $(N-1)$-times continuously differentiable sample paths a.s. and $f: \R^d \to \R^d$ is an $(N-1)$-times continuously differentiable vector field, and the initial condition for the SDE $x_0$ is distributed according to the Lebesgue measure. (What follows can be adapted to initial conditions that follow a probability distribution). In addition, we assume that observations are sampled as a function of latent variables plus some noise, which constitutes the likelihood
\begin{equation}
\label{eq: likelihood model}
y_t=g\left(x_t\right)+z_t, \quad y_t \in \R^{m},
\end{equation}
where $g: \R^d\to \R^{m}$ is $M$-times continuously differentiable, and $z_t$ is an $\R^m$-valued stationary Gaussian process with $M$-times continuously differentiable sample paths a.s., and we choose $M \leq N$ for later.

\subsubsection{Reformulation in generalised coordinates}

We reformulate the dynamics in generalised coordinates, which will simplify matters later. For the prior model we have, for all $t \in \T$
\begin{equation}
% D\bx_t = \mathbf{f}(\bx_t )+\tilde w_t, \quad \tilde w_t\sim \mathcal N(0,\Sigma_w), \quad \bx_t \in \mathbb{G}^m
% \mathbf D_N' \mathbf x^{(:N+1)}_0=\mathbf f(\mathbf x^{(:N+1)}_0).
\mathbf D'  \bx_t =\mathbf f(\bx_t)+ \bw_t, \quad \bw_t\sim \mathcal N(0,\mathbf \Sigma^w), 
\end{equation}
with $\mathbf D'\bx_t, \bw_t$ random variables on $\G^{d,N-1}$, $\bx_t$ on $\G^{d,N}$ and where $\mathbf \Sigma^w \in \Mat_{Nd,Nd}$ is given in terms of the autocovariance of $w$ as \eqref{eq: covariance gen fluctuations}. As for the likelihood model
\begin{equation}
\by_t = \mathbf{g}(\bx_t)+\bz_t, \quad \bz_t\sim \mathcal N(0,\mathbf\Sigma^z),
\end{equation}
with $\by_t, \bz_t$ random variables on $\G^{d,M}$ and $\mathbf\Sigma^z \in \Mat_{(M+1)d,(M+1)d}$ is given in terms of the autocovariance of $z$ as \eqref{eq: covariance gen fluctuations}, and the likelihood function in generalised coordinates $\mathbf{g}$ is obtained analogously to \cref{def: gen flow}. Specifically, the generalised likelihood is the function 
\begin{equation}
    \mathbf g: \G^{d,N}\to \G^{d,M}, \mathbf g(\rmz^{(:N)})\triangleq \begin{pmatrix}
    \mathrm{g}^{(0)}(\rmz^{(0)}) &\ldots& \mathrm{g}^{(M)}(\rmz^{(:M)})
\end{pmatrix}^\top
\end{equation}
where each component $\mathrm{g}^{(m)}: \G^{d,m}\to \R^d$ is defined implicitly through the relation 
\begin{equation}
    \mathrm{g}^{(m)}\left(x_t, \frac{d}{dt}x_t, \ldots, \frac{d^m}{dt^m}x_t\right)\triangleq\frac{d^m}{dt^m}g(x_t).
\end{equation}
See \eqref{eq: gen filt linearised flows} for the expression of $\mathbf g$ under the local linear approximation~\ref{ap: local lin approx}.

Altogether this furnishes a generative model under which we can do Bayesian inference:\footnote{Note that the prior here is an unnormalised distribution and becomes a probability distribution as soon as one specifies a probability measure for the initial condition in \eqref{eq: prior model}.}
\begin{equation}
\label{eq: gen mod}
\begin{split}
            % p(D\bx_t \mid \bx_t )&=\mathcal N(D\bx_t ;\mathbf{f}(\bx_t ),\Sigma_w)\\
        p(\by_t,\bx_t)&\triangleq p(\by_t\mid \bx_t)p(\bx_t)\\
           p(\by_t\mid \bx_t)&=\mathcal N\left(\by_t ;\mathbf{g}(\bx_t),\mathbf\Sigma^z\right) \\ \quad p(\bx_t)&=\mathcal N\left(\mathbf D'  \bx_t -\mathbf f(\bx_t);0,\mathbf\Sigma^w\right).
\end{split}
\end{equation}

\subsection{Approximate Bayesian inference under the Laplace approximation}

We have the following Bayesian inference problem:
\begin{equation}\label{bayes_rule}
p(\bx_t\mid \by_t )=\frac{p(\by_t\mid \bx_t)p(\bx_t)}{p(\by_t)}
\end{equation}

\subsubsection{Variational Bayesian inference}

And we want to compute---or approximate---the posterior on the left. We do this through variational Bayesian inference \cite{bleiVariationalInferenceReview2017,bishopPatternRecognitionMachine2006}, where we approximate the posterior distribution with an approximate posterior distribution $q(\bx_t)$, by minimizing the following free energy functional:
\begin{equation}
\label{eq: VFE_def}
\begin{split}
\F(q; \by_t)
&\triangleq \E_q[\log q(\bx_t)-\log p(\by_t,\bx_t)]\\
% & \triangleq \int_{\R^{d(N+2)}} q(z^{(:N+1)}) \log\left(\frac{q(z^{(:N+1)})}{p(\by_t, z^{(:N+1)})}\right)  \: dz^{(:N+1)}\\
&=\underbrace{\E_q[-\log p(\by_t,\bx_t)]}_{\text{Expected energy}}-\underbrace{\E_q[-\log q(\bx_t)]}_{\text{Entropy}}\\
&= \dkl\left[  q(\bx_t)   \mid   p(\bx_t\mid \by_t) \right]-\log p(\by_t)\\
&\geq -\log p(\by_t)
\end{split}
\end{equation}
In the penultimate line we can see that the free energy considered here is the KL divergence \cite{kullbackInformationSufficiency1951}---or relative entropy---between the approximate posterior and the true posterior, plus a constant. In other words, the free energy is minimized when and only when the approximate posterior equals the true posterior:
\begin{equation}
    \delta_q \F(q; \by_t)=0\iff  q(\bx_t) =  p(\bx_t\mid \by_t).
\end{equation}

\subsubsection{The Laplace approximation}

We are left with computing the free energy. The expected energy term in \eqref{eq: VFE_def} is generally challenging to compute exactly and a standard approximation is the Laplace approximation \cite{mackayBayesianInterpolation1992,fristonVariationalFreeEnergy2007,williamsBayesianClassificationGaussian1998}. In the Laplace approximation, we approximate the energy $V$ 
\begin{equation}\label{eq: energy}
    \begin{split}
        V(\by_t,\bx_t)&\triangleq -\log p(\by_t,\bx_t)\\
  &\doteq \frac{1}{2}\left(\by_t - \mathbf g(\bx_t)\right)^\top (\mathbf\Sigma^z)^{-1} \left(\by_t - \mathbf g(\bx_t)\right) +\frac{1}{2}\left(\mathbf D'  \bx_t -\mathbf f(\bx_t)\right)^\top (\mathbf\Sigma^w)^{-1}\left(\mathbf D'  \bx_t -\mathbf f(\bx_t)\right)
    \end{split}
\end{equation}

by a second order Taylor expansion at the mean of the approximate posterior $\bmu\triangleq \mu^{(:N)}\triangleq \E[q]$. From \eqref{eq: gen mod}:
% \begin{equation}
% \mathbf D'  \bx_t =\mathbf f(\bx_t)+ w^{(:N)}_t, \quad w^{(:N)}_t\sim \mathcal N(0,\mathbf\Sigma^w) 
% \by_t = \mathbf{g}(\rmx_t^{(:M)} )+z^{(:M)}_t, \quad z^{(:M)}_t\sim \mathcal N(0,\mathbf\Sigma^z),
% \end{equation}
\begin{equation}
\label{eq: Laplace approximation}
\begin{split}
    V(\by_t,\bx_t)&\approx V(\by_t,\bmu)+\nabla_{\bmu} V(\by_t,\bmu)^\top\left(\bx_t-\bmu\right)+\frac 12\left(\bx_t-\bmu\right)^\top\nabla_{\bmu}^2 V(\by_t,\bmu)\left(\bx_t-\bmu\right).
\end{split}
\end{equation}

One can equivalently say that the Laplace approximation approximates the true posterior distribution by a Gaussian centered at $\bmu$. Indeed, by Bayes rule, the energy and the negative log posterior distribution are equal up to a constant, and the negative log posterior is quadratic if only if the posterior is Gaussian. Under the Laplace approximation we therefore choose our approximate posterior to also be Gaussian:
\begin{equation}
q(\bx_t; \bmu, \bSigma) \triangleq \mathcal N(\bx_t;\bmu,\bSigma)= \frac{1}{\sqrt{(2\pi)^{d(N+1)} \det(\bSigma) }} \exp \left( -\frac{1}{2}(\bx_t - \bmu)^\top \bSigma^{-1} (\bx_t - \bmu)\right).
\end{equation} 
Under this parametric form, the free energy becomes a function of the parameters of the approximate posterior. The Laplace approximation is useful because the expected energy and entropy terms that form the free energy \eqref{eq: VFE_def} become straightforward to evaluate:
\begin{align}
%\E_q[-\log p(\by_t,\bx_t)]&=
\E_q[V(\by_t,\bx_t)]
    &\approx\E_q[V(\by_t,\bmu)]+ \E_q[\nabla_{\bmu} V(\by_t,\bmu)^\top\left(\bx_t-\bmu\right)]+\frac{1}{2}\E_q[\left(\bx_t-\bmu\right)^\top\nabla_{\bmu}^2 V(\by_t,\bmu)\left(\bx_t-\bmu\right)]\nonumber\\
    &=V(\by_t,\bmu)+\frac{1}{2}\tr [\nabla_{\bmu}^2 V(\by_t,\bmu)\bSigma] \label{eq: expected energy}\\
    \mathbb{E}_q[-\log q( \bx_t; \bmu, \bSigma)]
&=\frac{1}{2} \log \operatorname{det}\bSigma+\frac{d(N+1)}{2} \log (2 \pi \mathrm{e}) \label{eq: entropy}
\end{align}
where the latter is the well-known entropy of a Gaussian distribution.

Consequently, we obtain a simple expression for the free energy under the Laplace approximation $\FL$, by subtracting entropy \eqref{eq: entropy} from the approximated expected energy \eqref{eq: expected energy}, cf. \eqref{eq: VFE_def}:
\begin{equation}
\begin{split}
    \F(\bmu,\bSigma ; \by_t)&\approx \FL(\bmu,\bSigma ; \by_t)
    % &\triangleq  -\log p(\by_t,\mu)-\frac{1}{2}\tr \left[\nabla_{\bmu}^2 \log p(\by_t,\mu)\bSigma\right]-\frac{mo_x}{2} \log (2 \pi \mathrm{e}) -\frac{1}{2} \log \operatorname{det}\bSigma\\
    \triangleq V(\by_t,\bmu)+\frac{1}{2}\tr \left[\nabla_{\bmu}^2 V(\by_t,\bmu)\bSigma\right]-\frac{1}{2} \log \operatorname{det}\bSigma-\frac{d(N+1)}{2} \log (2 \pi \mathrm{e}).
\end{split}
\end{equation}

It remains to find the value of the parameters $\bmu, \bSigma$ that minimise the Laplace free energy, corresponding to the parameterisation of the approximate posterior that is closest to the true posterior.

\subsubsection{Optimising the covariance}

We can find the optimal covariance $\bSigma^*$ by computing the free energy gradient (using Jacobi's formula):
\begin{equation}
    \begin{split}
        \nabla_\bSigma\FL(\bmu,\bSigma ; \by_t)&=\frac{1}{2}\nabla_\bSigma\tr \left[\nabla_{\bmu}^2 V(\by_t,\bmu)\bSigma\right] -\frac{1}{2}\nabla_\bSigma\log \operatorname{det}\bSigma\\
        %&=\frac{1}{2}\nabla_s^2 V-\frac{1}{2}\left(\operatorname{det}\bSigma\right)^{-1}\operatorname{adj}\bSigma^\top 
        &=\frac{1}{2}\nabla_{\bmu}^2 V(\by_t,\bmu)-\frac{1}{2}\bSigma^{-1}\\ % computation analysis notebook generalised coordinates November 3 2023
    \end{split}
\end{equation}
and seeing when it vanishes:
\begin{equation}
\label{eq: optimal covariance}
    \nabla_\bSigma\FL(\bmu,\bSigma ; \by_t)=0\iff \bSigma=\bSigma^*(\by_t,\bmu)\triangleq \left (\nabla_{\bmu}^2 V(\by_t,\bmu)\right)^{-1}.
\end{equation}

It follows that the Laplace free energy evaluated at the optimal covariance reads as follows:
\begin{equation}
\label{eq: laplace FE optimal cov}
    \begin{split}
        \FL(\bmu,\bSigma^* ; \by_t)%&=V(\by_t,\mu)-\frac{mo_x}{2} \log (2 \pi) -\frac{1}{2} \log \operatorname{det}\bSigma^*(\by_t,\mu)\\
        &=V( \by_t,\bmu)+\frac{1}{2} \log \operatorname{det}\nabla_{\bmu}^2 V(\by_t,\bmu)-\frac{d(N+1)}{2} \log (2 \pi \mathrm{e}).
    \end{split}
\end{equation}

\subsubsection{Optimising the mean}

The mean $\bmu_t$ should evolve according to a similar dynamic as $\dot {\bx}_t =\mathbf D\bx_t$, as generalised coordinates of motion encode information about the higher order derivatives of the process and thus its time evolution. The optimal mean, in addition, should minimise Laplace free energy \eqref{eq: laplace FE optimal cov}.
Generalised filtering approaches this by a generalised gradient descent on Laplace free energy \eqref{eq: laplace FE optimal cov} with hyperparameter $\lambda>0$:
\begin{equation}
\label{eq: non-linear gen filt}
\begin{split}
        \dot {\bmu}_t &=\mathbf D\bmu_t -\lambda\nabla_{\bmu} \FL(\bmu_t,\bSigma^* ; \by_t) \\
    &=\mathbf D\bmu_t - \lambda\nabla_{\bmu} V( \by_t,\bmu_t)-\frac{\lambda}{2}\nabla_{\bmu} \log \det\nabla_{\bmu}^2 V(\by_t,\bmu_t).
\end{split}
\end{equation}
The first term in the flow $\mathbf D\bmu_t$ ensures that generalised coordinates are properly integrated in time so that they remain semantically meaningful as higher order derivatives; the second term $-\lambda\nabla_{\bmu} \FL(\bmu_t,\bSigma^* ; \by_t)$ is a potential term, ensuring that the generalised motion evolves toward the free energy minimum; $\lambda>0$ is a hyperparameter specifying the strength of the potential term, cf. Section~\ref{sec: path integral via Lagrangian method}. Readers familiar with higher order optimisation and sampling methods may find that they operate under the same principles as \eqref{eq: non-linear gen filt}: for example, in underdamped Langevin dynamics \cite{pavliotisStochasticProcessesApplications2014,barpGeometricMethodsSampling2022a} higher orders of motion are integrated in time using the first term of \eqref{eq: non-linear gen filt}, and the generalised motion is further influenced by a potential term. Indeed, under static data (i.e. constant $t \mapsto \by_t$) the potential term in \eqref{eq: non-linear gen filt} is static and generalised filtering reduces to an underdamped, high order optimisation scheme. In general though, generalised filtering is a higher order optimisation scheme on a moving target.

In \eqref{eq: non-linear gen filt}, the term $\nabla_{\bmu} V( \by_t,\bmu)$ can be interpreted as a weighted sum of generalised prediction errors (from \eqref{eq: energy}):
\begin{equation}
\label{eq: weighted sum of prediction errors}
\begin{split}
        \nabla_{\bmu} V( \by_t,{\bmu})%&=\frac{1}{2}\nabla_{\bmu} \left[(\by_t - \mathbf g(\mu))^\top \Sigma_z^{-1} (\by_t - \mathbf g(\mu)) +(D\mu - \mathbf{f}(\mu ))^\top \Sigma_w^{-1} (D\mu - \mathbf{f}(\mu ))\right]\\
        &=(\nabla \mathbf g({\bmu}))^\top(\mathbf\Sigma^z)^{-1}( \mathbf g({\bmu})-\by_t)+(\mathbf D' -\nabla \mathbf f({\bmu} ))^\top (\mathbf\Sigma^w)^{-1} (\mathbf D'{\bmu} - \mathbf{f}({\bmu} )).
\end{split}
\end{equation}

The latter term in the filtering dynamic \eqref{eq: non-linear gen filt} can be rewritten component-wise via Jacobi's formula as 
      \begin{equation}
      \label{eq: FE grad is third order}
        %\left( \nabla_{\bmu} \log \operatorname{det}\nabla_{\bmu}^2 V(\by_t,{\bmu}) \right)_i \triangleq 
        \partial_{\bmu_i}\log \det\nabla_{\bmu}^2 V(\by_t,\bmu)= \tr\left(\left (\nabla_{\bmu}^2 V(\by_t,\bmu)\right)^{-1}\partial_{\bmu_i}\nabla_{\bmu}^2 V(\by_t,\bmu)\right).
      \end{equation}

Therefore, \textit{generalised filtering \eqref{eq: non-linear gen filt}
 corresponds to a third order optimisation scheme on the energy} $V( \by_t,\bmu)$.

\begin{remark}[Domain of definition]
    It is important to realise that the Laplace free energy \eqref{eq: laplace FE optimal cov} is defined for $\bmu$'s such that $\det\nabla_{\bmu}^2 V(\by_t,\bmu)>0$, % This need not always hold, in cases where the Laplace approximation is not an equality. However the free energy gradient can be extended to the wider space where $\operatorname{det}\nabla_{\bmu}^2 V(\by_t,\mu)\neq 0$. Eventually, as the method stabilizes, it converges onto regions where the free energy is defined. We seem to leave the region where the free energy is defined when the algorithm is highly surprised.
however, its gradient can be extended through \eqref{eq: FE grad is third order} to the much wider region of $\bmu$'s such that $\det\nabla_{\bmu}^2 V(\by_t,\bmu)\neq 0$, which makes the simulation of \eqref{eq: non-linear gen filt} straightforward. Indeed, the continuous-time solution to generalised filtering \eqref{eq: non-linear gen filt} is continuous and cannot jump from a region where $\det\nabla_{\bmu}^2 V(\by_t,\bmu)<0$ to another where $\det\nabla_{\bmu}^2 V(\by_t,\bmu)>0$, such as a neighbourhood of the free energy minimum, as in between lies a singularity where the free energy gradient is not defined. But this issue becomes moot when numerically integrating \eqref{eq: non-linear gen filt} since discrete steps generally avoid such singularities which are a null set. This allows us to initialise \eqref{eq: non-linear gen filt} almost anywhere---that is, outside such singularities.
\end{remark}

\subsection{Summary of generalised filtering}
\label{sec: summary gf}

In summary, generalised filtering dynamically updates a Gaussian approximate posterior belief $q$ over hidden variables $\bx_t$ as one samples a time-series of observations $t\mapsto \by_t$. This dynamical updating is based upon a generative model $p(\by_t,\bx_t)$ and its associated energy function $V=-\log p$ as follows:
\begin{equation}
\label{eq: gen filt}
\begin{split}
    q(\bx_t; \bmu_t, \bSigma_t) &\triangleq\mathcal N(\bx_t;\bmu_t,\bSigma_t)\approx p(\bx_t \mid \by_t)\\
    \dot {\bmu}_t &=\mathbf D\bmu_t - \lambda\nabla_{\bmu} V( \by_t,\bmu_t)-\frac{\lambda}{2}\nabla_{\bmu} \log \det\nabla_{\bmu}^2 V(\by_t,\bmu_t)\\
    \bSigma_t&=\left (\nabla_{\bmu}^2 V(\by_t,\bmu_t)\right)^{-1}
\end{split}
\end{equation}
where the initial condition $\bmu_0$ is arbitrary and $\lambda>0$ is a hyperparameter scoring to what extent the solution is forced back to the free energy minimum. The generative model is given by \eqref{eq: gen mod} and the gradient of the log determinant term has the closed form \eqref{eq: FE grad is third order}. The dynamics \eqref{eq: gen filt} are simulated throughout the time interval spanning the time-series of observations $t\mapsto \by_t$. We defer further implementation details to \cref{sec: implementation details}.

\begin{remark}
In virtue of the second line of \eqref{eq: gen filt}, generalised filtering entails a third-order optimisation scheme for the mean of our posterior belief. However, to the extent that the Laplace approximation \eqref{eq: Laplace approximation} is valid, this third order term is approximately zero and is often neglected in practice so that the scheme becomes first order. We will now see what such simplifications are afforded by the local linear approximation.  
\end{remark}

\subsection{Simplifications under the local linear approximation}
\label{sec: simplification local linear}

We now show how generalised filtering looks under the local linear approximation~\ref{ap: local lin approx}, which simplifies equations considerably. First of all, recall that the local linear approximation involves ignoring all derivatives of the flows $f,g$ of orders higher than one---that is, assuming that they vanish. We employ this assumption throughout this section.

Under the local linear approximation~\ref{ap: local lin approx}, the generalised flows read as:

\begin{equation}
\label{eq: gen filt linearised flows}
    \mathbf f(\bx_t)=\begin{pmatrix}
        f\left(\rmx_t^{(0)}\right)\\
    \nabla f\left(\rmx_t^{(0)}\right)\rmx_t^{(1)}\\
    \vdots \\
        \nabla f\left(\rmx_t^{(0)}\right)\rmx_t^{(N)}
    \end{pmatrix}, \quad \mathbf{g}(\bx_t )=\begin{pmatrix}
        g\left(\rmx_t^{(0)}\right)\\
    \nabla g\left(\rmx_t^{(0)}\right)\rmx_t^{(1)}\\
    \vdots \\
        \nabla g\left(\rmx_t^{(0)}\right)\rmx_t^{(M)}
    \end{pmatrix}
\end{equation}

The gradient of the energy \eqref{eq: weighted sum of prediction errors} has the same functional form, but the Jacobians of the generalised flows within are given by considerably simpler expressions:

\begin{equation}
\label{eq: local linear large Jacobians}
\begin{split}
    \underbrace{\nabla \mathbf f({\bmu})}_{\Mat_{d(N+1), d(N+1)}}=\underbrace{\nabla f(\mu^{(0)})}_{\Mat_{d, d}}\otimes \operatorname{I}_{N+1},\:    \underbrace{\nabla \mathbf g({\bmu})}_{\Mat_{m(M+1), d(N+1)}}&=\underbrace{\nabla g(\mu^{(0)})}_{\Mat_{m, d}}\otimes \underbrace{\begin{pmatrix}
       1&&&&\overbrace{0}^{\Mat_{M+1, N-M}}\\
        & 1&&&0\\
        & &\ddots &&\vdots\\
        & &&1&0
    \end{pmatrix}}_{ \Mat_{M+1,N+1}}
\end{split}
\end{equation}

The Hessian of the energy has a simpler form:
\begin{equation}
\label{eq: energy Hessian local lin}
\begin{split}
        \nabla_{\bmu}^2 V( \by_t,\bmu)%&=\nabla_{\bmu} \left[(\nabla \mathbf g(\mu))^\top(\mathbf\Sigma^z)^{-1}( \mathbf g(\mu)-\by_t)+(D -\nabla \mathbf f(\mu ))^\top (\mathbf\Sigma^w)^{-1} (D\mu - \mathbf{f}(\mu ))\right]\\
        &=(\nabla \mathbf g({\bmu}))^\top(\mathbf\Sigma^z)^{-1}\nabla \mathbf g({\bmu})+(\mathbf D' -\nabla \mathbf f({\bmu} ))^\top (\mathbf\Sigma^w)^{-1} (\mathbf D' -\nabla \mathbf f({\bmu} ))
\end{split}
\end{equation}
The coefficients of this Hessian are affine linear combinations of products of first order derivatives of the flows $f,g$. Consequently all third order derivatives of the energy vanish:

\begin{equation}
\label{eq: 3rd order derivs vanish}
        \nabla_{\bmu}^3 V( \by_t,{\bmu})\equiv 0.
\end{equation}
This means that the the local linear approximation~\ref{ap: local lin approx} implies the Laplace approximation \eqref{eq: Laplace approximation}:
\begin{equation}
    \F=\FL.
\end{equation}
The optimal covariance $\bSigma^*$ is given as the inverse of \eqref{eq: energy Hessian local lin}, cf. \eqref{eq: optimal covariance}. And finally, the free energy gradient---evaluated at $\bSigma^*$---is just the energy gradient:
\begin{equation}
    \nabla_{\bmu} \FL(\bmu,\bSigma^* ; \by_t) 
    =\nabla_{\bmu} V( \by_t,\bmu)
\end{equation}
This is because the gradient of the log determinant term in the free energy \eqref{eq: laplace FE optimal cov} is a third order derivative of the energy \eqref{eq: FE grad is third order}, which vanishes under the local linear approximation by \eqref{eq: 3rd order derivs vanish}.

\subsection{Summary of local linearised generalised filtering}
\label{sec: local linear gf}

In summary, generalised filtering under the local linear approximation~\ref{ap: local lin approx} optimises a Gaussian approximate posterior belief $q$ over hidden variables $\bx_t$ as one samples a time-series of observations $t\mapsto \by_t$.  This dynamical updating is based upon a generative model $p(\by_t,\bx_t)$ \eqref{eq: gen mod} and its associated energy function $V=-\log p$ with flows $\mathbf f, \mathbf g$ \eqref{eq: gen filt linearised flows}, as follows:
\begin{equation}
\label{eq: linear gen filt}
\begin{split}
    q(\bx_t; \bmu_t, \bSigma_t) &\triangleq\mathcal N(\bx_t;\bmu_t,\bSigma_t)\approx p(\bx_t \mid \by_t)\\
    \dot {\bmu}_t &=\mathbf D\bmu_t - \lambda\nabla_{\bmu} V( \by_t,\bmu_t)\\
    \bSigma_t&=\left (\nabla_{\bmu}^2 V(\by_t,\bmu_t)\right)^{-1}
\end{split}
\end{equation}
where the initial condition $\bmu_0$ is arbitrary, $\lambda>0$ is a hyperparameter measuring to what extent the solution is forced back to the free energy minimum, the gradient and Hessian of the energy $V$ are given by \eqref{eq: weighted sum of prediction errors}, \eqref{eq: energy Hessian local lin} respectively, and the Jacobians of the generalised flows $\nabla\mathbf f, \nabla\mathbf g$ %and their Jacobians given by \eqref{eq: gen filt linearised flows} and 
within are given by \eqref{eq: local linear large Jacobians}. As with the non-linearised version, the dynamics \eqref{eq: gen filt} are simulated throughout the time interval spanning the time-series of observations $t\mapsto \by_t$.

\subsection{Practical and state-of-the-art implementation}
\label{sec: implementation details}

Here we discuss the choices that underlie any practical implementation of generalised filtering. We begin with an overview of the algorithm and then offer further details for each step.

\subsubsection{Summary of algorithm}

\textit{Offline:} the following steps are undertaken prior to sampling observations:
\begin{enumerate}
    \item \textbf{Choosing a generative model:} Our generative model constitutes our prior assumptions or domain knowledge about the data generating process. This includes a choice of flow functions $f, g$ and noise fluctuations in \eqref{eq: prior model}, \eqref{eq: likelihood model}, with, possibly, priors on their parameters, and may be hierarchical.
    \item \textbf{Choosing the order:} $M,N$ for expressing observations and latent states in generalised coordinates.
    \item \textbf{Choosing a generalised filtering method:} either the full (non-linear) version (\cref{sec: summary gf}) or the local linearised version (\cref{sec: local linear gf}).
    \item[$\Rightarrow$] Symbolic or automatic differentiation then supplies a generative model in generalised coordinates, e.g. \eqref{eq: gen mod}, and its associated potential function $V$, where these depend on whether one assumed a local linear approximation (i.e. \cref{sec: simplification local linear}).
    \item \textbf{Initialise hyperparameters and initial conditions:} e.g. the sampling rate of observations $dt>0$, and the posterior belief $q$ (initialised as a standard multivariate Gaussian).
\end{enumerate}

\textit{Online:} the following loop is repeated for each observation datapoint.
\begin{enumerate}
\setcounter{enumi}{4}
    \item \textbf{Sample observation:} $y_t$.
    \item \textbf{Data embedding in generalised coordinates:} embed $y_t$ in generalised coordinates up to order $M$, to obtain a generalised observation vector $\by_t$.
    \item \textbf{Filtering:} Numerically integrate the generalised filtering equations of motion (i.e., \eqref{eq: gen filt} or \eqref{eq: linear gen filt}) over the time interval $[t,t+dt)$, where the initial condition is the mean of the previous approximate posterior belief $q$.
    \item \textbf{Output:} Filtered distribution $q$ over the latent states (and parameters) of the generative model.
\end{enumerate}

\subsubsection{Numerical integration scheme}

Perhaps the most important implementation choice is the numerical integrator for the dynamics of the (generalised) mean $\bmu$ in \eqref{eq: gen filt} and \eqref{eq: linear gen filt}. When the model's flow $f$ and likelihood function $g$ are both linear, the Euler integrator might yield satisfactory integration with small step-sizes. However in the presence of any non-linearities this integration problem is usually stiff and requires a more accurate solver with an adaptive step-size. Here two choices are the Runge–Kutta–Fehlberg \cite{hairerSolvingOrdinaryDifferential2008} and the (recommended) more accurate albeit more expensive Ozaki solver \cite{ozakiBridgeNonlinearTime1992,fristonVariationalTreatmentDynamic2008}. %Such solvers are to be preferred in all cases since they allow for longer integration steps, since the cost of each step for each integration step is and their cost for one integration step is usually low compared to evaluating the free energy gradient, so that
% With the recommended Ozaki solver, it is commonly sufficient to execute only one or a handful of integration steps following each observation sample to obtain an accurate filtered distribution.

%In our simulations we use the Euler numerical integration scheme for the generalised filtering dynamics. The point is to numerically integrate the dynamics until the next data point is sampled. Another numeric integration method is accurate on relatively long time steps is the local linear integration method by Ozaki.

\subsubsection{Data embedding in generalised coordinates}

The generalised filtering schemes assume a time-series of observations in generalised coordinates $t \mapsto \by_t$ up to some order $M$ ($0\leq M\leq N$) as input. It is often the case that we do not possess the higher order motion of our observations, i.e. we operate in the absence of sensors for the velocity, acceleration, or higher orders of motion of our observations. Even if we have this information, we may want to obtain even higher orders of motion numerically.

Here we discuss how to obtain higher orders of motion numerically from zero-th order data. (Obtaining them from higher order observations is analogous). The most elemental way of recovering these higher orders is through finite differences, by setting
\begin{equation}
\begin{split}
     \rmy^{(1)}_t&\leftarrow \frac{  \rmy^{(0)}_t-  \rmy^{(0)}_{t-dt}}{dt},\quad \ldots \quad 
     \rmy^{(M)}_t\leftarrow \frac{\rmy^{(M-1)}_t-  \rmy^{(M-1)}_{t-dt}}{dt}
\end{split}
\end{equation}
where $dt$ is the length of the time interval  between each observation (i.e.~the sampling rate). 
%The disadvantage of this method is that it only considers the current and previous observation. The velocity estimates and those of higher orders of motion will be noisy, especially for large times steps.

The more sophisticated way of obtaining the higher orders of motion of observations is by inverting a Taylor expansion. Since generalised coordinates are based on Taylor expansions, this is the method that is most compatible with generalised filtering. The Taylor polynomial associated with the (as yet unknown) generalised coordinate vector $\by_t$ %$(\rmy^{(0)}_t,\rmy^{(1)}_t,\ldots,\rmy^{(M)}_t)$ 
is a polynomial of order $M$, which can be taken to fit $(M+1)$ points exactly. Take these points to be the known vectors
$(\rmy^{(0)}_t,\rmy^{(0)}_{t-dt},\ldots,\rmy^{(0)}_{t-Mdt})$ and $\by_t$ is uniquely determined. Explicitly, this is solving the following system of linear equations:
\begin{equation}
\label{eq: inverse taylor expansion}
    \rmy^{(0)}_{t-idt}=\sum_{n=0}^{M}\rmy^{(n)}_{t}
\frac{(-idt)^n}{n!},\quad \forall i=0,1,\ldots ,M
\end{equation}
for the unknowns $\rmy^{(1)}_t,\ldots,\rmy^{(M)}_t$. This is straightforward by inverting the matrix of Taylor coefficients.

\subsubsection{Sampling rate of observations}

In practice, the length of the time interval between each observation data point $dt$ can also be chosen (at the very least by down-sampling the time-series). This sampling rate should be chosen such that consecutive observations are most informative of the autocovariance of the observation process. If the sampling rate is too small, consecutive data will be very smooth; and conversely if the sampling rate is too large, consecutive data will be very rough. One usually sets the sampling rate to the intrinsic time-scale of the process, which can be taken to be the first zero-th crossing of the autocovariance of the observed time series.

\subsubsection{Choosing the order}
    \label{sec: order selection}

To select the order for generalised coordinates, one can leverage Taken's third embedding theorem \cite[Th.~3]{takensDetectingStrangeAttractors1981}. This theorem states that a smooth attractor observed through a smooth univariate map can be embedded in a real vector space through its successive derivatives. This embedding is guaranteed to exist when the dimension of the embedding space is more than  twice the (box-counting) dimension of the attractor~\cite{sauer1991embedology}. In other words, using generalised coordinates of order $2d$ is sufficient to capture the dynamics of (smooth) systems of dimension at most $d$. For instance, the Lorenz attractor has a box-counting dimension of approximately $2.06$~\cite{viswanath2004fractal}, and thus can be embedded using generalised coordinates of order $5$. In practice, a lower embedding order might be sufficient to achieve a satisfactory reconstruction accuracy. %, as multivariate observations may provide additional information on the dynamics and noisy observation make larger embedding order unnecessary. 
The optimal order can be determined by plotting the free-energy (over time) against the embedding order, and selecting the order that optimises the free-energy (through time). This yields the order with that best trades off accuracy of the filtered trajectory with model complexity. Indeed, this approach is recommended for selecting hyperparameters as well as the model itself, as discussed next. 
% \LD{@Johan: maybe add somethihg about order of motion being usually 4 to 6 and a justification with Taken's theorem.}

\subsubsection{Choosing a generative model and hyperparameters}
%free energy

Two important questions are what constitutes an appropriate choice of generative model for a given natural phenomenon (e.g. choice of noise in \eqref{eq: prior model}) and what hyperparameters to use when implementing a generalised filtering method. This can be framed as a problem of (Bayesian) model selection, whereby the functional form of the generative model and the routine's hyperparameters can be optimised to minimise the (time integral of) free energy. This is because the free energy \eqref{eq: VFE_def} equals complexity minus accuracy
\begin{equation}
\label{eq: comp acc fe}
\begin{split}
\F(q; \by_t)
&= \underbrace{\dkl[q(\bx_t) \mid p(\bx_t)]}_{\text{Complexity}}-\underbrace{\E_q[-\log p(\by_t \mid \bx_t)]}_{\text{Accuracy}},
\end{split}
\end{equation}
whereby maximising accuracy entails finding explanations $q$ for the data that are maximally accurate (i.e.~maximum likelihood), while minimising complexity ensures that these explanations are as simple as possible above and beyond the prior (cf. Occam's razor). In other words, the free energy \eqref{eq: VFE_def} is an objective function that jointly scores the quality of our generative model and inference scheme as an explanation for our data. 

In practice, identifying an appropriate generative model for accurately explaining an empirical phenomenon is usually the most challenging and critical prerequisite to successfully applying model-based methods such as (extended) Kalman filtering, model predictive control and generalised filtering. This difficulty reflects a deeper principle: the broader scientific endeavor itself can be viewed as the search for generative models that effectively describe and predict our empirical observations.

%Maximising accuracy w.r.t $q$ corresponds to finding the generalised latent state $\bx_t$ that maximizes the likelihood of the data $\by_t $ under our generative model. Minimizing complexity on the other hand favours explanations for the data that are maximally consistent with the prior.
% Note that doing this with successful results with the Laplace approximated free energy assumes that the Laplace approximation is a good approximation given the generative model at hand.

\subsubsection{Hierarchical models and parameter inference}
\label{sec: extensions GF}

Here we have presented generalised filtering in its simplest form, with a focus on mathematical clarity, but there are many ways in which it has been extended. Perhaps most important are parameter inference of the functions $f, g$ and fluctuations $w,z$ (i.e.~their generalised covariance $\bSigma^w,\bSigma^z$), and extensions to hierarchical generative models which include processes that evolve at multiple timescales of fast and slow (i.e.~multi-scale processes \cite{pavliotisMultiscaleMethodsAveraging2010}).\footnote{Without these parameter inference extensions generalised filtering is equivalent to dynamic expectation maximisation \cite{fristonVariationalTreatmentDynamic2008}.}

\subsubsection{State-of-the-art implementation}
\label{sec: sota implementation}

The state-of-the-art implementation of generalised filtering can be found in the SPM academic software under the script \emph{DEM.m} freely available here \url{https://www.fil.ion.ucl.ac.uk/spm/software/}. This has been optimised for nearly two decades for numerical stability by the neuroimaging community. (We will recapitulate implementations, which include versions in Python, in the software note, see Section~\ref{sec: conclusion}.
In this implementation, the method for numerical integration of generalised filtering is the Ozaki solver \cite{ozakiBridgeNonlinearTime1992,fristonVariationalTreatmentDynamic2008}, the method for embedding data into generalised coordinates is \eqref{eq: inverse taylor expansion}, and the order of motion for the data $M$ is always chosen to be equal to the order of motion for the latent states $N$. This implementation always uses the local linear approximation in generalised filtering, but has not been bench-marked against the non-linearised version---this could be the subject of future work. Nevertheless, there is a reason why the local linear approximation makes sense in the context of generalised filtering with the Laplace approximation. Discarding all derivatives of the flow, of order higher than one, has the advantage of making the functional form of the energy \eqref{eq: energy} nearly quadratic, which licenses the Laplace approximation. Conversely, dispensing with the local linear approximation can make the Laplace approximation a poor approximation (especially when $f,g$ are highly non-linear) hindering the method's performance.

\subsection{Simulations and results}
% https://github.com/johmedr/dempy/tree/main/figures

We show the results of a simple simulation of generalised filtering on a partially observed stochastic Lorenz system. Despite its small-scale, this is a relatively challenging filtering problem because the stochastic Lorenz system exhibits stochastic chaos \cite{law_data_2016,reich_probabilistic_2015,fristonStochasticChaosMarkov2021}.

\textbf{The generative process:} The latent state of the data generating process $x_t=(x_{0,t},x_{1,t},x_{2,t})$ is three dimensional and evolves according to a stochastic Lorenz system; an SDE with flow
\begin{equation*}
    f(x_t)=\begin{pmatrix}
        18 (x_{1,t} - x_{0,t})\\
    46.92 x_{0,t} - 2 x_{2,t} x_{0,t} - x_{1,t}\\
    2 x_{0,t} x_{1,t} - 4 x_{2,t}
    \end{pmatrix}
\end{equation*}
and fluctuations a stationary Gaussian process with Gaussian autocovariance $w_t$ (i.e.~white noise convolved with a Gaussian). The data we observe $y_t$ is one-dimensional: the sum of all three components of $x_t$, i.e.
\begin{equation*}
    g(x_{t})=x_{0,t}+x_{1,t}+x_{2,t},
\end{equation*}
plus observation noise $z_t$, a one-dimensional stationary Gaussian process with Gaussian autocovariance. We perform experiments for different levels of state and observation volatilities.

\textbf{The generative model:} For simplicity we take the generative model to be equal to the generative process.

\textbf{Simulation:} We show the result of generalised filtering for different magnitudes of  observation noise (i.e.~noise volatility); in a regime with low observation noise (Figure~\ref{fig: DEM}, \textit{left}) and another with high observation noise (Figure~\ref{fig: DEM}, \textit{right}) using the state-of-the-art implementation detailed in Section~\ref{sec: sota implementation}, using $N=M=5$ orders of motion (twice the box-counting dimension of the Lorenz attractor, see~\cref{sec: order selection}). Please see Figure~\ref{fig: DEM} for an illustration of real and inferred states as well as their difference measured in phase-space Euclidean distance.

\textbf{Results:} The generalised filter is able to infer the latent state and reconstruct the attractor: in the regime of low observation noise the reconstruction is very accurate, while in the regime of high observation noise the reconstruction is accurate overall but misses many details, however this is to be expected given the low signal to noise ratio in this case.

\begin{figure}[h]
\centering\includegraphics[width=0.47\textwidth]{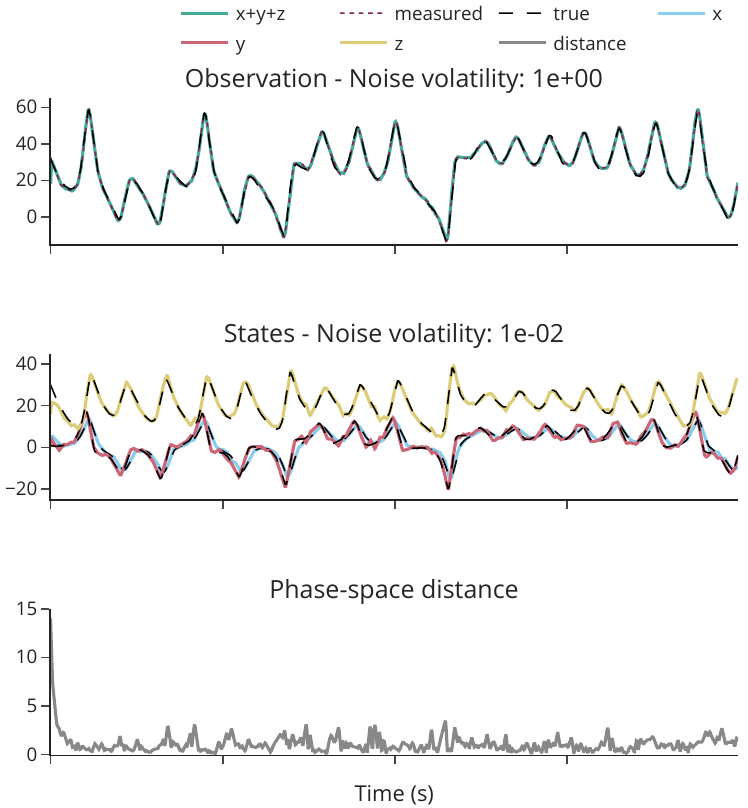}
\centering\includegraphics[width=0.47\textwidth]{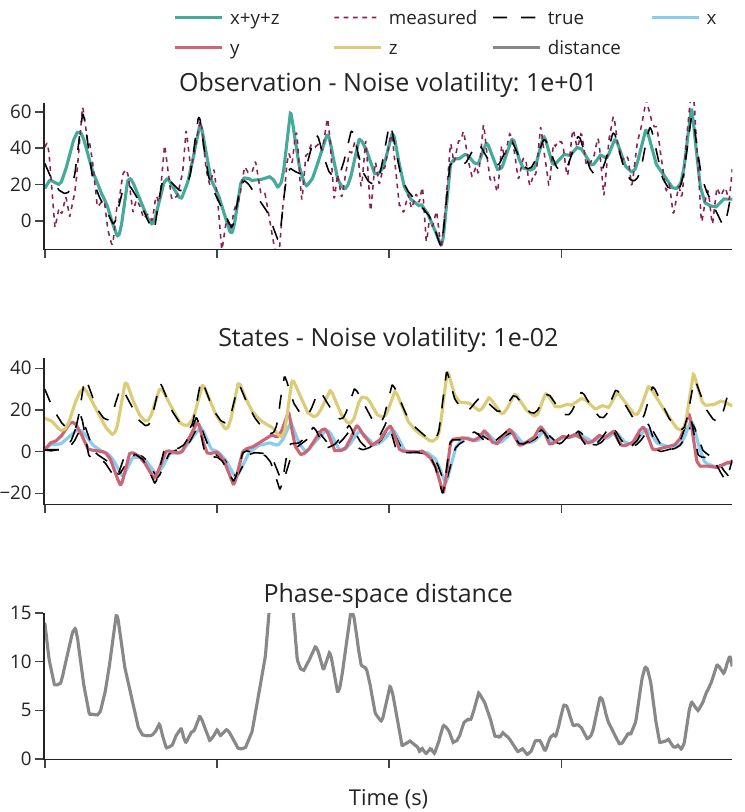}
\caption{\textbf{Generalised filtering for the Lorenz system.} This Figure illustrates a simulation of generalised filtering on a partially observed stochastic Lorenz system, in a regime of low and high observation noise (left and right panel respectively). Generalised filtering uses no particles and remains stable under strongly coloured observation noise. In the top panels, we see the one-dimensional time-series of noiseless observations in dashed black, corresponding to the sum of the three latent states. The data that were fed to the generalised filter comprised the \textit{measured} time-series in dashed purple; i.e. the dashed black time-series plus observation noise. In the middle panels, there are three dashed black time-series corresponding to the latent states, which evolve according to  a stochastic Lorenz system. In colour (yellow, red, blue) are the filtered time-series; that is, the most likely latent states at each time according to the approximate posterior belief $q$. In the top panel, the cyan time-series are  filtered observation time-series, obtained as the sum of the coloured time-series in the middle panels. The bottom panel scores the Euclidean distance between filtered time-series in latent space and true time-series.}
\label{fig: DEM}
\end{figure}

\textbf{Discussion:} This simulation uses a generative model that is equal to the generative process
to discard confounders such as model mis-specification when illustrating this method's ability. We refer the reader to \cite{fristonGeneralisedFiltering2010,bos2022free,meera2023adaptive,anil2021dynamic,balajiBayesianStateEstimation2011} for more realistic, higher-dimensional inferences in real world scenarios.

\chapter{Concluding remarks}
\label{chap: 4}

%\section{Not included in paper}

\section{Future directions}
\label{sec: future directions}

There are many directions that could attend future work:

\textbf{Contextualisation within the literature:} From a global perspective, we need a more comprehensive understanding of the convergence and differences between the theory of generalised coordinates and established approaches to the analysis and numerical treatment of stochastic differential equations with general classes of noise; for instance the relationship with standard stochastic analysis \cite{heSemimartingaleTheoryStochastic1992}, rough path theory \cite{frizCourseRoughPaths2020,frizMultidimensionalStochasticProcesses2010}, and the current theory of Markovian realisation \cite{lindquistLinearStochasticSystems2015,mitterTheoryNonlinearStochastic1981,tayorNonlinearStochasticRealization1989}.

\textbf{Generalising theoretical results:} From a technical perspective, many of the technical assumptions used throughout this work could be loosened and results generalised. For instance, one can presumably extend this theory to include equations with \textit{multiplicative noise}, and use a similar local linearisation procedure for the noise to control the number of terms in the generalised coordinates expansion for analysis and numerics.\footnote{The Wong-Zakai approximation theorems of \cref{sec: wong zakai} generalise to multiplicative noise SDEs in a slightly unexpected way, as was explained in \cref{footnote: wong}.} %\footnote{When approximating non-additive noise SDEs in the context of the Wong-Zakai theorems, one should be mindful of the fact that the limiting SDE contains an additional term in the drift, called the L\'{e}vy area correction term~\cite{PhysRevE.88.062150},\cite[Sec. 11.7.7]{pavliotisMultiscaleMethodsAveraging2010},\cite[Sec. 5.1]{pavliotisStochasticProcessesApplications2014},\cite[Sec. 3.4]{frizCourseRoughPaths2020}, \cite[Ch. 7]{ikeda-watanabe-89}.}
Another important extension would be to \textit{multi-scale SDEs} \cite{pavliotisMultiscaleMethodsAveraging2010}.\footnote{The Wong-Zakai theorems also generalise to approximating processes with SDEs that have multiple characteristic time-scales, albeit in another surprising way: the limiting SDE depends on the relative magnitude of the timescales of the approximating processes, and a whole one parameter family of limiting SDEs is possible~\cite{KupPavlSt04}. This finding has also been confirmed experimentally in noisy electric circuits~\cite{PMcWV_2013}.}
Additionally, it should be straightforward to generalise the theory to accommodate \textit{non-stationary} Gaussian process noise.
\begin{example}%https://www.stat.uchicago.edu/~lalley/Courses/385/Old/GaussianProcesses.pdf
For example, it would be nice to be able to work with noise of the form $w_t=\sum_{i=1}^m \xi_i \cos \left(\lambda_i t\right), t \in \R,$
where $\xi_1, \ldots, \xi_m$ are i.i.d. standard normal random variables and $\lambda_1, \ldots, \lambda_m\geq 0$ are a finite set of frequencies. This process is a mean-zero, non-stationary Gaussian process with analytic sample paths. Random Fourier series of this form arise naturally from the Karhunen-Lo\'{e}ve expansion~\cite[Sec. 1.4]{pavliotisStochasticProcessesApplications2014}. %(they are simply random linear combinations of cosine waves). %Note that for any finite set $F$ of cardinality larger than $m$ the random vector $w_F$ has a degenerate Gaussian distribution. \LD{Comput 20.9.23 on tablet 'analysis' derives kernel. Derivative of kernel doesn't vanish. This calls into question whether we can use the tech here for non stationary GPs.}
\end{example}

\cite[Theorem 4 and Corollary 1, p224-225]{gikhmanTheoryStochasticProcesses2004} in combination with Lemma~\ref{lemma: autocov of generalised fluctuations} could be used to obtain the autocovariance of the generalised non-stationary process, and hence extend the whole construction. Lastly, %although we focused exclusively on the case of a deterministic in initial condition,
the theory herein extends to arbitrary stochastic initial conditions, and the numerical integration methods to stochastic initial conditions that can be sampled from.

\textbf{Extending analysis:} There is potential to extend the analysis of the linear SDE in generalised coordinates to that of the local linearised SDE. In addition, \textit{it would be important to derive accuracy estimates for the local linearised version of an SDE compared to its non-linearised version} (in generalised coordinates). The latter may potentially be obtained by adapting error analyses on the local linearisation of SDEs from the numerical integration \cite{kloedenNumericalSolutionStochastic1992,de2007higher,jimenez1999simulation,biscay1996local,jimenez2009rate,jimenez2012convergence,de2010high} and Extended Kalman Filtering \cite{sarkka2023bayesian,jazwinski2007stochastic} literatures to our setting. Taken together these results would yield an approximate analytic solution to any SDE with analytic flow and sample paths, and, by the Wong-Zakai theorems, to fairly arbitrary SDEs.

\textbf{Accurate numerical simulation globally in time:} Numerical integration with generalised coordinates is accurate on short time intervals while established methods for numerical integration of SDEs driven by coloured noise (i.e.~approximation with a diffusion process in an extended state-space followed by standard numerical schemes) are accurate only on longer timescales. This begs the question as to whether these approaches could be combined in a multi-scale numerical scheme wherein integration steps of the latter methods are interpolated with the former, yielding accurate integration on short and long timescales.

\textbf{Refining generalised filtering implementations:} The theoretical foundation and derivation of generalised filtering presented in this paper shows the underlying assumptions and possible choices of implementation, such as committing to local linearisation versus non-linearisation. We hope that this treatment will help practitioners to refine existing---and devise improved---implementations. Practically, an intriguing project would be a comprehensive bench-marking of local linearised versus non-linearised generalised filtering. Another would be extending non-linearised generalised filtering to inferring parameters of the generative model---and to hierarchical generative models.

\textbf{Stochastic control via generalised coordinates:} Perhaps the largest omission from this paper are the established methods for stochastic control via generalised coordinates, developed in the area of continuous active inference \cite{fristonActionBehaviorFreeenergy2010,parrActiveInferenceFree2022}. Briefly, just as generalised filtering operates by minimising the quantity known as free energy, active inference includes an additional control variable in the generative model---the actions---which are optimised to minimise free energy as well. In other words, \textit{free energy becomes a unified objective for action and perception}.
Continuous active inference generalises many existing algorithms to control, such as PID control \cite{baltieriPIDControlProcess2019}, and may offer more robust and capable alternatives \cite{lanillosActiveInferenceRobotics2021,dacostaHowActiveInference2022a}. Continuous active inference can also be combined with discrete active inference \cite{dacostaActiveInferenceDiscrete2020,parrActiveInferenceFree2022} in continuous-discrete hierarchical generative models that can model human-level control \cite{priorelliModelingMotorControl2024,parrComputationalNeurologyMovement2021a}. We expect that a detailed treatment of continuous active inference building upon---and at a similar level of detail---as the treatment of generalised filtering herein (with, possibly, its integration with discrete active inference) could fuel theoretical and algorithmic advances in state-of-the-art control.%It would be very worthwhile to present a detailed treatment of continuous active inference building upon - and at a similar level of detail - as generalised filtering herein.

\section{Conclusion}
\label{sec: conclusion}

In this paper, we presented a theory for the analysis, numerical simulation, and filtering of stochastic differential equations with many times differentiable flows and fluctuations. This translates the problem of analysing trajectories of SDEs into the simpler problem of finding the higher order motion of the solution (velocity, acceleration, jerk etc) and recovering the trajectories as Taylor expansions or Taylor series. By the Wong-Zakai approximation theorems this can be used to analyse a wide range of stochastic differential equation driven a wide range of noise signals, which may or may not be Markovian.

The conclusion is that SDEs with analytic flows and fluctuations are the SDE analogs to analytic functions. Like analytic functions, they can be usefully expressed and studied by their Taylor series; analogously to analytic functions, they can uniformly approximate the solutions to rougher, continuous SDEs on finite time-intervals, where the Wong-Zakai theorems take the place of the Weierstrass approximation theorem. The computational amenability of Taylor expansions enables computationally straightforward methods for numerical simulation, filtering and control.  In this paper we formalised and developed these ideas, and argue that they have far-reaching implications throughout stochastic differential equations.

\subsubsection*{Software note}

All simulations are reproducible with freely available code: for Sections~\ref{sec: num int} and~\ref{sec: most likely path via Lagrangian} the code is available at \url{https://github.com/lancelotdacosta/Generalised_coordinates}; for Section~\ref{sec: GF}, the simulations can be reproduced with \url{https://colab.research.google.com/drive/1brsgc2gob-O_SjlgxyVg9R46VezIbP34?usp=sharing}. At the time of writing, the state-of-the-art implementation of generalised filtering can be found in the SPM academic software under the script \emph{DEM.m} freely available at \url{https://www.fil.ion.ucl.ac.uk/spm/software/}. We also draw attention to an implementation of generalised filtering in python available at \url{https://github.com/johmedr/dempy}.

\subsubsection*{Acknowledgements}

The authors thank Mehran H. Bazargani and Philipp Hennig for interesting discussions that contributed to improving this paper.

\subsubsection*{Funding statement}

LD is supported by the Fonds National de la Recherche, Luxembourg (Project code: 13568875). This publication is based on work partially supported by the EPSRC Centre for Doctoral Training in Mathematics of Random Systems: Analysis, Modelling and Simulation (EP/S023925/1). TP is supported by an NIHR Academic Clinical Fellowship (ref: ACF-2023-13-013). GP is partially supported by an ERC-EPSRC Frontier Research Guarantee through Grant No. EP/X038645, ERC
Advanced Grant No. 247031 and a Leverhulme Trust Senior Research Fellowships, SRF$\backslash$R1$\backslash$241055.

\bibliographystyle{abbrvnat} 
\bibliography{bib}

\appendix
\chapter{}
\section{Some frequently used notations}
\label{app: notation}

% $\R^d$ state space

$\T$ time domain, open interval $\subset \R$, where usually we assume that $0\in \T$ (without loss of generality).

%$\phi_t = e^{-\beta  t^2}$ Gaussian kernel -> change to normalised Gaussian kernel with std?

\textit{a.s.} almost surely.

% $\tilde\cdot$ to denote variables in generalised coordinates

% order: number of derivatives of $x$; $f, w$ are required to be order-1 differentiable. (will need to update text accordingly).

$ \G^{d,n}$ the space of generalised coordinates \eqref{eq: space of generalised coordinates}.

$\mathbf D, \mathbf D'$ matrices in generalised coordinates \eqref{eq: generalised variables for gen cauchy problem}.

$\mathbf f$ generalised flow, see \cref{def: gen flow} (a vector field in generalised coordinates).

$\doteq$ equality up to an additive constant.

$\stackrel{\ell}{=}$ equality in law, a.k.a. equality in distribution.

$\nabla, \nabla^2$ Jacobian and Hessian, e.g. $\nabla f: \R^d \to \R^{d \times d}$ is the Jacobian of $f$.

$ L^p$ Lebesgue space.

$\Mat_{d,n}$ space of real-valued matrices with $d$ lines and $n$ columns.

\section{A technical Lemma}
\label{app: technical}

This Lemma is useful for the analysis of the linear equation; specifically, for deriving Proposition~\ref {lemma: cov convergence}:

\begin{lemma}\label{lemma: limsup equality}
Suppose $(a_n)_{n \geq 0}$, $(b_n)_{n\geq 0}$ are sequences of non-negative real numbers such that $b_n = \max_{0\leq i\leq n} a_i$. Then
\begin{equation}\label{eq:cvgence_equiv}
\limsup_{\substack{n+m\to\infty \\ n,m\geq 1}}\left(\frac{a_{n+m}}{n!m!}\right)^{\frac{1}{n+m}}=\limsup_{\substack{n+m\to\infty \\ n,m\geq 1}}\left(\frac{b_{n+m}}{n!m!}\right)^{\frac{1}{n+m}}.
\end{equation}
\end{lemma}

\begin{proof}
Without loss of generality we will show \eqref{eq:cvgence_equiv} along non-decreasing sequences of indices: let $(n_k,m_k)_{k\geq 0}$ be a sequence of tuples of natural numbers such that $n_k+m_k\to\infty$ and $(n_k)_{k\geq 0}$, $(m_k)_{k\geq 0}$ are non-decreasing.
It suffices to show \eqref{eq:cvgence_equiv} along this sequence. Clearly, $0\leq \text{LHS} \leq \text{RHS}$ in \eqref{eq:cvgence_equiv}.

First suppose $(a_n)_{n\geq 0}$ is bounded by $c>0$. Then so is $(b_n)_{n\geq 0}$. So by Stirling's formula for the asymptotic of the factorial
$$
\begin{aligned}
\left(\frac{b_{n_k+m_k}}{n_k!m_k!}\right)^{\frac{1}{n_k+m_k}}&\leq \left(\frac{c}{n_k!m_k!}\right)^{\frac{1}{n_k+m_k}} 
\sim \left(\frac{c2\pi e^{n_k+m_k}}{n_k^{n_k+1/2}m_k^{m_k+1/2}}\right)^{\frac{1}{n_k+m_k}} \xrightarrow{k \to \infty}0.
\end{aligned}
$$
Thus in this case $\text{LHS} = \text{RHS}=0$ in \eqref{eq:cvgence_equiv}.

Suppose now that $(a_n)_{n\geq 0}$ is unbounded. We can then extract the infinite subsequences $(a_{n_{k_j}})_{j\geq 0}$, $(b_{n_{k_j}})_{j\geq 0}$ such that $a_{n_{k_j}} = b_{n_{k_j}}$ for all $j$. Since
\begin{equation}
\begin{aligned}
\limsup_{j\to\infty}\left(\frac{a_{n_{k_j}+m_{k_j}}}{n_{k_j}!m_{k_j}!}\right)^{\frac{1}{n_{k_j}+m_{k_j}}}&\leq \limsup_{k\to\infty}\left(\frac{a_{n_k+m_k}}{n_k!m_k!}\right)^{\frac{1}{n_k+m_k}} \\
\rotatebox{90}{$=\;\;$}\hspace{2.3cm}& \hspace{1.9cm}\rotatebox{90}{$\geq$} \\
\limsup_{j\to\infty}\left(\frac{b_{n_{k_j}+m_{k_j}}}{n_{k_j}!m_{k_j}!}\right)^{\frac{1}{n_{k_j}+m_{k_j}}}&\leq \limsup_{k\to\infty}\left(\frac{b_{n_k+m_k}}{n_k!m_k!}\right)^{\frac{1}{n_k+m_k}}
\end{aligned}
\end{equation}

it suffices to show the bottom inequality is an equality to show equality throughout. For this, it suffices to show that for $j$ large enough, if $b_{n_{k_j}+m_{k_j}}= b_{n_k+m_k}$ with $k\geq k_j$, then
\begin{equation}\label{eq: inequality to show}
\left(\frac{b_{n_{k_j}+m_{k_j}}}{n_{k_j}!m_{k_j}!}\right)^{\frac{1}{n_{k_j}+m_{k_j}}} \geq \left(\frac{b_{n_{k_j}+m_{k_j}}}{n_k!m_k!}\right)^{\frac{1}{n_k+m_k}} =\left(\frac{b_{n_k+m_k}}{n_k!m_k!}\right)^{\frac{1}{n_k+m_k}}.
\end{equation}
For $c>0$, define the function $f_c:\R_{>0}^2\to \R_{>0}$ by
$$f_c(x,y) \triangleq  \left(\frac{c}{\Gamma(x+1)\Gamma(y+1)}\right)^{\frac{1}{x+y}},$$
where $\Gamma$ is the gamma function. For \eqref{eq: inequality to show} to hold for all $j$ large enough it suffices by the mean value theorem that there are $C,M>0$ such that for all $c\geq C$ and $x+y\geq M$ we have $\partial_xf_c(x,y), \partial_yf_c(x,y) \leq 0$. 

We proceed to show this. We have
\begin{equation}\label{eq: partial x of f}
\partial_x f_c(x,y) = \bigg(-\frac{\psi(x+1)}{x+y}+\frac{-\log(c)+\log(\Gamma(x+1))+\log(\Gamma(y+1))}{(x+y)^2}\bigg)\left(\frac{c}{\Gamma(x+1)\Gamma(y+1)}\right)^{\frac{1}{x+y}}
\end{equation}
and
\begin{equation}\label{eq: partial y of f}
\partial_y f_c(x,y) = \bigg(-\frac{\psi(y+1)}{x+y}+\frac{-\log(c)+\log(\Gamma(x+1))+\log(\Gamma(y+1))}{(x+y)^2}\bigg)\left(\frac{c}{\Gamma(x+1)\Gamma(y+1)}\right)^{\frac{1}{x+y}}
\end{equation}
where $\psi \triangleq  \Gamma'/\Gamma$ is the digamma function. Note that $\psi(x+1) \geq \log(x)$ for all $x> 0$, and $\log(\Gamma(x+1)) = x\log(x) - x+ o(x)$ as $x\to\infty$. From \eqref{eq: partial x of f} we get for $x>0$
\begin{equation*}
\begin{aligned}
\partial_xf_c(x,y) &\leq (-(x+y)\log(x)-\log(c)+x\log(x)-x+ o(x) +y\log(y) -y +o(y))g_c(x,y) \\
&= (-\log(c)-x-y+o(x+y))g_c(x,y)
\end{aligned}
\end{equation*}
for some positive function $g_c:\R^2_{>0} \to \R_{>0}$, where the error term $o(x+y)$ does not depend on $c$. Thus, fixing a $C>0$, we see that there is $M>0$ such that for all $c\geq C$ and $x+y\geq M$, $\partial_xf_c(x,y) \leq 0$. From \eqref{eq: partial y of f} we the case of $\partial_yf_c$ is analoguous, which concludes the proof.
\end{proof}

\end{document}